\documentclass[a4paper,reqno,oneside,10pt]{amsart}

\usepackage[utf8]{inputenc}
\usepackage[english]{babel}
\usepackage{csquotes}		

\usepackage{geometry}

\usepackage{graphicx}
\usepackage[dvipsnames]{xcolor}
\usepackage{comment}

\usepackage[shortlabels]{enumitem} 

\usepackage{amsfonts}
\usepackage{amssymb}
\usepackage{amsthm}
\usepackage{amsmath}
\usepackage{dsfont}
    \let\mathbbm\mathds

\usepackage{mathtools}
	\mathtoolsset{centercolon} 
\usepackage{mathrsfs}
\usepackage{xfrac}

\usepackage[scr=boondox]{mathalfa}
\usepackage{tikz}
    \usetikzlibrary{hobby}
    \usetikzlibrary{arrows}

\usepackage{pgfplots}

\renewcommand{\paragraph}[1]{%
  \bigskip\noindent\textit{\underline{#1.}}
}

\makeatletter
\renewcommand\subsubsection{\@startsection{subsubsection}{3}%
  \z@{.5\linespacing\@plus.7\linespacing}{-.5em}%
  {\normalfont\itshape}}
\makeatother

\renewenvironment{proof}[1][\proofname]{%
  \par\medskip
  \noindent\textbf{{#1.}}\;\;%
}{%
  \hfill\qedsymbol\par\bigskip
}

\usepackage{aliascnt}

\newtheorem{proposition}{Proposition}[section]

\newaliascnt{theorem}{proposition}
\newtheorem{theorem}[theorem]{Theorem}
\aliascntresetthe{theorem}

\newaliascnt{lemma}{proposition}
\newtheorem{lemma}[lemma]{Lemma}
\aliascntresetthe{lemma}

\newaliascnt{corollary}{proposition}
\newtheorem{corollary}[corollary]{Corollary}
\aliascntresetthe{corollary}

\newaliascnt{question}{proposition}

\aliascntresetthe{question}

\theoremstyle{definition}

\newaliascnt{definition}{proposition}
\newtheorem{definition}[definition]{Definition}
\aliascntresetthe{definition}

\newaliascnt{assumptions}{proposition}

\aliascntresetthe{assumptions}

\newaliascnt{remark}{proposition}
\newtheorem{remark}[remark]{Remark}
\aliascntresetthe{remark}

\newaliascnt{example}{proposition}
\newtheorem{example}[example]{$\diamond$ Example}
\aliascntresetthe{example}

\usepackage[colorlinks=true,allcolors=MidnightBlue,pdfencoding=auto, psdextra]{hyperref}
\usepackage{cleveref}

\crefname{proposition}{proposition}{propositions}
\Crefname{proposition}{Proposition}{Propositions}

\crefname{theorem}{theorem}{theorems}
\Crefname{theorem}{Theorem}{Theorems}

\crefname{lemma}{lemma}{lemmas}
\Crefname{lemma}{Lemma}{Lemmas}

\crefname{corollary}{corollary}{corollaries}
\Crefname{corollary}{Corollary}{Corollaries}

\crefname{question}{question}{questions}
\Crefname{question}{Question}{Questions}

\crefname{definition}{definition}{definitions}
\Crefname{definition}{Definition}{Definitions}

\crefname{assumptions}{assumptions}{assumptions}
\Crefname{assumptions}{Assumptions}{Assumptions}

\crefname{remark}{remark}{remarks}
\Crefname{remark}{Remark}{Remarks}

\crefname{example}{example}{examples}
\Crefname{example}{Example}{Examples}

\numberwithin{equation}{section}

\title[Enthalpy dissipation and free boundary stability of the Stefan problem]{Energy dissipation and free boundary stability of the Stefan problem}
\date{\today}

\newcommand{\eps}{\varepsilon}
\newcommand{\de}{\partial}
\newcommand{\dif}{\,\text{d}}
\renewcommand{\phi}{\varphi}
\newcommand{\Interior}{\text{Int}}

\newcommand{\Hc}{\mathscr{H}}

\newcommand{\R}{\mathbb{R}}
\newcommand{\NN}{\mathbb{N}}
\newcommand{\sphere}{\mathbb{S}}
\newcommand{\spt}{\text{spt}\,}

\newcommand{\ind}{\mathbbm{1}}

\newcommand{\distop}{\mathrm{dist}}

\begin{document}

\author[F.~Paiano]{Filippo Paiano}

\author[B.~Velichkov]{Bozhidar Velichkov}

\address {Filippo Paiano \newline \indent
Dipartimento di Matematica, Universit\`a di Pisa \newline \indent
Largo Bruno Pontecorvo, 5, 56127 Pisa, Italy}
\email{filippo.paiano@phd.unipi.it}

\address {Bozhidar Velichkov \newline \indent
Dipartimento di Matematica, Universit\`a di Pisa \newline \indent
Largo Bruno Pontecorvo, 5, 56127 Pisa, Italy}
\email{bozhidar.velichkov@unipi.it}


\begin{abstract}
In this paper we prove the stability and the continuity of the free boundaries of two-phase and one-phase weak solutions to the Stefan problem in any dimension. Our proof has two main ingredients. The first one is an {\it energy conservation principle for weak solutions} that we derive for general domains in both the one-phase and the two-phase settings; this energy conservation principle provides the modulus of continuity for the displacement of the free boundary and is a flexible tool for analyzing long-time behavior of solutions. The second key ingredient for the control of the displacement of the free boundary is a new lower density bound for the melting region at free boundary points that move with maximal velocity. 
We also discuss some applications, showing a series of qualitative and quantitative results on the asymptotic behavior of weak solutions.
\end{abstract}
\maketitle

\tableofcontents

\section{Introduction}

The Stefan problem was introduced by Stefan \cite{Stefan1889:AcidBaseDiffusion,Stefan1889:HeatConduction,Stefan1890:Diffusion,Stefan1891:IceFormation} and by Lamé and Clapeyron \cite{LameClapeyron1831} as a model for heat conduction in phase transition phenomena. In its original formulation, the domain splits into two phases, solid and liquid, characterized by strictly negative and strictly positive temperatures $u^+$ and $u^-$, separated by a moving interface evolving with normal velocity
$$V_\nu =\frac{1}{2}\left( |\nabla u^+|- |\nabla u^-| \right),$$
where $\nu\in\sphere^{d-1}$ is the normal pointing towards the negative phase.\medskip 

In this paper, we work with the  {\it weak} (or {\it enthalpy}) {\it formulation}  of the Stefan problem (see \Cref{def:enthalpy_solution} and \Cref{sec:enthalpy-formulation}) in which the solution is a couple $(u,\mu)$, 
$$u:D\times[0,+\infty)\to\R\qquad\text{and}\qquad \mu:D\times[0,+\infty)\to[-1,1],$$
$u$ being the {\it temperature}  and $\mu$ the {\it mushy coefficient}, satisfying the integral identity 
    \begin{equation}\label{e:weak_solution-1}\tag{H-ST}
        \int_{D(t)}(u+\mu)\,\eta \dif x \bigg|_{t=t_1}^{t_2} 
        = \int_{t_1}^{t_2}\int_D (u+\mu) \de_t \eta 
            -\nabla u \cdot \nabla \eta \dif x \dif t,
    \end{equation}  
for all times $0\le t_1\le t_2\le +\infty$ and all test functions $\eta \in C^\infty_c(D\times\R)$, and the compatibility condition 
\begin{equation}\label{e:compatibility_enthalpy-phases-1}
        \mu(x,t) = 
        \begin{cases}
        \begin{array}{rl}
             1 & \text{if}\quad u(x,t)>0,\\
            -1 & \text{if} \quad u(x,t)<0,
        \end{array}
        \end{cases} \qquad\text{for almost-every}\quad x\in D.\tag{C-ST}
    \end{equation}
%
The weak formulation of the Stefan problem and the corresponding existence theory has been developed since the seminal works  of Kamin and Oleinik \cite{Kamenomostskaya,Oleinik60}, while an existence theory for weak solutions satisfying the basic regularity conditions from \Cref{def:enthalpy_solution} was developed only recently in \cite{Paiano_Velichkov2026:ellipticregularization}. The local behavior of the solutions and the evolution of the free interfaces is still far from being completely understood. Caffarelli and Evans \cite{CaffarelliEvans1983:ContinuityTwoPhaseStefan}, DiBenedetto \cite{DiBenedetto1982:Stefan1,DiBenedetto1980:Stefan2}, and Ziemer \cite{Ziemer1982:Stefan} have proved that the temperature $u$ is continuous in space-time, but this, by itself, does not allow to deduce information about the regularity of the mushy coefficient $\mu$ or about behavior of the water and the water regions 
\[
    \Omega_u^+(t) := \big\{ u(\cdot,t)>0\big\}\quad\text{and} \quad \Omega_u^-(t) := \big\{u(\cdot,t)<0\big\},
\]
which are related to $\mu$ through  the compatibility condition \eqref{e:compatibility_enthalpy-phases}.
The absence of such results is mainly due to the fact that the non-degeneracy of the temperature $u$ is not true in general; in particular, there are examples of solutions $(u,\mu)$ in which the temperature $u$ decays arbitrarily fast at points of the free boundary (see for instance \cite{Athanasopoulos_Caffarelli_Salsa_1996b,CaffarelliSalsa:GeomApproachToFreeBoundary}).\medskip

In this paper, we overcome this issue and we prove the $L^1$ continuity of the mushy coefficient $\mu$ (\Cref{thm:total-monotonicity-formula}), the stability of the two-phase free boundaries with respect to the initial datum (\Cref{thm:finite-melting}) and the stability in time of the one-phase free boundaries (\Cref{thm:finite-melting-one-phase}). For the  main ideas and the strategy of the proofs we refer to \Cref{sub:intro:energyconservation} (for \Cref{thm:total-monotonicity-formula}) and \Cref{sub:intro:stability} (for \Cref{thm:finite-melting} and \Cref{thm:finite-melting-one-phase}).\medskip

A key result for the proof of \Cref{thm:finite-melting} and \Cref{thm:finite-melting-one-phase} is the monotonicity of the \emph{total enthalpy} of $(u,\mu)$, which we prove in the following theorem.

\begin{theorem}[Monotonicity of the {\it total enthalpy}]\label{thm:total-monotonicity-formula}
    Let $D\subset\R^d$ be an open set.
    Let $g\in H^1_0(D)\cap L^1(D)$ and $\|h\|_{L^\infty(D)}\le1$.
    Let $(u,\mu)$ be the enthalpy solution in $D_\infty = D\times[0,+\infty)$ (in the sense of \Cref{def:enthalpy_solution}) with initial temperature $g$, initial mushy coefficient $h$, and with homogeneous Dirichlet Boundary conditions, that is, $u \equiv0$ on $\de_L D_\infty :=\partial D\times[0,+\infty)$.
   Then, the {total enthalpy} 
    \begin{equation}
        \mathcal{H}\big( u,\mu,h; \, t\big) := \int_D \big| u(x,t) + \big(\mu(x,t) - h(x)\big) \big| \dif x,
    \end{equation}
is finite for all $t\ge0$ and $\mathcal H(u,\mu,h;\cdot):[0,+\infty)\to[0,+\infty)$ is a continuous non-increasing function.
Moreover, the function 
$$[0,+\infty)\ni t\longmapsto \big(\mu(\cdot,t)-h\big)\in L^1(D),$$
is continuous with respect to the strong $L^1(D)$ topology.
\end{theorem}



\begin{theorem}[Stability of the mushy region]\label{thm:finite-melting}
    Let $D\subset\R^d$ be an open set and $\ell \in (0,1)$.
    Let $(u,\mu)$ be an enthalpy solution (in the sense of \Cref{def:enthalpy_solution}) in $D \times[0,+\infty)$ with homogeneous Dirichlet boundary conditions $u\equiv0$ on $\de_L D_\infty := \partial D \times(0,+\infty)$, initial temperature $g\in H^1_0(D) \cap L^1(D)$, and initial mushy coefficient $h\in L^\infty(D)$ that satisfies 
    \begin{equation}\label{e:intro-phase-separation}
    |h(x)|\le \ell <1 \qquad\text{for almost every}\quad x \in D.
\end{equation}
    Then, there exists $R_g := R_g(d,\ell,\| g\|_{L^1(D)})>0$, such that they hold the following inclusions
    \[
        \Omega^+_u(t) \subset \Omega_g^+ + B_{R_g} 
            \qquad \text{and}\qquad
        \Omega^-_u(t) \subset \Omega_g^- + B_{R_g}  \qquad\text{for all}\quad t\ge0.
    \]
    Moreover, $R_g \to 0$ as $\|g\|_{L^1(D)}\to0$.
\end{theorem}

\begin{theorem}[Continuity of the mushy region in the one-phase Stefan problem]\label{thm:finite-melting-one-phase}
    Let $D\subset\R^d$ be an open set and $\ell \in (0,1)$. 
    Let $(u,\mu)$ be an enthalpy solution of the one-phase Stefan problem in $D \times[0,+\infty)$ with homogeneous Dirichlet boundary conditions $u\equiv0$ on  $\de_L D_\infty := \partial D \times(0,+\infty)$, initial temperature $g\in H^1_0(D) \cap L^1(D)$, and initial mushy coefficient $h\in L^\infty(D)$ that satisfies \eqref{e:intro-phase-separation}.
    Then, there exists a modulus of continuity $R:[0,+\infty)\to[0,+\infty)$ depending on $\ell$, $d$, and $\|g\|_{L^1(D)}$ such that the following inclusions hold
    \[
        \Omega^+_u(t+s) \subset \Omega^+_u(t) + B_{\rho(t,s)}
    \qquad\text{for all}\quad t\ge0\quad\text{and}\quad s>0,
    \]
    where 
    $\rho(t,s):=R(\|u(t+s)-u(t)\|_{L^1})$. In particular, the positivity region $\Omega_u^+(t)$ varies continuously with respect to the Hausdorff distance. 
\end{theorem}

\begin{remark}
The condition \eqref{e:intro-phase-separation} guarantees that all the three initial phases $\Omega_g^+$, $\Omega_g^-$, and $\{g=0\}$ are are separately distinguished by a positive energy gap of $(1-\ell)$, ruling out possible anomalous spreading of the free boundaries (\Cref{prop:finite-speed-inner}).
Configurations with initial enthalpy $h$ satisfying the condition \eqref{e:intro-phase-separation} appear for instance in physical models describing the
thermal equilibrium at the melting temperature which has been reached through a slow heating or cooling, for which complex microscopic structure arises, that are approximated at macroscopic level by \eqref{e:intro-phase-separation}; see,  for instance, 
\cite{WorsterReesJones2015:SeaIce}.
The same condition also describes energetic stable configurations in models that involve more complex material, where such configuration is reached as a result of heating due to radioactive
decay or electrical resistance (see, for instance, \cite{Solomon2003:StefanProblems}). Finally, we notice that in the classical weak formulation of the Stefan problem, $h\equiv0$ in $\{g=0\}$, so the condition \eqref{e:intro-phase-separation} is automatically satisfied. For a discussion about the role of the initial enthalpy $h$ we refer to \cite{Paiano_Velichkov2026:ellipticregularization}, where we developed the corresponding existence theory for enthalpy solutions in this setting.  
\end{remark}

We notice that \Cref{thm:finite-melting-one-phase} implies that for every open set $\Omega$ one can find non-trivial solutions of the one-phase Stefan problem which are arbitrarily close to $\Omega$. 

\begin{corollary}\label{cor:asymptotic-regular}
    For all open sets $\Omega\subset \R^d$ and $\eps>0$, there exists a nonnegative $g_\eps \in H^1(\R^d)$ with $\Omega = \Omega_{g_\eps}^+$ such that the relative enthalpy solution $(u_\eps,\mu_\eps)$ in $\R^{d+1}_+$ with initial temperature $g_\eps$ and initial mushy coefficient $h\equiv0$ such that
    \[
       |\Omega_{u_\eps}^+(t)\setminus\Omega|<\eps\qquad\text{and}\qquad  \Omega \subset \Omega_{u_\eps}^+(t) \subset \Omega + B_{\eps}\qquad \text{for all}\quad t \ge0.
    \]
\end{corollary}

In the remaining part of this introduction, we introduce more in details \Cref{thm:total-monotonicity-formula,thm:finite-melting,thm:finite-melting-one-phase}, namely: we discuss their physical interpretations, their proofs, the challenges they present, and which novelties they bring with the respect to the current literature on the subject.

\subsection{The energy conservation law}\label{sub:intro:energyconservation}
The Stefan problem is formally derived from the first principle of thermodynamics, usually called: the energy conservation law. 
Indeed, both the pressure and the total volume remain fixed during the phase transition, and therefore the enthalpy coincides (up to an additive constant) with the internal energy of the system.
Moreover, we decompose the enthalpy as $H = u + \mu$ where $u$ is the {\it temperature} and $\mu$ is the {\it mushy coefficient}. The mushy coefficient $\mu$ carries the information about the phase of the material and represents the statistical limit of the infinitesimal mixture between of liquid and solid phases, which can be formally stated\footnote{We notice that the validity of this formula for weak solutions $(u,\mu)$ in the sense of \Cref{def:enthalpy_solution} at a point $(x,t)$ is currently an open problem.} as
\[
    \mu(x,t) = \lim_{r\to0} \frac{\big| \{ \text{liquid at time }t\} \cap B_r(x) \big| - \big| \{\text{solid at time }t\} \cap B_r(x) \big|}{| B_r|} \in [-1,1].
\]
Thus, $\{\mu\equiv1\}$ is the liquid phase, while $\{\mu\equiv-1\}$ is the solid one.
However, this energy formulation naturally allows for $\{|\mu|<1\}$ to have nonempty interior, thereby permitting intermediate states in which liquid and solid coexist in a microscopic mixture.
Such regions have been firstly investigated (numerically) by Atthey \cite{Atthey74}, that first called \emph{mushy region at time $t\ge0$} the set
\begin{equation}\label{e:mushy-intro}
    \mathcal M(t) := \big\{ x \in D : |\mu(x,t)| <1 \big\},
\end{equation}
showing a first example where this region appears as result of a heat source. Such properties have been then widely investigated (see, for instance, \cite{Meirmanov81:ExampleMushy,Primicerio83:Mushy,RogersBerger1984:MonotonicityMushyRegionTwoPhaseStefan,FasanoPrimicerio85:Mushy-variablemelting,GotzZaltzman:NonincreaseMushyInhomogeneousPb,Meirmanov_1992}).

Once this energetic decomposition of the enthalpy as $H = u +\mu$ is established, the conservation of energy prescribes that variations in enthalpy coincide with the heat flux through the boundary, namely
\begin{equation}\label{e:heat-flux}
\frac{\dif}{\dif t}\int_D \big(u(x,t) + \mu(x,t)\big) \dif x  = \int_{\de D} \de_\nu u(x,t) \dif \Hc^{d-1}x.
\end{equation}
Formally this leads to the following nonlinear diffusion problem
\begin{equation}\label{e:stefan-inclusion}
    \de_t (u+\mu) = \Delta u,
\end{equation}
which can be made rigorous (see \Cref{def:enthalpy_solution}) in the sense of integral identities, requiring that $(u,\mu)$ satisfies \eqref{e:weak_solution-1} and \eqref{e:compatibility_enthalpy-phases-1}.
This weak formulation is called the \emph{enthalpy formulation} and it is consistent with the classical one, due to the compatibility condition \eqref{e:compatibility_enthalpy-phases-1}.

\begin{center}
    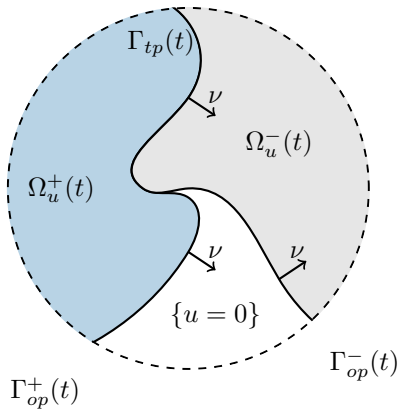
\begin{figure}[ht]
        \begin{tikzpicture}[use Hobby shortcut,scale=1.2]
		\begin{scope}
			\clip (0,0) circle (2);
			\draw[thick, fill = MidnightBlue!20] (-1,2) .. (0,1) .. (-0.5,0) .. (0,-0.1) .. (0,-0.7) .. (-1.8,-2) .. (-3,-2) .. (-2,2);
			\draw[thick, fill = black!10] (-1,2) .. (0,1) .. (-0.5,0) .. (0,0) .. (1,-1) .. (2.5,-2) .. (3,-2) .. (2,2);

			\draw[dashed, thick] (0,0) circle (1.99);
			
                \draw[thick,->] (0,1) -- (0.3,0.8);
                \draw[thick,->] (0,-0.7) -- (0.3,-0.9);
                \draw[thick,->] (1,-1) -- (1.3,-0.8);

		\end{scope}
			\node at (-1.4,0) {$\Omega_u^+(t)$};
			\node at (1,0.5) {$\Omega_u^-(t)$};
			\node at (0.3,-1.4) {$\{ u=0\}$};

            \node at (-0.3,1.6) {$\Gamma_{tp}(t)$};
            \node[rotate=-135, shift={(0,3.3)}, rotate=135] at (0,0) {$\Gamma^-_{op}(t)$};
            \node[rotate=145, shift={(0,3.3)}, rotate=-145] at (0,0) {$\Gamma^+_{op}(t)$};

            \node at (1.2,-0.7) {$\nu$};
            \node at (0.3,-0.7) {$\nu$};
            \node at (0.3,1) {$\nu$};
			
        \end{tikzpicture}
         \caption{An example of a solution of the Stefan problem with non-empty zero set. }
        \label{fig:non-empty-mushy}
    \end{figure}
\end{center}

Within this energetic framework, \Cref{thm:total-monotonicity-formula} can be interpreted as an analogue of the $L^1$-decay property for caloric functions with homogeneous Dirichlet boundary data.
Indeed, formally testing $\eta \equiv1$ in \eqref{e:weak_solution-1} yields
\begin{equation}\label{e:energy-conservation-law-intro}
\frac{\dif}{\dif t}\int_D \big(u(x,t) + \mu(x,t)\big) \dif x = 0,
\end{equation}
showing that the total energy is conserved.

This observation already appears in the works of Ushakov \cite{Ushakov1980:EnergyConservationLaw}, where the conservation law is established for one-phase solutions such that the temperature $u$ satisfies Neumann boundary conditions on $\partial D \times (0,+\infty)$, has initial datum $u(x,0)=g(x)$ with compact support, and the mushy coefficient satisfies $\mu \equiv \ell \in (-1,1)$ in $\{u=0\}$.\medskip

To extend these classical results to a more general setting, one must address certain integrability issues.
In particular, for unbounded domains $D\subset\R^d$ with $|D|=+\infty$ the integral of $(u+\mu)$ may fail to be well-defined. 
Indeed, the only a priori assumption on $\mu$ is the uniform bound $|\mu|\le 1$, which does not ensure integrability. To address this issue, we incorporate the initial datum $h(x):=\mu(x,0)$ into the definition of the total enthalpy $\mathcal H(t)$ and consider its absolute value rather than the integral itself. With this choice, at time $t=0$ the energy coincides with the $L^1$-norm of the temperature.
Within this formulation, $\mathcal H(t)$ measures the deviation of $(u,\mu)$ from the fixed zero energy level ($u\equiv 0$ and $\mu(x,t) = h(x)$). The decay of $\mathcal H(t)$ in time then follows from the energy conservation principle for one-phase solutions (\Cref{lemma:energy-conservation}) together with the Comparison Principle.

\Cref{thm:total-monotonicity-formula} also addresses the continuity of the energy $\mathcal H(t)$ in time. 
This property is not a direct consequence of a one-phase energy conservation law, rather, it reflects a genuine regularity feature of $u$ and $\mu$. 
Indeed, in the absence of additional assumptions on the initial datum $h$ (and hence on the geometry of the free boundary), uncontrolled phase mixing may occur. 
Such mixing could, in principle, produce compensating cancellations that preserve the global energy balance \eqref{e:energy-conservation-law-intro}, while still leading to discontinuities in $\mathcal H(t)$.

We rule out this phenomenon in \Cref{lemma:mono:continuity}. There we prove a stronger result, namely that $u(\cdot,t)$ and $\mu(\cdot,t)$ are separately strongly continuous in $L^1(D)$ with respect to time. 
This implies the continuity of $\mathcal H(t)$ since \eqref{e:compatibility_enthalpy-phases-1} implies that for all $t\ge 0$ and for almost every $x\in D$,
\[
    u(x,t) >0 \Rightarrow \mu(x,t) \ge h(x) 
        \qquad\text{and}\qquad
    u(x,t) < 0 \Rightarrow \mu(x,t) \le h(x).
\]
Consequently, we can rewrite 
\[
    \mathcal H(t) = \| u(\cdot,t) \|_{L^1(D)} + \| \mu(\cdot,t) - h \|_{L^1(D)},
\]
which, combined with the $L^1$-strong continuity of $u$ and $\mu$, yields the desired result.
To the best of the author’s knowledge, this result has no analogue in the existing literature and is novel even in the case of bounded domains $D \subset \mathbb{R}^d$.

\begin{remark}[On the sharp hypotheses for \Cref{thm:total-monotonicity-formula}]
    The monotonicity of $\mathcal H(t)$ for solutions defined in bounded sets $D\subset\R^d$ could be addressed through the continuity with respect initial data property for the solutions of the Stefan problem (see, for instance, \cite[Ch. I, Theorem 9]{Meirmanov_1992}). 
    The same argument could be adapted also in the unbounded case, following the strategy adapted in \cite{Paiano_Velichkov2026:ellipticregularization} for the Comparison Principle.
    Nevertheless, we favor our methodology for its flexibility, as it covers a broader class of cases with essentially no modification. 
    Specifically, \Cref{thm:total-monotonicity-formula} holds for an enthalpy solution $(u,\mu)$ given initial data $g \in L^1(D)$ and $h\in L^\infty(D)$, provided:
\begin{enumerate}[{\itshape (i)}]
    \item $u \in L^2_{loc}([0,+\infty);H^1_{loc}(D))$;
    \item\label{mono:gen:existence} There exist entire enthalpy solutions $(\overline u,\overline \mu)$ and $(\underline u,\underline \mu)$ with $\underline u, \overline u \in L^2_{loc}([0,+\infty);H^1_{loc}(\mathbb{R}^d))$, sharing the initial mushy coefficient $h$ but with initial temperatures $\overline{u}(x,0) = g(x)^+$ and $\underline{u}(x,0) = g(x)^-$;
    \item There exists a dense sequence of times $\{t_i\}_{i\in\mathbb{N}}$ in $[0,+\infty)$ such that the previous condition holds for $(u(\cdot,t_i),\mu(\cdot,t_i))$;
    \item The domain $D$, the boundary condition on $\de D \times (0,+\infty)$, and the solution $(u,\mu)$ permit the Comparison Principle for both the Stefan problem and the heat equation with respect to nonnegative entire solutions.
\end{enumerate}

These properties are naturally satisfied in our framework, though they also hold in other settings, such as homogeneous Neumann problems in smooth domains with suitable homogeneity assumptions (see \cite{GyryaSaloff-Coste2011:NeumanDirichlet-Heat-Kernels}), or for entire one-phase solutions with bounded initial temperatures, that is $g \in H^1_{loc}(\R^d) \cap L^1(\R^d) \cap L^\infty(\R^d)$ (see \cite{DingDuGuo2021:StefanProblemFisher-unbounded}).
\end{remark}

Once \Cref{thm:total-monotonicity-formula} is established, we can return to the energy conservation law.
Specifically, the results of Ushakov do not apply when $\mu(x,t)\not\equiv 0$ in $\{u=0\}$, and therefore they do not extend to the two-phase problem.
Moreover, without the $L^1(D)$-continuity in time of $u$ and $\mu$, one lacks the regularity needed to define the integral of $(u+\mu)$ for every $t \ge 0$, and the conservation law can only be established for almost every $t\ge0$.

In order to generalize this classical result, we first define the signed enthalpy $\mathcal H_0(t)$ as
\begin{equation}\label{e:energy-two-phase}
\mathcal H_0(t) = \mathcal H_0(u,\mu,h;t) := \int_{\mathbb{R}^d} \big( u(x,t) + (\mu(x,t) - h(x)) \big), \mathrm{d}x,
\end{equation}
where $h(x) = \mu(x,0)$. The energy conservation law then reads as follows.

\begin{theorem}[Energy Conservation Law]\label{thm:energy-conservation-law}
    Let $(u,\mu)$ be an enthalpy solution of the Stefan problem in $\mathbb{R}^d \times [0,+\infty)$ with initial temperature $g \in H^1(\mathbb{R}^d)$ and initial mushy coefficient $h \in L^\infty(\mathbb{R}^d)$. If $g \in L^1(\mathbb{R}^d)$, then $\mathcal H_0(t)$ is well defined for all $t \ge 0$ and is constant in time, that is,
    \[
        \mathcal H_0(t) = \int_{\R^d} g(x) \dif x \qquad \text{for all} \quad t\ge0.
    \]
\end{theorem}

\Cref{thm:energy-conservation-law} is essentially a corollary of \Cref{thm:total-monotonicity-formula}, and follows by testing \eqref{e:weak_solution} with a suitable family of nonnegative functions $\eta_R \nearrow 1$ (see \Cref{subsec:mono:monotonicity}).

\subsection{Long time behavior of the free boundary of the Stefan problem}\label{sub:intro:stability}
The evolution of the free boundary of solutions of the Stefan problem have been widely investigated since the 1980s.
These studies mainly deal with two configurations, which are the one-phase case problem with persistent energy injection \cite{Matano1983:AsymptoticBehaviorStefan,QuirosVazquez2001:AsymptoticConvergenceStefan,Kim_2006:Longtime,DuMatanoWang2014:RegularityAsymptoticBehaviorStefan} or the case where the initial data is a small perturbation of a stable configuration \cite{Meirmanov_1992,AthanasopoulosCaffarelliSalsa1998:FlatFreeBoundariesAreSmooth,Hadzic_Shkoller_2015:sphereStefan}.
In the first case, it is known that solutions with compact initial data and energy injection (given by a heat source or by prescribing lateral boundary data in exterior domains) becomes smooth for large times and converge to certain radial configurations. 
We also refer to \cite{QuirosVazquez2001:AsymptoticConvergenceStefan,Kim_2006:Longtime} for the analogous result in the framework of the one-phase Hele-Shaw problem.
In \cite{Meirmanov_1992,AthanasopoulosCaffarelliSalsa1998:FlatFreeBoundariesAreSmooth} the authors prove that two-phase solutions that are close to affine configurations share the same behavior, while in \cite{Hadzic_Shkoller_2015:sphereStefan} the authors prove that one-phase radial configurations are \emph{stable} under small and smooth perturbations. All these results are stability theorems for solutions which are close to some prescribed configuration. We notice that \Cref{thm:finite-melting,thm:finite-melting-one-phase}, on the other hand, are stability results for general initial data. 

The key ingredients and the main difficulties in the proof of \Cref{thm:finite-melting,thm:finite-melting-one-phase} are the following. 

\begin{itemize}
\item {\it $L^1$ stability.} In order to prove \Cref{thm:finite-melting,thm:finite-melting-one-phase} we need quantitative estimates for the free boundary movement from the monotonicity of the total enthalpy. We first notice that the enthalpy monotonicity (\Cref{thm:total-monotonicity-formula}), together with the energy condition \eqref{e:intro-phase-separation} on $h$, implies the following measure estimate (we refer to \Cref{lemma:measure-est}):
\begin{equation}\label{e:measure-estimate-intro}
    \big|\Omega_u^+(t) \setminus \Omega_g^+\big| 
        + \big|\Omega_u^-(t) \setminus \Omega_g^-\big|
        \le \frac{1}{1-\ell} \| g\|_{L^1(D)}.
\end{equation}
The latter concludes the first part of the proof as it provides an $L^1$ bound on the phases' growth, while \Cref{thm:finite-melting,thm:finite-melting-one-phase} deal with the distance traveled by the free boundary, which corresponds to an $L^\infty$ bound.\medskip

\item {\it Lack of uniform density estimates.} 
In geometric and free boundary problems, the passage from $L^1$ to $L^\infty$ bounds is usually obtained via density estimates or non-degeneracy properties at free boundary points; however, it is well-known that, in general, there are not such estimates for the Stefan problem: 
\cite{KingLaceyVasquez1995:AnglesHeleShaw,King95:WaitingTimes_HeleShaw_Stefan,Athanasopoulos_Caffarelli_Salsa_1996b,CaffarelliSalsa:GeomApproachToFreeBoundary,ChoiKim2006:WaitingTime}.
We solve this issue in \Cref{lemma:nofiltration}, where we prove such estimates at the fastest moving free boundary points. This suffices, as they act as natural barriers from above (and below) controlling the spreading of the phases of the solution.\medskip

\item {\it Density estimates at the fastest moving points.} In \Cref{lemma:nofiltration} we show that whenever a point $y_0 \in \Gamma_g^+$ is the first to cover a distance $r_0>0$ away from $\Gamma_g^+$ (and $\Gamma_g^+$ is \emph{flat} in $B_{2r_0}(y_0)$) reaching $x_0\in \Gamma^+(t_0)$ at time $t_0>0$, then in $B_{{r_0}}(x_0)$ the movement of the whole free boundary in that direction is comparable.
More precisely, we show that there exists a strictly increasing function $c_d = c_d(d,\|g\|_{L^1(D)}) : (0,+\infty) \to (0,+\infty)$ such that the following density estimate holds
\begin{equation}
    \label{e:density-small-intro}
    \big| \Omega_u^+(t_0) \cap B_{r_0}(x_0) \big| \ge (1-\ell)^{\sfrac d2} c_d(r_0).
\end{equation}
Furthermore, as the proof of \Cref{lemma:nofiltration} heavily relies on the parabolic boundary Harnack inequality, which degenerates as $r_0\to0$ (and thus $t_0\to0$) or as $\|g\|_{L^1}\to+\infty$\footnote{The boundary Harnack inequality actually involved the $L^\infty$ norm of the temperature $u$. However, thanks to the finite-speed of propagation property (\Cref{prop:finite-speed-inner}), controlling the $L^1$ norm of $g$ or the $L^\infty$ norm of $u$ is equivalent.}, if there exists a dimensional constant $C_d>0$ (for instance $C_d=100d$) such that $r_0> 1/C_d$ and $\|g\|_{L^1}\le C_d$, then we can improve \eqref{e:density-small-intro} showing that there exists a constant $c_d = c_d(d)>0$ which depends only on the dimension such that
\begin{equation}
    \label{e:density-big-intro}
    \frac{\big| \Omega_u^+(t_0) \cap B_{r_0}(x_0) \big|}{\big|B_{r_0} \big|} \ge (1-\ell)^{\sfrac d2}c_d .
\end{equation}
Our approach is inspired by the free boundary regularity theory developed in \cite{Athanasopoulos_Caffarelli_Salsa_1996a,Athanasopoulos_Caffarelli_Salsa_1996b,AthanasopoulosCaffarelliSalsa1998:FlatFreeBoundariesAreSmooth,DeSilvaForcilloSavin}, and the statement of the \Cref{lemma:nofiltration} can be also viewed as an Harnack inequality for the free boundary at such points.
\end{itemize}

\begin{remark}[about the density estimate \eqref{e:density-big-intro}]
We stress that for \eqref{e:density-big-intro} we do not assume any a priori non-degeneracy property for the solution; instead, we construct two regular barriers (one from above and one from below) that detect the speed of $\Gamma^+_u(t)$ at each time, thereby capturing any spatial or temporal non-homogeneity of the free boundary speed.
We also notice that our approach do not rely on reflection principles or on the comparison with some radial profile; this allows to treat initial data with arbitrary free boundary. 
\end{remark}

We conclude this introduction by pointing out the differences between \Cref{thm:finite-melting,thm:finite-melting-one-phase}.
Both the theorems deals with the expansion of the phases and provide information about how much each phase could grow.
In the one-phase problem, as the positive phase $\Omega_u^+(t)$ is monotonically increasing in time (\Cref{prop:one-phase:positivity-increasing}), it is sufficient to provide an exhaustive description of the free boundary over time, that leads to the modulus of continuity $R$ of \Cref{thm:finite-melting-one-phase}.

\begin{center}
    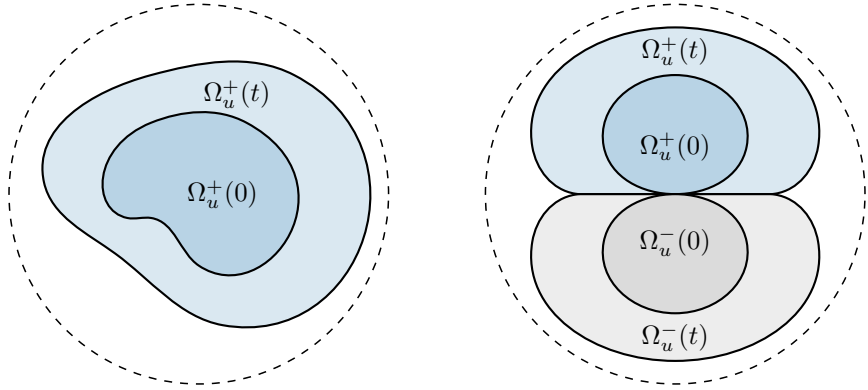
\begin{figure}[ht]
	\begin{tikzpicture}[use Hobby shortcut,scale=0.9]
		\begin{scope}[shift={(-3,0)},scale=0.7]
			\clip (0,0) circle (4);
			\draw[thick, fill = MidnightBlue!10,closed,shift={(-1,-0.5)},scale=1.5] (0,2) .. (-1.5,0.5) .. (-0.5,-0.5) .. (1,-1.5) .. (2,2);

			\draw[thick, fill = MidnightBlue!20,closed,shift={(-1,-0.5)}] (0,2) .. (-0.5,0) .. (0,0) .. (1,-1) .. (2,2);

			\draw[dashed, thick] (0,0) circle (3.99);

						\node at (0.5,0) {$\Omega_u^+(0)$};
						\node at (0.8,2.1) {$\Omega_u^+(t)$};

		\end{scope}

        \begin{scope}[shift={(4,0)},scale=0.7]
			\clip (0,0) circle (4);
            
			\draw[thick, fill = MidnightBlue!10] 
                (-2,0) -- (0,0) -- (2,0) .. (2.1,0.01) .. (3,1) .. (0,3.5) .. (-3,1) .. (-2.1,0.01) .. (-2,0);
			\draw[thick, fill = MidnightBlue!20,closed] 
                (-1.5,1) .. (0,0).. (1.5,1) .. (0,2.5);
                
			\draw[thick, fill = black!7,yscale =-1]     
                (-2,0) -- (0,0) -- (2,0) .. (2.1,0.01) .. (3,1) .. (0,3.5) .. (-3,1) .. (-2.1,0.01) .. (-2,0);
			\draw[thick, fill = black!14,closed,yscale=-1]     
                (-1.5,1) .. (0,0).. (1.5,1) .. (0,2.5);

			\draw[dashed, thick] (0,0) circle (3.99);

						\node at (0,-1) {$\Omega_u^-(0)$};
			\node at (0,1) {$\Omega_u^+(0)$};
			\node at (0,3) {$\Omega_u^+(t)$};
			\node at (0,-3) {$\Omega_u^-(t)$};

		\end{scope}

    \end{tikzpicture}
         \caption{An example of one-phase solution with liquid phase asymptotically bounded on the left, and an example of two phase solution with persistent mushy region and branching points on the right.}
        \label{fig:asymptotic-mushy-intro}
    \end{figure}
\end{center}

\subsection{Plan of the paper}
In \Cref{sec:one-phase} we introduce some preliminary definitions and results, with a focus on the one-phase solutions. 
\Cref{sec:monotonicity-formula} we prove \Cref{thm:total-monotonicity-formula} as well as \Cref{thm:energy-conservation-law}. 
In \Cref{sec:asymptotic-analysis} we prove \Cref{thm:finite-melting,thm:finite-melting-one-phase}. 
We conclude the paper in \Cref{sec:examples}, where we provide some interesting example of solutions with bounded dynamics and persistent branching points.

\addtocontents{toc}{\protect\setcounter{tocdepth}{0}}
\section*{Notation}
We fix some notation that will come handy during the whole paper. Through the whole paper we will work with sets $D\subset\R^d$ and $E \subset \R^d\times[0,+\infty)$, where $d\ge2$ is fixed along the paper.\medskip

\paragraph{Space-time sets}
    We use $E\subset \R^d \times[0,+\infty)$ to indicate a space-time set, then we give the following definitions.
    \begin{itemize}
        \item $E(t) := E \cap \{ x\in \R^d : (x,t) \in E\}$ is the slice of $E$ at a fixed time.
        \item For $0\le t_1\le t_2$ we say that $E$ is \emph{cylindrical} if $E = D \times (t_1,t_2]$ for some $D\subset \R^d$, and we define the following boundary set:
        \begin{itemize}
            \item the \emph{lateral boundary of $E$} as $\de_L E := \de D \times(t_1,t_2)$;
            \item the \emph{parabolic boundary of $E$} as $\de_P E := \de_L E \cap E(t_1)$.
        \end{itemize}
    \end{itemize}\medskip

\paragraph{Spatial sets} For $D\subset\R^d$ we define the following sets:
    \begin{itemize}
        \item for all $T>0$, $D_T := D \times [0,T)$ (with $D_\infty := D \times[0,+\infty)$); we will continue to use the notation $D\times[0,T)$ whenever we value it more suitable in the context;
        \item if $D = \R^d$, we use the notation $\R^{d+1}_+ := \R^d\times[0,+\infty)$;
        \item for all $t\ge0$, we write $D(t) := D \times\{t\}$.
        \item for all open set $D \subset \R^d$ and $T>0$ we define $H^1(D_T)$ and $H^1(D)$ as the Sobolev spaces in $\R^{d+1}$ and $\R^d$ (respectively). We also define $H^1_{0,P}(D_T)$ and $H^1_{0,L}(D_T)$ respectively as the closure of $C^{\infty}_c(D\times(0,+\infty))$ and $C^\infty(D\times\R)$.
        Having homogeneous Dirichlet boundary datum means $u \in H^1_{0,L}(D_T)$. 
    \end{itemize}

\paragraph{Cubes} Given $(x_0,t_0)\in \R^d\times(0,+\infty)$ and $r>0$, we define the following objects:
    \begin{itemize}
        \item $B_r(x_0) \subset \R^d$ is the ball of radius $r>0$ centered in $x_0$;
        \item $C_r(x_0) := B_r(x_0) \times[0,+\infty) \subset \R^d\times[0,+\infty)$ is the cylinder over $B_r(x_0)$;
        \item if $r<2\sqrt{t_0}$, then $Q_r(x_0,t_0) \subset \R^d \times(0,+\infty)$ is the parabolic cube centered in $(x_0,t_0)$ of size $r$ defined as
            \[
                Q_r(x_0,t_0) := \left\{(x,t) \in \R^d\times(0,+\infty) : |x-x_0|<r \quad \text{and}\quad t_0-\frac{r^2}{2} < t \le t_0 \right\}.
            \]
    \end{itemize}
If $x_0=0$, $t_0=1$, and $r =1$ we will write $B = B_r(0)$, $Q = Q_1(0,1)$.

\paragraph{Temperature}
    We will call \emph{temperature} $u$ a function $u : D \times[0,+\infty) \to \R$, and we write $u = u^+ + u^-$ dividing its positive and negative parts. Thus we define
    \begin{itemize}
        \item its positive (and negative) phases as
        \[
            \Omega^\pm_u := \big\{(x,t) \in D\times[0,+\infty) : \pm u(x,t) >0 \big\};
        \] 
        \item $\Gamma^\pm_u := \de \Omega^{\pm}_u$ its positive and negative interfaces;
        \item $\Gamma_u = (\de\Omega^+_u \cup \de\Omega^-_u ) \cap D= (\Gamma^+_u \cup \Gamma^-_u )\cap D$, the \emph{free boundary};
        \item $\Gamma_{op}^\pm := \Gamma^\pm \setminus \Gamma^\mp$ is the \emph{one-phase part} of the free boundary, $\Gamma_{tp} := \Gamma^+ \cap \Gamma^-$ is the \emph{two-phase one}.
    \end{itemize}\medskip

\paragraph{Integral and differential operators} 
    For any $u : D \times[0,+\infty)\to\R$, $D\subset\R^d$, and $E \subset D \times[0,+\infty)$ we adopt the following notation:
    \begin{itemize}
        \item with $|D|$ and $|E|$, we indicate both the $d$ and $(d+1)$ dimensional Lebesgue measure, while we use $\Hc^{d-1}(D)$ for the $(d-1)$-Hausdorff measure of $D\subset\R^d$;
        \item we use all the following convention for the integrals at fixed time, depending on the circumstances
        \[
            \int_{E(t)} u \dif x = \int_{E}u(x,t) \dif x = \int_{E} u \dif x \bigg|_{\tau=t};
        \]
        \item we omit the $D\times[0,+\infty)$ term in the integrals
        \[
            \iint u \dif x \dif t := \iint_{D\times[0,+\infty)} u (x,t) \dif x \dif t;
        \]
        \item we write $u \in C^{\alpha}_x(D)$ for $\alpha \in(0,1)$ if $u$ is $\alpha$-H\"older in space, and $u\in C^{0,1}(D)$ if $u$ is Lipschitz. 
        We write $u \in C^{1,\alpha}_x( D)$ if $u \in C^1_x(D)$, and $\nabla u \in C^\alpha_{x}(D;\R^d)$, with $\alpha\in(0,1]$;
        \item a set $D\subset\R^d$ has locally Lipschitz boundary if $\de D$ locally described by the graph of a (space) Lipschitz function;
        $\de D$ is locally $C^{1,\alpha}_x$, for $\alpha\in(0,1)$, if it is locally the graph of a $C^{1,\alpha}_x(D)$ function;
        \item if $E$ is open, $k,\ell \in \NN$, then we write $u\in C^{k,\ell}_{x,t}(E)$ if $u$ is $k$ times differentiable in $D\times(0,+\infty)$ in the space variables, $\ell$ times in the time one and all these derivatives are continuous. We write $u\in C^{k,\ell}_{x,t}\left(\overline E\right)$ if the derivatives are continuous up to the boundary;
        \item we adopt the following conventions
            \[
                D u = D_{x,t} u, \qquad \nabla u = \nabla_x u,\qquad \text{and}\qquad \Delta u = \Delta_x u.
            \]
    \end{itemize}

\addtocontents{toc}{\protect\setcounter{tocdepth}{1}}

\section{Preliminary results for one-phase solutions}\label{sec:one-phase}
The results contained in this section are a prerequisite for the analysis that follows.
Specifically, in \Cref{sec:enthalpy-formulation} we rigorously introduce the enthalpy formulation and its comparison principle, while in \Cref{sec:classical-formulation} we present the classical formulation (we refer to \cite{Paiano_Velichkov2026:ellipticregularization} for a complete introduction). In the subsequent sections we focus on the properties of one-phase solutions.
Namely: in \Cref{sec:l1-decay} we present some relations with caloric functions while in \Cref{sec:one-phase-subsol} we exploit some monotonicity properties of one-phase solutions.
We do not claim the results of this section to be completely original; however, we give the proofs of a series of results that we cannot find in the literature within the same generality we need for the following discussion.

\subsection{Enthalpy formulation}\label{sec:enthalpy-formulation}
The weak formulation of the Stefan problem (or equivalently the enthalpy formulation) for Dirichlet boundary data is given by an integral identity for the enthalpy of the system, or more precisely, for the pair $(u,\mu)$. We refer to \cite{Paiano_Velichkov2026:ellipticregularization} for further discussion on the topic and for the proofs of the results presented in this section.

\begin{definition}[Enthalpy solution]\label{def:enthalpy_solution}
	Let $D\subset\R^d$ be an open set and $T \in (0,+\infty]$.
	A couple $(u,\mu)$ of Lebesgue measurable functions $u:D_T\to\R$ and $\mu:D_T\to\R$ is an \emph{enthalpy solution} (or equivalently \emph{weak solution}) of the Stefan problem in $D_T$ if the following hold:
	\begin{enumerate}[{\itshape (i)}]
		\item\label{item:as:regularity} \textit{Regularity of the temperature in time.}
		$u \in L^2_{loc}([0,T);H^1(D))$ and $\de_t u \in L^2_{loc}([0,T);L^2(D))$.
		\item\label{item:as:continuity} \textit{Continuity of the mushy coefficient in time.} The function $\mu:[0,T)\to L^\infty(D)$ is continuous with respect to the weak-$\ast$ topology in $L^\infty(D)$, that is:
		\[
		\int_{D} \mu(x,t) \, \eta(x) \dif x = \lim_{\tau \to t} \int_{D}\mu(x,\tau) \, \eta(x) \dif x,
		\]
		for all $t\in[0,T)$ and all $\eta\in L^1(D)$.
		\item\label{item:as:compatibility}\textit{Compatibility.} For all $t\in[0,T)$, there exists $N_t \subset D$ such that $|D\setminus N_t|=0$ and
		\begin{equation}\label{e:compatibility_enthalpy-phases}\tag{C-ST}
			\mu(x,t) = 
			\begin{cases}
				\begin{array}{rl}
					1 & \text{if}\quad u(x,t)>0,\\
					-1 & \text{if} \quad u(x,t)<0,
				\end{array}
			\end{cases} \qquad\text{for every}\quad x\in N_t.
		\end{equation}
		\item \label{item:as:weakformulation} \textit{Weak formulation.}
		The following integral identity holds
		\begin{equation}\label{e:weak_solution}\tag{H-ST}
			\int_{D(t)}(u+\mu)\eta \dif x \bigg|_{t=t_1}^{t_2} 
			= \int_{t_1}^{t_2}\int_D (u+\mu) \de_t \eta 
			-\nabla u \cdot \nabla \eta \dif x \dif t,
		\end{equation}
		for all $0\le t_1\le t_2 < T$ and all $\eta \in C^\infty_c(D\times\R)$.
	\end{enumerate}
\end{definition}

\begin{definition}[Entire solutions]
    Let $(u,\mu)$ as in \Cref{def:enthalpy_solution}, we say that $(u,\mu)$ is an \emph{entire solution} if $D = \R^d$ and $T= +\infty$, and thus $(u,\mu)$ is defined in $\R^{d+1}_+$.
\end{definition}

The same approach gives the notion of subsolution and supersolution.

\begin{definition}[Dirichlet enthalpy sub/supersolutions]\label{def:weak-sub-super}
	Let $D \subset \R^d$ and $T \in (0,+\infty]$. 
	Let also $u$ and $\mu$ be measurable functions in $D\times(0,+\infty)$.
	We say that a couple $(u,\mu)$ is an \emph{enthalpy subsolution} (resp. \emph{supersolution}) \emph{with  Dirichlet boundary conditions} if it satisfies the conditions \ref{item:as:regularity}-\ref{item:as:compatibility} of \Cref{def:enthalpy_solution} and if the following integral inequality holds
	\begin{equation}\label{e:subsolution-enthalpy}
		\int_{D(t)} (u+\mu) \eta \dif x\biggl|_{t=t_1}^{t_2} \le \int_{t_1}^{t_2}\int_{D} (u+\mu) \de_t\eta - \nabla u \cdot \nabla \eta \dif x \dif t\,,\qquad (\text{resp.}\; \ge),
	\end{equation}
	for all $0\le t_1\le t_2< T$ and all nonnegative  test functions $\eta\ge 0$,  where $\eta \in C^\infty_c(D\times \R)$.
\end{definition}

One of the main properties of enthalpy solutions is the comparison principle, that (among other properties) implies the uniqueness of solutions.

\begin{theorem}[Comparison principle \cite{Paiano_Velichkov2026:ellipticregularization}]\label{lemma:comparison-stefan}
	Let $D\subset\R^d$ be an open set and $T^\ast\in(0,+\infty]$.
	Let $(u_1,\mu_1)$ and $(u_2,\mu_2)$ be respectively a Dirichlet enthalpy subsolution and a Dirichlet enthalpy supersolution  in $D_{T^\ast}$. Assume that
	\[
	(u_2- u_1)^- \in H^1_{0,P}(D_{T^*}) \quad \text{and} \quad \mu_1(x,0) \le \mu_2(x,0) \quad \text{for almost every} \quad x\in D.
	\]
	Then, 
	for all $T\in(0,T^\ast)$, we have:
	\[
	u_1(x,T)\le u_2(x,T) 
	\quad\text{and}\quad
	\mu_1(x,T) \le \mu_2(x,T)\quad \text{for almost every}\quad x\in D.
	\]
\end{theorem}

A special class of solutions are the so called one-phase solutions.

\begin{definition}[One-phase enthalpy solutions]
    Let $D \subset\R^d$ be an open set and $T\in(0,+\infty]$. 
    The pair $(u,\mu)$ is a one-phase (nonnegative) solution of the Stefan problem in $D_T$ if it is an enthalpy solution in $D_T$ and if for all $t\in[0,T]$, $u(\cdot,t) \ge0$ almost everywhere in $D$.
\end{definition}

\begin{remark}
    The comparison principle (\Cref{lemma:comparison-stefan}) implies that if $u \ge0$ almost everywhere on the parabolic boundary $\de_P D_T$, then $(u,\mu)$ is a one-phase solution.
\end{remark}

Since the sign of the solutions' temperature is preserved, we can specialize the notion of sub/super-solutions to the one-phase problem.

\begin{definition}[one-phase enthalpy sub/super-solutions]
    Let $D \subset \R^d$ and $T \in (0,+\infty]$. The pair $(u,\mu)$ is a one-phase subsolution (resp. supersolution) in $D_T$ if it is a subsolution (resp. supersolution) in the sense of \Cref{def:weak-sub-super} and for all $t\in[0,T]$, $u(x,t)\ge0$ almost everywhere in $D$.
\end{definition}

\begin{remark}
    The theory is “skew"-symmetric with the respect to sign changes. 
    Namely, if $(u,\mu)$ is an enthalpy subsolution/solution/supersolution in $D_T$ and for all $t\in[0,T]$, $u(x,t)\le0$ almost everywhere, then $(-u,-\mu)$ is a one-phase nonnegative supersolution/solution/subsolution.
\end{remark}

\subsection{Classical formulation}\label{sec:classical-formulation}
In this section, we briefly introduce the notions of classical solutions of the Stefan problem, focusing on two-phase solutions with non-empty mushy region or one-phase solutions. We will also introduce the notion of the classical sub/super-solutions in the aforementioned cases.
The latter will play a handful role through all the paper.

Let us fix an open set $D \subset \R^d$, let $u : D \to \R$ be a locally Lipschitz function that represent the temperature, and let us define the space-time free boundaries $\Gamma^+_u$ and $\Gamma^-_u$ as
\[
    \Gamma_u^+ := \de \Omega_u^+ \cap D \times[0,+\infty) 
        \qquad\text{and}\qquad
    \Gamma_u^- := \de \Omega_u^- \cap D \times[0,+\infty).
\]
We assume that $\Gamma^+_u$ and $\Gamma^-_u$ are traversal to $D \times\{t\}$ so that, for all $t\ge0$, we define the sets $\Gamma_u^\pm(t)$ and $\Sigma_u^\pm(t)$ as follows
\[
    \Gamma_u^\pm(t) := \Gamma_u^\pm \cap (D\times\{t\}) 
        \qquad\text{and}\qquad 
    \Sigma_u^\pm(t) := \big\{ x \in D : (x,t) \in \Gamma^\pm_u\big\} \subset D,
\]
are $(d-1)$-dimensional Lipschitz surfaces.
Moreover, we define space-time normal $\widetilde\nu^\pm$ to $\Gamma^\pm_u$ (pointing outwards $\Omega_u^+$ and inwards $\Omega_u^-$) and we write
\[
    \widetilde\nu^\pm(x,t) = \big(\widetilde\nu_x^\pm(x,t),\widetilde\nu_t^\pm(x,t) \big) \in \R^d\times\R.
\]
We call $\nu^\pm$ the space normal of $\Sigma_u^\pm(t)$ (with the same sign convention), since $\Gamma_u^\pm$ are traversal to $D \times\{t\}$, then it satisfies the following identity
\[
    \nu^\pm(x,t) = \frac{\widetilde\nu_x^\pm(x,t)}{|\widetilde\nu_x^\pm(x,t)|}.
\]
Transversality allows us also to characterize the  normal speed $V_\nu^\pm$ of $\Sigma_u^\pm$ as follows
\[
    V_\nu^\pm(x,t) = - \frac{\widetilde\nu_t^\pm(x,t)}{|\widetilde\nu_x^\pm(x,t)|}.
\]
With these definition in mind, we can now state the definition of classical solutions, subsolutions and supersolutions of the Stefan problem.

\begin{definition}[Classical solution (resp. sub/supersolution)]\label{def:classical-sol}
    Let $D\subset\R^d$ be an open set with smooth boundary and $h\in C(D)$ with $|h(x)|<1$ and let $u : D \times[0,+\infty) \to \R$ be a continuous function.
    Let us also assume that the following properties hold:
    \begin{enumerate}[(a)]
     \item\label{item:classical-1} $u:D\times[0,+\infty)\to\R$ is a Lipschitz continuous function in space-time satisfying 
    \begin{equation}\label{e:regularity-assumptions-u-classical-solutions}
        u^+ \in C^{1,1}_{x,t}\big(\overline{\Omega_u^+}\big) \cap C^2_x \big( \Omega_u^+\big) 
            \qquad\text{and}\qquad
        u^- \in C^{1,1}_{x,t}\big(\overline{\Omega_u^-}\big) \cap C^2_x \big( \Omega_u^-\big);
    \end{equation}     
    \item\label{item:classical-2} the surfaces $\Gamma^\pm:=\partial\Omega_u^\pm\cap\big(D\times(0,+\infty)\big)$ are Lipschitz in space-time;
    \item\label{item:classical-3} there are $
    \Gamma_{tp}\subset\Gamma^+\cap\Gamma^-$, $\Gamma^+_{op}\subset\Gamma^+\setminus\Gamma^-$ and $\Gamma^-_{op}\subset\Gamma^-\setminus\Gamma^+$ such that:
    \begin{itemize}
    \item $\Gamma_{tp}$ and $\Gamma_{op}^\pm$ are relatively open subsets of $\Gamma^\pm$,
    \item $\Gamma_{tp}$ and $\Gamma_{op}^\pm$ are $C^1$-regular $d$-dimensional manifolds transversal to each $D\times\{t\}$,
    \item $\Gamma_{tp}$, $\Gamma_{op}^+$, and $\Gamma_{op}^-$ cover almost all the free boundary, that is
    \[
        \mathcal H^{d}\big(\Gamma^+ \setminus (\Gamma_{tp} \cup \Gamma_{op}^+ \big)\big) =0 
        \qquad\text{and}\qquad 
        \mathcal H^{d}\big(\Gamma^- \setminus (\Gamma_{tp} \cup \Gamma_{op}^- \big)\big) =0.
    \]
    \end{itemize}
    \end{enumerate}
    Then, we say that $u$ is a classical solution (resp. sub/supersolution) of the Stefan problem with respect to the initial mushy coefficient $h(x)$ if:
  \begin{equation}\label{e:classical-solution-full}\tag{ST}
	\begin{cases}
		\de_t u^+ = \Delta u^+ & \text{in}\quad \Omega^+_u,\\
		\de_t u^- = \Delta u^- & \text{in}\quad \Omega^-_u,\\
		V_\nu = \frac{1}{1-h(x)}|\nabla u^+| &\text{on}\quad \Gamma_{op}^+,\\
		V_\nu = \frac{1}{h(x)-1} |\nabla u^-| &\text{on}\quad \Gamma_{op}^-,\\
		V_\nu =\frac{1}{2}\left( |\nabla u^+|- |\nabla u^-| \right)&\text{on}\quad \Gamma_{tp}. 
	\end{cases}\qquad\text{(resp. $\le0$ or $\ge0$).}
\end{equation}
\end{definition}

\begin{definition}[Classical one-phase solution (resp. sub/supersolution)]\label{def:classical-sol-op}
    Let $D\subset\R^d$, $h \in C(D)$ be as in \Cref{def:classical-sol}. Let $u : D \times[0,+\infty)\to\R$ be a nonnegative continuous function.
    Let us assume that \ref{item:classical-1}-\ref{item:classical-3} hold and that 
    \begin{equation}\label{e:op-classical-growth}
        \Omega_u^+(t_1) \subset \Omega_u^+(t_2) \qquad\text{for all}\quad t_1 \le t_2.
    \end{equation}
    Then we say that $u$ is classical one-phase solution (resp. sub/supersolution) of the Stefan problem with respect to the initial mushy coefficient $h(x)$ if
    \begin{equation}\label{e:classical-solution-op}\tag{ST-op}
	\begin{cases}
		\de_t u^+ = \Delta u^+ & \text{in}\quad \Omega^+_u,\\
		V_\nu = \frac{1}{1-h(x)}|\nabla u^+| &\text{on}\quad \Gamma_{op}^+.\\
	\end{cases} \qquad\text{(resp. $\le0$ or $\ge0$).}
    \end{equation}
\end{definition}

\begin{proposition}[\cite{Paiano_Velichkov2026:ellipticregularization}]\label{prop:classical-are-weak}
    Let $D\subset\R^d$, $h \in C(D)$, and $u : D \to \R$ be a classical solution (resp. sub/supersolution) of the Stefan problem with respect to the initial mushy coefficient $h$ in the sense \Cref{def:classical-sol} or \Cref{def:classical-sol-op}.
    Let $\mu : D \times[0,+\infty) \to [-1,1]$ be defined as
    \[
        \mu(x,t) = 1 \qquad\text{in}\quad \Omega_u^+,
            \qquad
        \mu(x,t) = -1 \qquad \text{in}\quad \Omega_u^-,
            \qquad\text{and}\qquad
        \mu(x,t) = h(x) \qquad\text{otherwise.}
    \]
    Then $(u,\mu)$ is an enthalpy solution (resp. sub/supersolution) of the Stefan problem in the sense of \Cref{def:enthalpy_solution}.
\end{proposition}

\begin{remark}
    The reason for the monotonicity assumption \eqref{e:op-classical-growth} in \Cref{def:classical-sol-op} is two-fold.
    First, from a physical point of view there is no heat flux pushing the free boundary backward, so its natural. Second, for enthalpy solutions such property holds (see \Cref{prop:one-phase:positivity-increasing}); therefore, since classical solutions (resp. sub/super-solutions) define enthalpy solutions the assumption holds for solutions and supersolutions. 
    It is natural to assume \eqref{e:op-classical-growth} also in the subsolution case as we define sub/supersolution mainly to understand the behavior of solutions, and so it leads to no loss of generality.
\end{remark}

\subsection{Relation with the heat equation and \texorpdfstring{$L^1$}{L1} integrability properties}\label{sec:l1-decay}

It comes with no surprise that Stefan problem has several links with the heat equation. 
In particular, for what is our concern, we use them only to use the comparison principle between subcaloric and supercaloric functions.

\begin{definition}[Caloric, subcaloric and supercaloric functions]
    Let $D\subset\R^d$, and $E \subset D \times[0,+\infty) = D_\infty$ be open sets, respectively in $\R^d$ and in $D_\infty$. We say that $u : D\times[0,+\infty) \to \R$ is \emph{subcaloric} (resp. \emph{supercaloric}) in $E$ if, for all nonnegative $\eta \in H^1_0(E)$ with bounded support it holds
    \begin{equation}\label{e:heat-subsolution-interior}
        \iint_{E} u \,\de_t\eta - \nabla u \cdot \nabla \eta \dif x \dif t\ge0. \qquad (\text{resp.}\; \le ).
    \end{equation}
    We say that $u$ is \emph{caloric} if it is both subcaloric and supercaloric.
\end{definition}

If $E = D_T = D \times[0,T)$ and $u$ is subcaloric (resp. supercaloric) in $E$, then for all nonnegative $\eta \in H^1(D_T)$ with bounded support and such that $\eta\equiv 0$ on $\de_L D_T$, then for all $0\le t_1 < t_2\le T$ it holds
\begin{equation}\label{e:heat-subsolution}
    \int_{D(t)} u\, \eta \dif x \bigg|_{t=t_1}^{t_2} \le \int_{t_1}^{t_2}\int_D u \,\de_t \eta - \nabla u \cdot \nabla \eta \dif x \dif t. \qquad (\text{resp.}\; \ge).
\end{equation}
The latter is classical and can be easily obtained by testing the inequality with $\eta(x,t) \, \ind^{\eps}_{[t_1,t_2]}(t)$, where $\ind^{\eps}_{[t_1,t_2]} \in C^\infty_c((t_1,t_2))$ is a smooth approximation of the characteristic function $\ind_{[t_1,t_2]}$.

\begin{proposition}\label{prop:subcaloric}
    Let $D\subset\R^d$ as above and $(u,\mu)$ be an enthalpy subsolution (resp. supersolution)  of the Stefan problem in $D_\infty = D\times[0,+\infty)$. Then $u^+$ is subcaloric (resp. $u^-$ is supercaloric) in $D_\infty$.
\end{proposition}

\begin{proof}
    We only deal with the case of subsolutions, the other follows analogously. 

    For all $\sigma>0$, we introduce $p_\sigma \in C^\infty(\R)$ such that,
    \[
        p_\sigma(z) = 0 \quad\text{if}\quad z\le \sigma, \qquad p_\sigma(z) = 1\quad\text{if}\quad z\ge 2\sigma, \qquad\text{and}\qquad 0 \le p_\sigma'(z) \le \frac2\sigma.
    \]
    In particular, for all $\eta \in C^\infty_c(D\times(0,+\infty))$, $p_\sigma(u)\, \eta \in H^1_0(\Omega_u^+)$ is non-negative and with compact support in $D\times(0,+\infty)$.
    Therefore, since $\mu\equiv1$ in $\Omega^+$,  the subsolution condition \eqref{e:subsolution-enthalpy} reads as follows
    \[
            \iint (u+1) \de_t\big(p_\sigma(u) \eta\big) - \nabla u \cdot \nabla \big(p_\sigma(u) \eta\big) \dif x \dif t \ge0.
    \]
    Then, the time derivative reads as
    \begin{align*}
        \iint (u+1) \de_t\big(p_\sigma(u) \eta\big) \dif x \dif t
            &= \iint u \big(\de_tp_\sigma(u) \eta + p_\sigma(u)\de_t\eta\big) \dif x \dif t
                + \iint \de_t\big(p_\sigma(u) \eta\big)\dif x \dif t\\
            &= \iint  \big( u\,\de_t\eta - \de_t(u\,\eta)\big) p_\sigma(u)\dif x \dif t\\
            &= - \iint \de_tu\,\eta\, p_\sigma(u)\dif x \dif t,
    \end{align*}
    where the second term in the first line is zero since $p_\sigma(u)\,\eta \in H^1_0(\Omega^+_u)$.
    
        Lastly, the spatial term reads as
    \begin{align*}
        \iint \nabla u \cdot \nabla \big(p_{\sigma}(u) \eta\big) \dif x \dif t 
            &=\iint p'_{\sigma}(u) |\nabla u|^2\eta + (\nabla u \cdot \nabla \eta)p_{\sigma}(u) \dif x \dif t,
    \end{align*}
    and thus, since $p_{\sigma}'\ge 0$, dropping this nonnegative term from the inequality, the subsolution condition reads as follows
    \[
        - \iint \big(\de_t u \, \eta + (\nabla u \cdot \nabla\eta)\big) p_\sigma(u) \dif x \dif t \ge0.
    \]
    Finally, since $p_\sigma(u) \to \ind_{\Omega_u^+}$ as $\sigma\to0$, we conclude the proof by integrating by parts $u^+$ in the time variable.
\end{proof}

In the special case of solutions, we have the following.

\begin{proposition}\label{prop:caloric}
Let $D \subset\R^d$, $T \in (0,+\infty]$, and let $(u,\mu)$ be an enthalpy solution of the Stefan problem in $D_T = D \times[0,T)$.
Then, the sets $\Omega^\pm\subset D \times(0,T)$ are open and $u^\pm$ are caloric in $\Omega^\pm$.
\end{proposition}

\begin{proof}
    The temperature $u$ is continuous by the classical results in \cite{DiBenedetto1980:Stefan2,DiBenedetto1982:Stefan1,CaffarelliEvans1983:ContinuityTwoPhaseStefan,Ziemer1982:Stefan,BaroniKuusiUrbano:TwoPhaseStefan}, and thus the sets $\Omega^\pm$ are open.
    Let us restrict ourselves to the positivity set, since the other case follows with no changes. 
    It makes sense to consider $\eta \in C^\infty_c(\Omega^+)$ as a test function.
    Then we can argue as in the previous proposition, taking $p_\sigma$ and recovering
    \[
        - \iint \big(\de_t u \, \eta + (\nabla u \cdot \nabla\eta)\big) p_\sigma(u) \dif x \dif t = \iint |\nabla u|^2\,p'_\sigma(u)\,\eta \dif x \dif t.
    \]
    But now, since $\eta\in C^\infty_c(\Omega^+)$, there exists $\sigma_0>0$ such that 
    \[
        \eta(x,t) \equiv0 \qquad \text{if} \quad u(x,t) \le \sigma_0,
    \]
    and thus, if $\sigma < \sfrac{\sigma_0}{2}$, then $p'_\sigma(u)\eta \equiv0,$
    and taking the limit as $\sigma\to0$ we conclude the proof.
\end{proof}

\begin{remark}
    It is a classical result that the extension of nonnegative caloric function to zero is subcaloric.
    However, since it includes the case of subsolutions, \Cref{prop:subcaloric} does not follow from \Cref{prop:caloric}, since $\Omega^\pm_u$ may be not be open if $u$ is not sufficiently regular (and the regularity holds only for $u$ that are temperatures of  solutions of the Stefan problem).
\end{remark}

Since $u^+$ is a solution of the heat equation and belongs to $H^1(D_T)$, for all $T>0$, then the standard comparison property of the heat equation holds.

\begin{lemma}\label{lemma:comparison-heat}
    Let $D\subset\R^d$, $T>0$, and $(u,\mu)$ be an enthalpy subsolution of the Stefan problem in $D_T = D\times[0,T)$.
    Let also $v\in H^1(D_T)$ is a supercaloric nonnegative function in $D_T$. 
    If $u\le v$ on $\de_P D_T$, then
    \[
        u(x,t) \le v(x,t) \qquad\text{for almost every}\quad (x,t) \in D_T. 
    \]
\end{lemma}

\begin{proof}
    Let us call $w := (u-v)\in H^1(D_T)$. Then $w$ is subcaloric and $w\le 0$ on $\de_P D_T$, and thus $w^+$ is an admissible test for $w$.
    Thus, for all $t_0>0$, it holds
    \begin{align*}
        \int_{D(t_0)} (w^+)^2 \dif x 
            &\le \int_0^{t_0} \int_D w \de_t (w^+) - \nabla w \cdot \nabla (w^+) \dif x\dif t\\
            &= \int_0^{t_0} \int_D \frac12\de_t((w^+)^2) - |\nabla (w^+)|^2 \dif x \dif t \le \frac12 \int_{D(t_0)} (w^+)^2 \dif x,
    \end{align*}
    and thus $w^+(\cdot,t_0) \equiv0$ for all $t>0$ and hence we conclude the proof.
\end{proof}

Once we know that temperatures of one-phase solutions are subcaloric, the usual $L^1$-decaying for the homogeneous Dirichlet problem follows, that is.

\begin{lemma}[$L^1_x$-decaying]\label{prop:L1-integrability}
    Let $D\subset \R^d$ be an open set, let $h\in L^\infty(D)$, with $|h(x)|\le 1$, and let ${g \in H^1_0(D)\cap L^1(D)}$.  
    Let $(u,\mu)$ be the enthalpy solution of the Stefan problem in $D_\infty=D\times[0,+\infty)$ with initial temperature $g$, $u \equiv 0$ on $\de_L D_\infty$, and with initial mushy coefficient $h$. 
    Then, $u(\cdot,t)\in L^\infty(D)$ for all $t>0$. Moreover, for all $0\le t_1 \le t_2 \le +\infty$ it holds
    \begin{equation}\label{e:mono-L1-alltimes}
        \| u(\cdot,t_2)^+ \|_{L^1(D)} \le \| u(\cdot,t_1)^+ \|_{L^1(D)}
            \qquad\text{and}\qquad
        \| u(\cdot,t_2)^- \|_{L^1(D)} \le \| u(\cdot,t_1)^- \|_{L^1(D)}.
    \end{equation}
    In particular, $u \in L^1(D_T)$ for all $T>0$.
\end{lemma}

\begin{proof}
    Since $u^+$ and $u^-$ are respectively subcaloric and supercaloric and vanishes on the lateral boundary, the estimate follows by comparison with the caloric functions in $\R^d$ and initial temperature $g^+$ and $g^-$ respectively.
\end{proof}

\subsection{Growth monotonicity of one-phase solutions}\label{sec:one-phase-subsol}
Here we investigate the monotonicity properties of one-phase solutions. 
These can be interpreted as the \emph{mushy region not draining energy from the system}.

\begin{lemma}\label{prop:one-phase:increasing-enthalpy}
    Let $D\subset\R^d$ be an open set, let $g\in H^1(D)$ be nonnegative, and $h\in L^\infty(D)$ with $|h(x)| \le 1$.
    If $(u,\mu)$ is a one-phase enthalpy solution of the Stefan problem with initial temperature $g$ and initial mushy coefficient $h$, then, for all $0 \le t_1 < t_2$ there exists $N = N(t_1,t_2)\subset D$ with $|N|=0$ such that
    \[
        \mu(x,t_2) \ge \mu(x,t_1) \qquad\text{for all}\quad x \in D\setminus N.
    \]
\end{lemma}

\begin{proof}
    The proof is a direct application of the Comparison Principle. Indeed, if $(\underline u,\underline \mu)$ is the one-phase stationary solution with initial temperature $\underline g \equiv0$ and initial mushy coefficient $\underline h = \mu(x,t_1)$, then for all $t\ge0$ they hold
    \[  
        \underline u \equiv 0 
        \qquad\text{and}\qquad
        \underline \mu(x,t) = \mu(x,t_1) \qquad\text{for almost every}\quad x\in D.
    \]
    Thus, by the Comparison principle, for all $t\ge0$, it holds
    \[
        \mu(x,t_2) \ge \mu(x,t_1) \qquad\text{for almost every}\quad x\in D,
    \]
    and so we conclude.
\end{proof}

The previous, however, does not involve a monotonicity of $\Omega_u^+(t)$, that is the scope of the following lemma.   This is well-known when the initial temperature $g \in C(D)$ is continuous. We give here a proof in the general case.

\begin{lemma}[Monotonicity of the positivity region]\label{prop:one-phase:positivity-increasing}
    Let $D\subset\R^d$ be an open set, let $g\in H^1(D)$ be nonnegative such that $g\not\equiv0$, and $h\in L^\infty(D)$ with $|h(x)| \le 1$.
    If $(u,\mu)$ is a one-phase enthalpy solution of the Stefan problem with initial temperature $g$ and initial mushy coefficient $h$, then, for all $0\le t_1 < t_2$ it holds
    \begin{equation}\label{e:one-phase:positivity-increasing}
        \big| \Omega_u^+(t_1) \setminus \Omega_u^+(t_2) \big|=0 \qquad\text{for all}\quad t\ge0.
    \end{equation}
\end{lemma}

\begin{proof}
    For all $t>0$, $u(\cdot,t)$ is continuous and thus $\Omega_u^+(t)$ is an open set; so classical results for regular caloric functions apply and the lemma follows. 
    On the other hand, for $t=0$, $g\in H^1(D)$, and therefore $\Omega_g^+$ is a \emph{quasi-open} set, and so we need to check that the heat equation satisfies the same functional properties also on these sets.
    Indeed, in what follows we recover them relying on the semigroup formulation of the heat equation (see \cite{ReedSimon1978:MMMPIV-operators,Brezis2011:FA} and the references therein), together with the functional properties on quasi-open sets (see \cite{BucurButtazzo2005:Shape-Optimization,Velichkov:BookShapeOptimization}).

    We schematically analyze the functional properties we need.
    \begin{itemize}
        \item Up to multiplying $g$ by a suitable cut-off, we can assume $\Omega_g^+$ to be bounded and $g\in H^1_0(\Omega_g^+)$ (the cut-off localizes the problem to $D\cap B_R$, for some $R>0$).
        \item For each \emph{quasi-connected} component $E\subset\Omega_g^+$ (see~\cite[Theorem~4.4.3]{Velichkov:BookShapeOptimization} for a definition), it holds that ${g\in H^1_0(E)}$ (see~\cite[Lemma 4.1.4]{BucurButtazzo2005:Shape-Optimization}). Therefore, we can assume $\Omega_g^+$ to be quasi-connected.
        \item Since $\Omega_g^+$ is bounded, the spectrum of $-\Delta$ is discrete. Moreover, since $\Omega_g^+$ is quasi-connected, the first eigenvalue is simple  and the corresponding eigenfunction is positive (\cite[Lemma 4.4.5]{Velichkov:BookShapeOptimization}), i.e.,
        \begin{equation}\label{e:simplicity-eigenvalue}
            \lambda_1(\Omega_g^+) < \lambda_2(\Omega_g^+) 
                \qquad\text{and}\qquad
            \Omega_g^+ = \{\phi_1 >0\},
        \end{equation}
        where $\lambda_i$ are the Dirichlet eigenvalues in $H^1_0(\Omega_g^+)$ and $\phi_1$ is the first eigenfunction.

        \item For all $\lambda<\lambda_1(\Omega_g^+)$, the resolvent $((-\Delta)-\lambda)^{-1}$ is positivity preserving, that is,
        \[
            ((-\Delta)-\lambda)^{-1}f\ge0 \qquad \text{for all non-negative}\quad f\in L^2(\Omega_g^+).
        \]
        Therefore, also the heat semi-group $(e^{-t\Delta})_{t>0}$ is positivity preserving \cite[page~204]{ReedSimon1978:MMMPIV-operators}.
        \item By \cite[Theorem XIII.44]{ReedSimon1978:MMMPIV-operators}, positivity preserving semi-groups, whose generators satisfy \eqref{e:simplicity-eigenvalue} is positivity improving, that is, for all nonnegative $f\in L^2(\Omega_g^+)$, $f\not\equiv 0 $, and for all $t>0$, it holds
        \begin{equation}\label{e:def:semigroup}
           v_f(x,t) :=  (e^{-t\Delta} f)(x) > 0 \qquad\text{for almost every}\quad x \in \Omega_g^+.
        \end{equation}
        \end{itemize}

    Thus, taking $g\in H^1_0(\Omega_g^+)$, where $\Omega_g^+$ is a quasi-connected bounded set, if $v_g$ is as in \eqref{e:def:semigroup}, then semi-group theory (see \cite[\S7]{Brezis2011:FA}) implies that $v_g \in C^1([0,+\infty);H^1_0(D))$. Thus, integrating by parts in time, it satisfies the following integral identity
    \[
        \int_0^T\int_{\Omega_g^+} v_g \de_t \eta - \nabla  v_g \cdot \nabla \eta \dif x\dif t=0,
    \]
    for all $T>0$ and $\eta \in H^1_0(\Omega_g^+\times(0,T))$.
    Let us now define $\widetilde v_g$ the extension to zero of $v_g$ in $D$.
    Then $\widetilde v_g$ a subsolution of the heat equation in $D\times[0,+\infty)$, that is,
    \[
        \iint \widetilde v_g \,\de_t \eta - \nabla  \widetilde v_g \cdot\nabla \eta \dif x \dif t \ge0 \qquad\text{for all}\quad \eta \in C^\infty_c(D\times(0,+\infty)).
    \]
    Lastly, let us define $ \widetilde \mu_g\in L^\infty(D\times(0,+\infty))$ as 
    \[
          \widetilde \mu_g(x,t) = 1 \qquad\text{if}\quad x\in \Omega_g^+
            \qquad\text{and}\qquad
          \widetilde \mu_g(x,t) = h(x) \qquad \text{otherwise}.
    \]
    The pair $(\widetilde v_g,\widetilde \mu_g)$ satisfies \eqref{e:compatibility_enthalpy-phases} and (since $\tilde\mu_g$ does not depend on time) 
    \[
        \iint \widetilde\mu_g\,\de_t \eta \dif x \dif t =0 \qquad \text{for all}\quad \eta\in C^\infty_c(D\times(0,+\infty)),
    \]
    and therefore $(\widetilde v_g,\widetilde \mu_g)$ is a subsolution of the Stefan problem.
    By construction (and \Cref{prop:one-phase:increasing-enthalpy}) it follows that
    \[
        \widetilde v_g(x,0) \le g(x) \qquad\text{and}\qquad \widetilde \mu_g(x,0) \le \mu(x,0) \qquad\text{for almost every}\quad x\in D,
    \]
    and therefore, the comparison principle (\Cref{lemma:comparison-stefan}) implies that, for all $t>0$, it holds
    \[
        0 < \widetilde v_g(x,t) \le u(x,t) \qquad \text{for almost every}\quad x \in \Omega_g^+,
    \]
    and this concludes the proof.
\end{proof}

\begin{remark}
    In the literature, there are results on the behavior of the enthalpy in the zero region. In particular, Berger and Rogers in \cite{RogersBerger1984:MonotonicityMushyRegionTwoPhaseStefan} proved that the \emph{mushy region} is non increasing. 
    Namely, let $(u,\mu)$ be an enthalpy solution of the Stefan problem in $D\times[0,+\infty)$, then the mushy region at time $t\ge0$, $\mathcal M(t) \subset D$ is defined as
    \[
        \mathcal M(t) := \big\{x \in D : |\mu(x,t)|< 1 \big\}.
    \]
    In \cite{RogersBerger1984:MonotonicityMushyRegionTwoPhaseStefan}, the authors proved that for all $0\le t_1 < t_2$ it holds
    \[
        \big|\mathcal M(t_2) \setminus \mathcal M(t_1)\big|=0.
    \]
    \Cref{prop:one-phase:positivity-increasing} on the other hand gives that 
    \[
        \big|\Omega_u^+(t_1) \setminus \Omega_u^+(t_2)\big|=0.
    \]
We notice that $\Omega_u^+(t)\cap\mathcal M(t)=\emptyset$ for every $t>0$. In particular, if $\Omega_u^+(t)\cup\mathcal M(t)=D$ for all $t\ge 0$, then the two results are equivalent.
On the other hand this may not happen since no conditions are imposed over $h(x)$ (except for $|h(x)|\le1$) nor over the geometry of $\Omega_u^+(t)$, so the presence of regions where
\[
    u \equiv 0 \qquad\text{and}\qquad \mu \equiv 1,
\]
is not (a priori) ruled out.
\end{remark}

\section{Total enthalpy monotonicity}\label{sec:monotonicity-formula}

In this section we prove \Cref{thm:total-monotonicity-formula}. We split the proof into two parts: in \Cref{subsec:mono:monotonicity} we prove the well-posedness and monotonicity of $\mathcal H(t)$, while in \Cref{subsec:mono:continuity} we prove its continuity.

\subsection{Proof of the monotonicity of the enthalpy}\label{subsec:mono:monotonicity}

We prove the first half of \Cref{thm:total-monotonicity-formula} relying on an energy conservation law for one-phase entire solutions. Before doing so, we observe that, by \Cref{prop:one-phase:increasing-enthalpy},  for all one-phase solutions $(u,\mu)$ it holds $\mu(x,t) \ge h(x)$, and thus
\[
    \mathcal H(t) = \mathcal H_0(t) = \int_{D} \big(u(x,t) + (\mu(x,t) - h(x)) \big) \dif x.
\]
Therefore, we can define the signed enthalpy $\mathcal H_0(t)$ with respect to all $\widetilde h\in L^\infty(D)$ such that $h_0(x) \le h(x)$ for almost every $x\in D$.

\begin{lemma}\label{lemma:energy-conservation}
    Let $g\in H^1(\R^d)$ be nonnegative and let ${h\in L^\infty(D)}$ satisfies $|h(x)|\le 1$.
    Let the pair $(u,\mu)$ be the corresponding one-phase enthalpy solution in ${\R^d\times(0,+\infty)}$ with initial temperature $g$ and initial mushy coefficient $h$.
    Let $\widetilde h \in L^\infty(\R^d)$ satisfies $\widetilde h(x) \le h(x)$ for almost every $x\in D$.
    If the signed initial total enthalpy with respect to $\widetilde h$ is finite, namely, 
    \[
        \widetilde{\mathcal H}_0(0):= \int_{\R^d} \big(g(x) + \big(h(x)-\widetilde h(x)\big)\big) < +\infty,
    \]
    then 
    $$\widetilde{\mathcal H}_0(t):=  \int_{\R^d} \big(u(x,t) + (\mu(x,t) - \widetilde h(x)) \big) \dif x,$$ 
is constant in time.
\end{lemma}

\begin{proof}[Proof of \Cref{lemma:energy-conservation}]
By \Cref{prop:one-phase:increasing-enthalpy}, for all $t>0$ it holds
\[
    \mu(x,t) \ge h(x) \ge \widetilde h(x) \qquad \text{for almost every}\quad x\in D.
\]
Therefore the enthalpy $\widetilde{\mathcal H}_0(t)$ is well defined for all $t\ge0$, although it may a priori be infinite.
In order to show that $\widetilde{\mathcal H}_0 (t)$ is well-posed and finite, we directly show that, for all $t>0$, it holds $\widetilde{\mathcal H}_0(t) = \widetilde{\mathcal H}_0(0)$.

Let us define $\rho_R : \R^d \to [0,1]$ as $\rho_R(x):= \psi (\sfrac{|x|}{R})$ where $\psi \in C^\infty_c(\R)$ is such that
\[
\psi(\zeta) = 1 
    \qquad\text{if}\quad \zeta\le \frac12, \qquad
    \psi(\zeta) = 0 \quad \text{if}\quad\zeta \ge1,
    \qquad\text{and}\qquad
    \|\psi'\|_{L^\infty(\R)} \le 4.
\]
Then, for all $R>0$ and $t>0$, it holds
\begin{equation}\label{e:identity-Rd}
    \int_{\R^d}\big( u + (\mu-\widetilde h)\big)\rho_R \dif x \bigg|_{\tau=0}^{t} 
        = -\int_{0}^t\int_{\R^d} \nabla u \cdot \nabla \rho_R \dif x\dif \tau
        = \int_{0}^t \int_{\R^d} u\,\Delta \rho_R \dif x \dif \tau,
\end{equation}
where the $\widetilde h$ terms cancel out since both $\rho_R$ and $\widetilde h$ are time-independent.
By \Cref{prop:L1-integrability}, we know that $u \in L^1(\R^d\times(0,T))$ for all $T\ge0$; therefore, since 
\[
    |\Delta \rho_R| \le \frac{C_d}{R^2} \ind_{B_{R}\setminus B_{\sfrac R2}} \xrightarrow[R\to+\infty]{}0 \qquad \text{pointwise in}\quad \R^d,
\]
the Dominated Convergence Theorem applies and we get
\[
    \lim_{R\to+\infty}\int^t_0\int_{\R^d} u \,\Delta \rho_R \dif x \dif \tau =0.
\]

It remains to show that the integrals on the left-hand side of \eqref{e:identity-Rd} converge to
$(\widetilde{\mathcal H}_0(t) - \widetilde{\mathcal H}_0(0))$.
However, $\rho_R \nearrow 1$ and for all $t\ge0$, it holds
\[
    \big(u(x,t) + (\mu(x,t) - \widetilde h(x) \big) \ge0 \qquad \text{for almost every}\quad x\in D,
\]
thus the Monotone Convergence Theorem implies that 
\[
    \widetilde{\mathcal H}_0(t) = \lim_{R\to+\infty} \int_{\R^d}\big( u(x,t) + (\mu(x,t)-\widetilde h(x))\big)\rho_R \dif x.
\]
Finally, since $\widetilde{\mathcal H}_0(0) < \infty$, we get
\[
    \widetilde{\mathcal H}_0(t) - \widetilde{\mathcal H}_0(0) = \lim_{R\to+\infty} \int_{D} \big(u + (\mu-\widetilde h)\big)\rho_R \dif x\bigg|_{\tau=0}^t = \lim_{R\to+\infty} \int^t_0\int_{\R^d} u \,\Delta \rho_R \dif x \dif \tau =0,
\]
and so we conclude the proof.
\end{proof}

We can now prove the statement in the general case.

\begin{proof}[Proof of the Monotonicity in \Cref{thm:total-monotonicity-formula}]
    By \eqref{e:compatibility_enthalpy-phases}, for all $t\ge0$ and for almost every $x \in D$, it holds
    \[
     u(x,t) +\mu(x,t)-h(x) >0, 
    \]
    if and only if
    \[
        u(x,t) >0 \qquad \text{or}\qquad u(x,t)=0 \qquad   \text{and}\qquad \mu(x,t)>h(x).
    \]
    This occurs because, outside from the zero set, the mushy coefficient is completely determined by the temperature.
    Therefore, in what follows, we will prove \Cref{thm:total-monotonicity-formula} by showing that the quantities $\mathcal H^+(t)$ and $\mathcal H^-(t)$, defined as
    \[
        \mathcal H^+(t) := \int_{D(t)} \big(u^+ + (\mu-h)^+ \big) \dif x
            \qquad\text{and}\qquad
        \mathcal H^-(t) := \int_{D(t)} \big(u^- + (\mu-h)^- \big) \dif x,
    \]
    are finite and (respectively) monotonically decreasing and increasing in time.
    We  notice that the monotonicity of $\mathcal H^+$ and $\mathcal H^-$ implies the monotonicity of $\mathcal H$ since (recall that, within our convention,  $u^-\le0$),
    \[
        \mathcal H(t) = \mathcal H^+(t) - \mathcal H^-(t).
    \]

    First, we prove that $\mathcal H^+(t)$ is well defined and finite for all $t>0$.
    Let $(\overline u,\overline \mu)$ be the one-phase solution in $\R^d\times[0,+\infty)$ with initial temperature $\overline g (x)
    := g^+(x)$ and initial mushy coefficient $\overline h(x) := h(x)$.
    By the Comparison Principle, for all $t>0$, they hold
    \[
        u(x,t) \le \overline u(x,t),
            \qquad
        h(x) \le \overline \mu(x,t)
            \qquad \text{and}\qquad 
        \mu(x,t) \le \overline \mu(x,t) \qquad\text{for almost every}\quad x \in D,
    \]
    and therefore, \Cref{lemma:energy-conservation} implies that
    \[
        \int_{D(t)}\big( \overline{u} + (\overline \mu - h) \big) \dif x = \mathcal H_0(\overline u,\overline\mu,h;t) 
        = \mathcal H_0(\overline u,\overline\mu,h;0)
        = \| g^+\|_{L^1(D)}
        = \mathcal H^+(0).
    \]
    Finally, by the Monotone Convergence Theorem, for $\rho_R$ defined as in the proof of \Cref{lemma:energy-conservation}, we get
    \begin{align*}
    \mathcal H^+(t) 
        = \lim_{R\to+\infty} \int_{D}\big(u^+ + (\mu - h)^+\big)\rho_R\dif x 
        \le \lim_{R\to+\infty} \int_{\R^d}\big(\overline u + (\overline \mu - h)\big)\rho_R\dif x 
        = {\mathcal H}(\overline u,\overline \mu, h;t) 
        = \mathcal H^+(0).
    \end{align*}
    
   Similarly, we define $(\underline u,\underline \mu)$ as the solution in $\R^d$ with initial temperature $\underline g := g^-$ and initial mushy coefficient $\underline h := h$, and so we find that
    \[ 
        \mathcal H^-(t) \ge \mathcal H^-(0),
    \]
    and thus, for all $t\ge0$, $\mathcal H^+(t)$ and $\mathcal H^-(t)$ are finite and it holds
    \[
        \mathcal H(t)=\mathcal H^+(t) - \mathcal H^-(t)\le \mathcal H^+(0) - \mathcal H^-(0) = \mathcal  H(0) <+\infty.
    \]

    For general $0\le t_1 < t_2$ we argue in the same way. Let us take $(\overline u,\overline \mu)$, the enthalpy solution in $\R^{d+1}_+$ with initial temperature $\overline g$ and initial mushy coefficient $\overline{h}$ defined as follows
    \[
        \overline g (x) = u(x,t_1)
            \qquad\text{and}\qquad
        \overline h(x) = h(x) \vee \mu(x,t_1).
    \]
    As in the previous case, by the Comparison Principle it follows that
    \[
        u(x,t) \le \overline u(x,t-t_1), \qquad
        h(x) \le \overline\mu(x,t-t_1) \qquad\text{and}\qquad
        \mu(x,t) \le \overline \mu(x,t-t_1),
    \]
    and so we conclude that for all $t_2 > t_1$ it holds
    \[  
        \mathcal H^+(t_2) \le \widetilde{\mathcal H}_0 (\overline u,\overline\mu,h;t_2-t_1) 
        = \widetilde{ \mathcal H}_0 (\overline u,\overline\mu,h;0) = \mathcal H^+(t_1).
    \]
    The same reasoning applies to $\mathcal H^-(t)$, and thus, by the same steps (taking now $\underline h = h \wedge \mu(\cdot,t_1)$), we prove that, for all $t_2>t_1$, it holds
    \[
        \mathcal H^-(t_2) \ge \mathcal H^-(t_1),
    \]
    and thus
    \[
        \mathcal H(t_2)=\mathcal H^+(t_2) - \mathcal H^-(t_2)\le \mathcal H^+(t_1) - \mathcal H^-(t_1)= \mathcal H(t_1),
    \]
    which concludes the proof of part one of \Cref{thm:total-monotonicity-formula}.
\end{proof}

Once the integrability of $\mathcal H(t)$ is assured, the proof of \Cref{thm:energy-conservation-law} follows as \Cref{lemma:energy-conservation}.

\begin{proof}[Proof of \Cref{thm:energy-conservation-law}]
    By \Cref{thm:total-monotonicity-formula}, $\mathcal H^+(t)$ and $\mathcal H^-(t)$ are both well-defined and finite, therefore for all $t\ge0$ it holds
    \[
        \mathcal H_0(t) = \lim_{R\to+\infty} \int_{D(t)} \big( u(x,t) + (\mu(x,t) - h(x)) \big)\rho_R \dif x,
    \]
    where $\rho_R$ is the same function defined in the proof of \Cref{lemma:energy-conservation}. 
    Then, by the Dominated Convergence Theorem it follows that
    \[
    \mathcal H_0(t) - \mathcal H_0(0) = 0 \qquad\text{for all}\quad t\ge0,
    \]
    and so we conclude the proof.
\end{proof}

\subsection{Continuity of the Total Enthalpy}\label{subsec:mono:continuity}

In this section we complete the proof of \Cref{thm:total-monotonicity-formula} showing that $u$ and $\mu$ are separately $L^1(D)$-strongly continuous.

\begin{lemma}\label{lemma:mono:continuity}
    Let $D\subset \R^d$ be an open set. Let $g \in H^1_0(D)\cap L^1(D)$ and $h \in L^1(D)$ which satisfies $|h(x)|\le 1$. 
    Let $(u,\mu)$ be the corresponding enthalpy solution in $D_\infty$ with initial temperature $g$, homogeneous Dirichlet boundary conditions, and initial mushy coefficient $h$.
    Then $u(\cdot,t)$ and $\mu(\cdot,t)$ are (separately) $L^1(D)$-strongly continuous, that is, for all $t_0\ge0$ it holds
    \[
        \lim_{t\to t_0} \| u(\cdot,t) - u(\cdot,t_0) \|_{L^1(D)} + 
        \|\mu(\cdot,t) - \mu(\cdot,t_0) \|_{L^1(D)} =0.
    \]
\end{lemma}

\begin{remark}
\Cref{lemma:mono:continuity} implies the continuity of the total enthalpy $\mathcal H(t)$ and conclude the proof of \Cref{thm:total-monotonicity-formula}, since 
\begin{equation}\label{e:continutity-mu-with-h}
    \left\| \mu(\cdot,t) -  \mu(\cdot,t_0)\right\|_{L^1(D)} = 
    \left\| \big(\mu(\cdot,t) - h \big) -  \big( \mu(\cdot,t_0)-h\big) \right\|_{L^1(D)},
\end{equation}
and the $L^1(D)$-strong continuity in time implies that
\[
    \lim_{t\to t_0} \mathcal H(t) 
    = \lim_{t\to t_0} \int_{D(t)} \big|u + \big(\mu-h \big)\big| \dif x
    = \int_{D(t_0)} \big|u + \big(\mu-h \big)\big| \dif x
    = \mathcal H(t_0).
\]
\end{remark}

\begin{proof}[Proof of \Cref{lemma:mono:continuity}]
    We treat the continuity of the temperature and of the mushy coefficient separately.

    \paragraph{Continuity of the temperatures}
    Let $0\le t_0 < T $. Since $u \in H^1(D_T)$, there exists a sequence $\{t_i\}_{i\in\NN}$ of times such that $t_i\to t_0$ and for which
    \[
        u(x,t_i) \xrightarrow[i\to+\infty]{} u(x,t_0) \qquad\text{for almost every}\quad x\in D.
    \]
    Let us then define $\underline t := \inf\{t_i: i\in\NN\}$, $v_0(x) := |u(x,\underline t)|$, and the function $v : \R^d \times[\underline t,+\infty) \to \R$ as
    \[
        v(x,t) := (K_{t-\underline t} \ast v_0)(x),
    \]
    where $K_{t-\underline t}$ is the usual heat kernel in $\R^d$ at time $(t-\underline t)\ge0$.
    Since $u^+$ is subcaloric and $u^-$ is supercaloric, for all $t\ge\underline t$ there exists $N_t\subset D$, $|D\setminus N_t|=0$ such that
    \[  
        |u(x,t)| \le v(x,t) \qquad\text{for all}\quad x \in N_t.
    \]
    Since $v \in C^0([\underline t,+\infty);L^1(\R^d))$ (see, for instance, \cite{Brezis2011:FA}), $v(\cdot,t_i)$ converge to $v(\cdot,t_0)$ strongly in $L^1(\R^d)$ and, up to subsequences, 
    \[
        v(x,t_i) \xrightarrow[i\to+\infty]{} v(x,t) \qquad\text{for almost every}\quad x \in \R^d.
    \]
    Thus, by Fatou's Lemma applied to $(v(\cdot,t_i) + v(\cdot,t_0) - |u(\cdot,t_i)-u(\cdot,t_0)|)$ it follows that
    \[  
        \int_D 2 v(x,t_0)  \dif x \le \liminf_{i\to+\infty} \int_D v(x,t_i) + v(x,t_0) - |u(x,t_i)- u(x,t_0)| \dif x,
    \]
    and thus, since $v(\cdot,t)$ is $L^1$-strongly continuous, we get
    \[ 
        \limsup_{i\to+\infty} \int_D |u(x,t_0)- u(x,t_i)| \dif x
        \le 0.
    \]
    To conclude, we observe that the limit is independent from the subsequence choice; therefore we get that
    \[
        \lim_{t\to t_0} \| u(\cdot,t) - u(\cdot,t_0) \|_{L^1(D)} =0,
    \]
    without passing to subsequences, and thus the temperature is $L^1(D)$-strongly continuous.
    \qed

\paragraph{Continuity of the mushy coefficient}
We show the continuity in time for $(\mu(\cdot,t) -h) \in L^1(D)$, as it is an equivalent property due to \eqref{e:continutity-mu-with-h}.
First we treat the one-phase case, then we use the result for one-phase solutions together with the results in \cite{RogersBerger1984:MonotonicityMushyRegionTwoPhaseStefan} to conclude the proof in the two-phase case.\medskip

\noindent{\it Continuity of the mushy coefficient in the one-phase case.} Let $(u,\mu)$ be a one-phase solution.
By \Cref{prop:one-phase:increasing-enthalpy}, the mushy coefficient $\mu$ is monotone non-decreasing in time and $(\mu(x,t) - h(x))\ge0$ for almost every $x\in D$; therefore, for all $t_0\ge0$, there exist left and right limits, such that
\begin{equation}\label{e:mu-continuity-left-right}
    0 \le \lim_{t\nearrow t_0} \mu(x,t)- h(x) \le \mu(x,t_0) - h(x) \le \lim_{t\searrow t_0} \mu(x,t) - h(x) \qquad\text{for almost every}\quad x\in D.
\end{equation}
Moreover, as the mushy coefficient is non-decreasing, up to subsequences $(\mu(x,t_0+1) + h(x))$ dominates both sequences, and therefore these limits are $L^1(D)$-strong. As the mushy coefficient $\mu(\cdot,t)$ is weak-$\ast$ continuous in $L^\infty(D)$ in time (see \Cref{def:enthalpy_solution}), we get 
\begin{equation}\label{e:mono:limit-mushy-one-phase}
    \lim_{t\nearrow t_0} \mu(x,t)- h(x) 
        = \mu(x,t_0) - h(x) = \lim_{t\searrow t_0} \mu(x,t) - h(x).
\end{equation}
As the limits do not depend on the subsequence, we recover the full convergence as $t\to t_0$. This completes the proof of \Cref{lemma:mono:continuity} in the one-phase case.

\begin{remark}
    The inequalities \eqref{e:mu-continuity-left-right} follow from \Cref{prop:one-phase:increasing-enthalpy}.
    Such inequalities hold for all times $t_1 < t_2$, but only outside a set $N=N(t_1,t_2)\subset D$, with $|D\setminus N|=0$. 
    As the family of sets $N$ depends on the choice of times, the Dominated Convergence Theorem and yet \eqref{e:mono:limit-mushy-one-phase} holds only up to subsequences. The argument then conclude thanks to the full weak-$\ast$ continuity of $\mu$ in time.
\end{remark}

\noindent{\it Continuity of the mushy coefficient in the two-phase case.} Let $(u,\mu)$ be a two-phase solution and let us fix $t_0\ge0$. In order to prove the continuity of $\mu(\cdot,t)$ at $t_0\ge0$, we need to understand the interaction between the positive and negative phases.
For all $t\ge0$, we define the sets $A(t)$, $B(t)$ as
\[
    A(t) := \{ x \in D : \mu(x,t) = 1\} \qquad\text{and}\qquad B(t) := \{ x \in D: \mu(x,t) = -1\},
\]
and the set $E(t)\subset D$ as
\[
    E(t) := A(t) \cup B(t) = D \setminus \mathcal M(t).
\]
Berger and Rogers' results \cite{RogersBerger1984:MonotonicityMushyRegionTwoPhaseStefan} imply that $E(t)$ is monotonically non-decreasing, namely,
\[
    |E(t_1) \setminus E(t_2)| =0 \qquad\text{for all}\quad 0\le t_1 < t_2,
\]
and the mushy coefficient is constant in $\mathcal M(t)$, that is,
\[
    \mu(x,t) = h(x) \qquad\text{for almost every}\quad x \in \mathcal M(t) = D \setminus E(t).
\]
Therefore, we can rewrite $\mu(x,t) - h(x)$ as follows
\begin{equation}\label{e:mushy-decomposition}
    \begin{split}
        \mu(x,t) - h(x) 
            &= \ind_{A(t)}(x) - \ind_{B(t)}(x) - h(x) \ind_{E(t)}\\
            &= \big(1-h(x) \big) \ind_{A(t)}(x)
                -\big(1+h(x)\big) \ind_{B(t)}(x).
    \end{split}
\end{equation}
Through standard algebraic manipulations, for all $t\ge0$ the following identities hold
    \begin{align*}
        2 \ind_{A(t)} &= (1-h) \ind_{A(t)} - (1+h)\ind_{B(t)} + (1+h) \ind_{E(t)} = (\mu-h) + (1+h) \ind_{E(t)},\\
        \intertext{and}
        2 \ind_{B(t) } &= (1+h)\ind_{B(t)} - (1-h) \ind_{A(t)} + (1+h) \ind_{E(t)} = (h-\mu) + (1-h) \ind_{E(t)}.
    \end{align*}
As in the one-phase case, the monotonicity of $E(t)$ implies that $(1\pm h)\ind_{E(t)}$ is strongly continuous in $L^1(D)$ in time.
Therefore, the linearity of the weak-$\ast$ convergence implies that
\[
    \ind_{A(t)} \xrightharpoonup[t \to t_0]{} \ind_{A(t_0)} 
        \qquad \text{and}\qquad
    \ind_{B(t)} \xrightharpoonup[t\to t_0]{} \ind_{B(t_0)}\qquad \text{weakly-$\ast$ in}\quad L^\infty(D) \qquad\text{for all}\quad t_0\ge0.
\]
Thus, the weak-$\ast$ limits of these sequences of characteristic functions are characteristic functions. This implies that, for all $R>0$, they converge strongly in $L^2(D\cap B_R)$, and (up to subsequences) almost everywhere in $D$. 
The same holds for $(1-h)\ind_{A(t)}$ or $(1+h)\ind_{B(t)}$, as $h$ does not depend on time.
Therefore, up to subsequences, for all $t_0\ge0$, it holds
\[
    \mu(x,t) - h(x) \xrightarrow[t\to t_0]{} \mu(x,t_0) - h(x) \qquad\text{for almost every}\quad x\in D.
\]

To show that the convergence is actually $L^1(D)$-strong, we need uniform barriers. Once again we introduce the pairs $(\overline u, \overline \mu)$ and $(\underline u, \underline \mu)$ as the enthalpy solutions in $\R^{d+1}_+$ with the initial data $(g^+,h)$ and $(g^-,h)$, respectively.
Then, the comparison principle and \Cref{prop:one-phase:increasing-enthalpy} imply that, for all $t \le t_0+1$, the following holds
\[
\underline \mu(x,t_0+1) \,\le\, \underline \mu(x,t)\, \le\, \mu(x,t) \le\, \overline{\mu}(x,t) \,\le\, \overline{\mu}(x,t_0+1) \qquad\text{for almost every}\quad x \in D,
\]
and so, since $(\overline \mu(\cdot,t)-h) \in L^1(D)$ and $(\underline \mu(\cdot,t)-h) \in L^1(D)$ for all $t\ge0$, the Dominated Convergence Theorem implies that
\[
    \lim_{t\to t_0} \left\|  \mu(\cdot,t) - \mu(\cdot,t_0) \right\|_{L^1(D)} 
    = \lim_{t\to t_0}
        \left\| \big( \mu(\cdot,t) - h \big) - \big( \mu(\cdot,t_0) - h \big) \right\|_{L^1(D)} 
    = 0,
\]
and thus we conclude the proof of \Cref{lemma:mono:continuity} also in the general two-phase case.
\end{proof}

\section{Asymptotic behavior of solutions with finite initial enthalpy}\label{sec:asymptotic-analysis}

In this section we prove \Cref{thm:finite-melting}, so in all this section we always assume that there exists $\ell \in (0,1)$ such that the initial mushy coefficient satisfies
\begin{equation}\label{e:phase-separation-section}
    |h(x)|\le \ell <1 \qquad\text{for all} \quad D\setminus \big(\Omega_g^+ \cup \Omega_g^-\big),
\end{equation}
where $g$ is the initial temperature. 
Through this section, whenever it leads to no confusion, we will often write $|h(x)|\le \ell<1$ in $D$ instead of $D\setminus (\Omega_g^+ \cup \Omega_g^-)$.

The first result we need in our analysis is the measure estimate we already presented in the introduction. 
We notice that estimates in the spirit of \eqref{e:measure-estimate-intro} can be found also in \cite{DiazFasanoMeirmanov90:MushyEvolution} in the context of the one-phase problem.

\begin{lemma}[Measure bounds]\label{lemma:measure-est}
    Let $D\subset\R^d$ be an open set and $\ell\in(0,1)$. Let $g\in H^1_0(D)\cap L^1(D)$ and let $h \in L^\infty(D)$ satisfies \eqref{e:phase-separation-section}.
    Let the pair $(u,\mu)$ be the enthalpy solution of the Stefan problem in $D_\infty = D \times[0,+\infty)$ with initial temperature $g$, homogeneous Dirichlet boundary conditions, and with initial enthalpy coefficient $h$.
    Then they hold the following measure estimates
    \begin{equation}\label{e:measure-est-gen-h}
    \left\vert\Omega_u^+(t)\setminus \Omega_g^+  \right\vert +     \left\vert\Omega_u^-(t)\setminus \Omega_g^-  \right\vert 
    \le 
        \frac{1}{1-\ell}
        \| g\|_{L^1(D)},
    \end{equation}
    and
    \begin{equation}\label{e:measure-est-mixing}
        \left\vert \Omega_u^+(t) \cap \Omega_g^-\right\vert + 
    \left\vert \Omega_u^-(t) \cap \Omega_g^+\right\vert \le \frac{1}{2} \| g\|_{L^1(D)}.
    \end{equation}
\end{lemma}

\begin{proof}
    The Lemma is a consequence of \Cref{thm:total-monotonicity-formula}.
    By \eqref{e:compatibility_enthalpy-phases}, \eqref{e:mushy-decomposition}, and \eqref{e:phase-separation-section}, for almost every $x\in D$, they hold the following inequalities
    \[
        (1-\ell)\big( \ind_{\Omega_u^+(t)} - \ind_{\Omega_g^+}\big)^+ \le \big(\mu(\cdot,t)-h\big)^+ 
        \qquad \text{and}\qquad
        (1-\ell)\big( \ind_{\Omega_u^-(t)} - \ind_{\Omega_g^-}\big)^- \ge \big(\mu(\cdot,t)-h\big)^-.
    \]
    Thus, \eqref{e:measure-est-gen-h} follows integrating these inequalities over $D$.
    The inequality \eqref{e:measure-est-mixing} follows by analogous considerations. We observe that, for all $t\ge0$ and for almost every $x\in D$, they hold
    \[
        2 \ind_{\Omega_u^+(t) \cap \Omega_g^-} \le \big(\mu(\cdot,t)-h\big)^+ 
            \qquad \text{and}\qquad
        2 \ind_{\Omega_u^-(t) \cap \Omega_g^+} \ge \big(\mu(\cdot,t)-h\big)^-,
    \]
    and thus we conclude the proof by integrating the latter inequality over $D$.
\end{proof}

As we pointed out in the introduction, the latter is not sufficient to conclude, as the sole measure estimate does not give information about where the phases are expanding.
To provide such information, we need some density estimates, that are the content of the next lemma.
As we already discussed, such lemma holds for one-phase solutions.

\begin{lemma}[Fixed scale density estimates]\label{lemma:nofiltration}
    Let $D\subset\R^d$ be an open set and $\ell\in (0,1)$.
    Let $(u,\mu)$ be a one-phase enthalpy solution with homogeneous Dirichlet boundary condition, initial temperature $g\in H^1_0(\R^d)\cap L^1(\R^d)$, and initial mushy coefficient $h \in L^\infty(\R^d)$ that satisfies \eqref{e:phase-separation-section}.
    There exists exists a monotone increasing function 
    \[
        c_d : (0,+\infty) \longrightarrow (0,+\infty),
    \]
    depending only on the dimension $d$ and continuously on the $L^1$-norm of $g$ with the following property. 
    Let $R_0>0$ be a positive constant and assume that $\Omega_g^+$ satisfies the two-sided ball condition with radius $R_0$.
    Then, for all $r_0 \in (0,R_0/10)$ and for all $(x_0,t_0) \in \Gamma^+$  such that
    \begin{equation}\label{e:asymptotic-first-time}
        x_0 \in \Gamma_u^+(t_0) \cap \de (\Omega_g^+ + B_{r_0}) \qquad\text{and}\qquad
        t_0 := \min\Big\{ t>0 : \Gamma^+_u(t) \cap \de (\Omega_g^+ + B_{r_0}) \ne\emptyset \Big\},
    \end{equation}
    the following density estimate holds
    \begin{equation}\label{e:nofiltration-small}
        \big|\Omega_u^+(t_0) \cap B_{r_0}(x_0)\big| \ge c_d(r_0) (1-\ell)^{\sfrac d2}.
    \end{equation}
\end{lemma}

We discuss and prove \Cref{lemma:nofiltration} in \Cref{subsec:density-estimates}. Here we only observe that the aforementioned time $t_0>0$ exists. Indeed $\Gamma^+ \cap \de ((\Omega_g^+ + B_{r_0})\times(0,+\infty)) \subset \R^d\times(0,+\infty)$ is closed, since the temperatures of enthalpy solutions are continuous \cite{DiBenedetto1982:Stefan1,CaffarelliEvans1983:ContinuityTwoPhaseStefan,Ziemer1982:Stefan,BaroniKuusiUrbano:TwoPhaseStefan}; moreover in \Cref{prop:finite-speed-inner} we prove that the free boundary moves with finite speed and therefore such $t_0>0$ exists.

Once we have proved the density estimates from \Cref{lemma:nofiltration}, the proof of \Cref{thm:finite-melting} follows.

\begin{proof}[Proof of \Cref{thm:finite-melting}]
    We prove the statement for $\Omega_u^+(t)$, the analogous result for negative phase $\Omega_u^-(t)$ follows accordingly.
    
We fix $R>0$ such that $\Gamma^+ \cap ( (\Omega_g^+ + B_{2R})\times(0,+\infty))\ne\emptyset$, and  
    we define the set 
    \begin{equation}\label{e:Omega-enlarged}
        \Omega_R^+ := \R^d \setminus \big(\big( \R^d \setminus \big( \Omega_g^+ + B_{2R} \big) \big) + B_R \big),
    \end{equation}
    so that $\Omega_g^+ \subset \Omega_R^+ \subset \Omega_g^+ + B_R$ and it satisfies the two-sided ball condition with radius $R>0$.
    Let us now consider the one-phase enthalpy solution $(w_\eps,\mu_\eps)$ in $\R^{d+1}_+$ with initial temperature $g_\eps$ defined as
    \[
        g_\eps(x) := g^+ + \big(\eps\,\wedge \,\distop(x,\R^d\setminus \Omega_R^+)\big) e^{-|x|^2},
    \]
    so that $\Omega_R^+ = \Omega_{g_\eps}^+$ and $\|g_\eps\|_{L^1} \le \|g^+\|_{L^1}+ C_d\,\eps$, for some $C_d>0$.
    Let us also fix the initial mushy coefficient $h_\eps$ defined as
    \[
        h_\eps(x) = 1 \qquad \text{in}\quad \Omega_R^+ \qquad\text{and}\qquad h_\eps(x) = 1-\ell \qquad \text{in} \quad \R^d \setminus \Omega_R^+,
    \]
    so that it satisfies \eqref{e:phase-separation-section}. 
    
    The comparison principle implies that, for all $t\ge0$ and $\eps>0$, it holds
    \[
        u(x,t) \le w_\eps(x,t) \qquad \text{for almost every}\quad x\in D,
    \]
    and therefore, for all $r<R/10$ there exists $(x_{0,r},t_{0,r})\in \Gamma^+_{w_\eps}$ such that  $x_{0,r} \in \Gamma^+_{w_\eps}(t_0) \cap (\Omega_g^+ + B_r)$ and
    \[
    t_{0,r} := \min\big\{t>0 :\Gamma^+_{w_\eps} \cap \big((\Omega_g^+ + B_r)\times(0,+\infty)\big)\ne\emptyset  \big\}.
    \]
    Thus \Cref{lemma:nofiltration} applies to $(w_\eps,\mu_\eps)$ at $(x_{0,r},t_{0,r})$, and so
    \[
        \big|\Omega_{w_\eps}^+(t) \cap B_r(x_{0,r}) \big| \ge c(r)(1-\ell)^{\sfrac d2},
    \]
    Combining the latter estimates with the measure estimates provided by \Cref{lemma:measure-est}, we get
    \[
    c(r)(1-\ell)^{\sfrac d2} \le \big|\Omega_{w_\eps}^+(t) \cap B_r(x_{0,r}) \big| \le \big|\Omega_{w_\eps}^+(t) \setminus \Omega_R^+ \big| \le \|g^+\|_{L^1(D)},
    \]
    and since $c(r)$ is positive, strictly increasing and continuous, taking $\eps\to0$ we get that
    \[
    r \le c^{-1}\left(\frac{\|g^+\|_{L^1(D)}}{(1-\ell)^{\sfrac d2}}\right) =: \overline R_g \qquad\text{for all}\quad r\in(0,R/10).
    \]
    Thus, taking $R=11\overline R_g$, by the comparison principle we get that
    \[
        \Omega_u^+(t) \subset \Omega_{w_\eps}^+(t) \subset \Omega_{g}^+ + B_{R_g^+} \qquad\text{for all}\quad t\ge0,
    \]
    where $R_g^+ := 12 \overline R_g$.
    Through the same argument we prove that there exists $R_g^->0$ such that
    \[
        \Omega_u^-(t) \subset\Omega_g^+ + B_{R_g^-} \qquad \text{for all}\quad t\ge0,
    \]
    and hence we conclude the proof of \Cref{thm:finite-melting} by taking $R_g := (R_g^+ \vee R_g^-)$.
\end{proof}

The proof of \Cref{thm:finite-melting-one-phase} follows by similar ideas.

\begin{proof}[Proof of \Cref{thm:finite-melting-one-phase}]
    By \Cref{prop:one-phase:positivity-increasing}, the positivity sets $\Omega_u^+(t)$ are monotone in time and therefore  $\mu(x,t) = h(x)$ in $\Interior(D\setminus \Omega_u^+(t))$ for all $t\ge0$.
    Therefore, as for the proof of \Cref{thm:finite-melting}, for all $t\ge0$, $R>0$ and $\eps>0$ we can define $\Omega_R^+(t)$, $g_{t,\eps}$, $h_{t}$ respectively as
    \begin{align*}
        \Omega_R^+(t) &:= \R^d \setminus \big(\big( \R^d \setminus \big( \Omega_u^+(t) + B_{2R} \big) \big) + B_R \big),
    \end{align*}
    \begin{align*}
        g_{t,\eps} &:= u(t) + \big( \eps \wedge \distop(x,\R^d\setminus \Omega_R^+(t) \big)e^{-|x|^2}, 
    \end{align*}
    \begin{align*}
        h_t(x) &= 1 \qquad\text{in}\quad \Omega_R^+(t) \qquad \text{and}\qquad h_{t}(x) = 1-\ell \qquad \text{in}\quad \R^d\setminus\Omega_R^+(t).
    \end{align*}
    In particular it holds that
    \[
    g_{t,\eps}(x) \ge u(x,t) \qquad\text{and} \qquad h_{t}(x) \ge \mu(x,t) \qquad\text{for almost every}\quad x \in D.
    \]
    As in \Cref{lemma:measure-est}, by the monotonicity of $\Omega_u^+(t)$ and of the enthalpy (\Cref{thm:total-monotonicity-formula}), we get that 
    \[  
    \big|\Omega_{u}^+(t+s) \setminus \Omega_u^+(t) \big| 
    \le \frac{1}{1-\ell} \left(\|u(t) \|_{L^1(D)} - \|u(t+s) \|_{L^1(D)}\right)
    \le
    \frac{1}{1-\ell} \|u(t) - u(t+s) \|_{L^1(D)},
    \]
    for all $s>0$. Thus, by \Cref{lemma:nofiltration} we conclude as in the proof of \Cref{thm:finite-melting}.
\end{proof}

\begin{proof}[Proof of \Cref{cor:asymptotic-regular}]
    Let us fix $\eps>0$, and we define $g_\delta \in H^1(\R^d)$ as
    \[
        g_\delta(x) := \big(\delta \wedge \distop(x,\R^d\setminus \Omega)  \big) \frac{e^{-\frac{|x|^2}{2}}}{(2\pi)^d} , 
    \]
    so that $\Omega_{g_\delta}^+ = \Omega$ and $\| g_\delta \|_{L^1(D)}\le \delta$.
    Let then $(u_\delta,\mu_\delta)$ be the solution with initial temperature $g_\delta$ and initial mushy coefficient $h = \ind_\Omega$.
    By \Cref{thm:finite-melting} there exists $R_{\delta}>0$ such that 
    \[
        \Omega \subset \Omega_{u_\delta}^+(t) \subset \Omega + B_{R_\delta} \qquad\text{for all}\quad t\ge0.
    \]
    Since $R_\delta \to 0$ as $\delta\to0$, for all $\eps>0$ there exists $\delta >0$ such that $R_\delta < \eps$.
\end{proof}

\subsection{On the fixed scale density estimates for moving points}\label{subsec:density-estimates} 

The proof of \Cref{lemma:nofiltration} requires a preliminary results, a finite speed of propagation property.

\begin{proposition}[Finite speed of propagation - closing rate]\label{prop:finite-speed-inner}
    Let $D\subset\R^d$ and $\ell\in(0,1)$.
    Let $(u,\mu)$ be the enthalpy solution in $D_\infty$ with homogeneous Dirichlet boundary conditions, initial temperature $g \in H^1_0(D) \cap L^1(D)$, and initial mushy coefficient $h\in L^\infty(D)$ that satisfies \eqref{e:phase-separation-section}.
    If there exists $x_0\in D$ and $R_0>0$ such that
    \[
        g \equiv 0 \qquad\text{in}\quad B_{R_0}(x_0),
    \]
    then, for all $\eta\in(0,\sfrac{1}{10})$ there exists $t_1 = t_1(d,\ell,R_0\,\| g\|_{L^1},\eta)>0$ such that it holds
    \[
        u(x,t) \equiv 0 \qquad \text{in}\quad B_{(1-\eta)R_0}(x_0) \qquad \text{for all} \quad t \le t_1.
    \]
\end{proposition}

\begin{proof}
    We prove that $u(\cdot,t) \le 0$, while the other inequality work by changing the sign of the barrier. Let us assume that $x_0=0$.
    Let $s : (0,1) \to (0,+\infty)$ be a strictly increasing function defined as 
    \[
        s(t) = s_\eta(t) :=  \frac{C_1\, \|g^+\|_{L^1(D)}}{1-\ell} \,t + \frac{\eta}{2},
    \]
    where $C_1= C_1(d,R_0)>0$ is a constant to be made precise later. With this definition, $s(0) = \sfrac{\eta}{2}$, and we define $t_0 >0$ as the time such that $s(t_0) = \sfrac12$.
    Let us then define the function
    \[
        R : (0,t_0) \longrightarrow (0,R_0), \qquad R(t):= \big(1-s(t)\big)R_0,
    \]
    and the function 
    $\overline{w} : \R^d \times (0,t_0) \to (0,+\infty)$ as
    \[
        \overline{w}(x,t) = \int_{\R^d\setminus B_{R(t)}} G_{R}^{E}(x,y;t)\, g^+(y) \dif y,
    \]
    where $G_{R}^E = G_{R(t)}^{E}$ is the Green function associated to $\R^d\setminus B_{R(t)}$ and with homogeneous Dirichlet boundary datum.
    We claim that for a specific choice of $C_1>0$, $\overline{w}$ is a supersolution of the Stefan problem. 

    By construction $R'(t)<0$, while the maximum principle of the heat equation implies that $\de_RG^E_{R}<0$. Therefore, by direct computation we get the following inequality
    \[
        (\de_t - \Delta)\overline{w}(x,t) = \int_{\R^d\setminus B_{R(t)}} R'(t)\,\de_R G^E_{R(t)}(x,y;t) g^+(y) \dif y\ge0 \qquad\text{in} \quad\Omega^+_{\overline{w}} = \big\{(x,t) \in \R^{d+1}_+ : |x| \ge R(t) \big\}.
    \]

    On $\Gamma^+_{\overline{w}}(t) = \de B_{R(t)}$, the Gaussian bounds for the Green functions (see Grigor'yan and Saloff-Coste \cite{GrigoryanSaloff-Coste2002:GaussianBoundExteriorsets} and Zhang \cite{Zhang:2003:GaussianBoundsManifolds}) read as follows
    \begin{equation}\label{e:gaussian-bound}
        G_{R}^E(x,y;t) \le C_{2} \frac{D_d(x,t) D_d(y,t)}{(2\pi t)^{\sfrac d2}} e^{-C_3\frac{|x-y|^2}{t}} \qquad \text{for all}\quad x,y\in \R^d\setminus B_{R(t)} \quad\text{and}\quad t>0,
    \end{equation}
    where $C_{i}= C_{i}(d,R_0)>0$ for $i=2,3$ and $D_d$ are defined as follows: 
    \[
    D_2(x,t) = \frac{\log\left(\sfrac{|x|}{R(t)}\right)}{\log(1+\sfrac{\sqrt{t}}{R(t)})+\log\left(\sfrac{|x|}{R(t)}\right)} \quad\text{if}\quad d=2 
        \qquad\text{and}\qquad
    D_d(x,t) = \frac{|x|-R(t)}{\sqrt{t}\wedge 1}\wedge 1 \quad\text{if}\quad d\ge 3.
    \]
    By direct computation, for all $d\in\NN$ there exists $C_4 = C_4(d,R_0,\eta)>0$ such that
    \[
        \big| D_d(x,t) \big|  + \big|\nabla_x D_d(x,t) \big| \le \frac{C_4}{\sqrt{t}},
    \]
    and therefore, since for all $t\le t_0$, $R_0/2 < R(t) < (1-\sfrac{\eta}{2})R_0$,  there exists $C_5 = C_5 (d,R_0)>0$ such that
    \[
        \left|\nabla_x \left(D_d(x,t) e^{-C_3\frac{|x-y|^2}{t}}\right) \right| \le C_5 \qquad \text{for all}\quad t\le t_0, \quad x\in \Gamma^+_{\overline{w}}(t), \quad\text{and}\quad y \in \Omega_g^+.
    \]

    The estimate for the Green function (up to the boundary)  implies an estimate over the gradient $|\nabla \overline{w}|$ on $\Gamma_{\overline{w}}^+(t)$, therefore we get the following estimate:
    \begin{align*}
        |\nabla u(x,t)| \le |\nabla \overline w(x,t)| 
            &\le C_5 \|g^+\|_{L^1(D)} \qquad\text{for all}\quad t\le t_0 \quad \text{and}\quad x \in \Gamma_{\overline{w}}^+(t).
    \end{align*}

    Since $V_\nu = - R'(t) = R_0\, s'(t)$, taking $C_1 := C_5/R_0$ we get
    \[
        V_\nu = \frac{C_5 }{1-\ell}  \|g^+\|_{L^1(D)} \ge \frac{1}{1-\ell} |\nabla \overline{w}| \qquad \text{on}\quad \Gamma^+_{\overline{w}}(t),
    \]
    and so $\overline{w}$ is a supersolution of the Stefan problem.
    In particular, it holds
    \[
    s(t) \le \eta \qquad \text{for all}\quad t<t_1 =  t_1(d,R_0,\eta) =  \frac{(1-\ell)R_0}{C_1\, \|g^+\|_{L^1(D)}},
    \]
    and therefore, since $\overline w$ is a supersolution of the Stefan problem we get that
    \[
        u(x,t) \le \overline{w}(x,t) =0 \qquad\text{in} \quad B_{(1-\eta)R_0} \qquad\text{for all} \quad t\le  t_1,
    \]
    and so we conclude the proof.
\end{proof}

Before giving the proof of \Cref{lemma:nofiltration}, we make a few remarks about it, in order to make clear what is the geometric means of certain hypotheses, in order to clarify the heuristic argument of the proof.

\begin{remark}[Exterior ball condition]
    \Cref{lemma:nofiltration} is stated for a target point $(x_0,t_0)$ and and assuming the exterior ball condition for a ball with radius $R_0>r_0$.
    Their role is twofold:
    \begin{itemize}
        \item Within these assumptions, there exists a unique $y_0\in \Gamma^+_g$ such that $|x_0-y_0|=r_0$.
        In particular, the Lemma is not investigating the point $x_0$, but rather the evolution of the free boundary points that lie nearby $y_0$, showing that at time $t_0$ they lies close to the point $x_0$.

        \item The choice of $t_0>0$ as the first time the free boundary moves of such length together with the exterior ball condition prevents $(x_0,t_0)$ to be a collapsing singularity of the free boundary (see \cite{Fioravanti_Ros-Oton_Torres-Latorre_2025}).
        This also assures that the positive phase spread from $y_0$ to $x_0$.
        
    \end{itemize}    
\end{remark}

\begin{remark}[Interior ball condition]
    The free boundary of the barrier from below $v$ consists in a family of traveling paraboloids with a certain speed law $s(t)$ and opening $1/r_0$.
    The interior ball condition has two consequences:
    \begin{itemize}
        \item It allows the family of paraboloids to touch $B_{r_0}(x_0)$ at initial time, and thus it invades the ball for all positive times $t_0>0$.

        \item Up to a constant, $s'(t)$ coincides with the normal speed of the free boundary. 
        In order to show $v$ is  classical subsolution, and so that
        \[
            s'(t) \le |\nabla  v(x,t)| \qquad\text{for all}\quad x\in \Gamma_v^+(t),
        \]
        we heavily rely on the boundary Harnack principle for caloric functions. To get a control that depends only on $r_0>0$, we need the size of such phase to be sufficiently large, and this is provided by the interior ball condition with $R_0>r_0$.
    \end{itemize} 
\end{remark}

\begin{proof}[Proof of \Cref{lemma:nofiltration}]
First, without loss of generality we take $R_0 = 10r_0$; second, since $\Omega^+_g$ satisfies the two-sided ball condition, there exists a unique point $y_0\in \Gamma^+_g$ such that $|x_0 - y_0|=r_0$ and up to translation and rotation, we can assume $y_0 = R_0 e_d$, $x_0 = 9r_0 e_d$ and that the ball of radius $R_0$ that touches $\Gamma^+$ from the outside at $y_0$ to be centered in the origin. We notice that with choice of coordinates the origin is the center of the exterior tangent ball (of radius $10r_0$) to $\Gamma_g^+$ at $y_0$. We also set $z_0 := {8}r_0 e_d$ and we notice that $x_0,z_0\in B_{10 r_0}\subset \R^d\setminus\Omega_g^+$.
Let us also fix $\eta = \sfrac{1}{10}$ (we will keep writing $\eta>0$ along the proof for the sake of readability).
By \Cref{prop:finite-speed-inner}, there exists $t_1 = t_1(d,r_0,\|g\|_{L^1(D)},\eta)>0$ such that 
\[
    u(\cdot,t) \equiv 0 \qquad\text{in} \quad B_{(10-\eta)r_0} \qquad \text{for all} \quad t \le t_1.
\]

Finally, we recall the notation for spherical coordinates: for all $x \in \R^d\setminus\{0\}$ there exists $\rho = |x|$ and $\alpha \in \sphere^{d-1}$ such that $ x = (\rho,\alpha)$.\medskip

We now prove the statement by constructing two barriers $w$ and $v$, respectively from above and and from below, so that the speeds of the respectively free boundaries are comparable up to a constant $C = C(d,r_0,\|g\|_{L^1})>0$. Thus, once $\Gamma^+_w(t)$ reaches $x_0$, the latter has moved enough, so it guarantees \eqref{e:nofiltration-small}.
The barriers $w$ and $v$ will be (respectively) a supersolution and a subsolution of the Stefan problem only for $t\ge t_1$ and we need this waiting time in order to apply the Boundary Harnack Principle to show that 
\[
   v(\cdot,t_0) \le u(\cdot,t_0) \le w(\cdot,t_0).
\]
Moreover, by \eqref{e:gaussian-bound} we get that there exists $C = C(d,r_0,\|g\|_{L^1(D)})>0$ such that
\[
    \|u(\cdot,t)\|_{L^\infty} \le C \qquad\text{for all}\quad t\ge t_1,
\]
and so we will consider solutions that are provided with uniform $L^\infty$-bounds.

\paragraph{The barrier from above}
Let us consider the auxiliary caloric problem defined by
\[
	\begin{cases}
		\de_t \overline{w} = \Delta \overline{w} &\text{in} \quad B_{12r_0}\times(0,t_0)\\
		\overline{w}(x,0) = g(x) &\text{in} \quad B_{12r_0}\\
		\overline{w} (x,t) = u(x,t) &\text{on}\quad \de B_{12r_0} \times (0,t_0).
	\end{cases}
\]
We define the auxiliary function $\widetilde{A} : \sphere^{d-1} \times [0,t_0) \to [0,1]$ as
\[
	\widetilde{A}(\alpha,t) :=   \overline{w}(11r_0,\alpha,t).
\] 
For every $\sigma>0$, we can define $A^\sigma : \sphere^{d-1} \times[0,t_0) \to [0,1]$ as its (space) inf-convolution 
\[
	A^\sigma(\alpha,t) := \inf_{\beta\in\mathbb S^{d-1}} \Big\{\widetilde{A}(\beta,t) + (\sfrac{1}{\sigma})^2 |\alpha - \beta|^2\Big\}. 
\]
By definition 
\[
A^\sigma(\alpha,t)\le \widetilde{A} (\alpha,t)\qquad\text{for all}\quad \alpha\in\sphere^{d-1}\quad\text{and}\quad t \ge 0.
\]
Since $t\ge t_1$, there exists $\sigma = \sigma(d,r_0,\|g\|_{L^1})>0$ sufficiently small, for which  \[
A^\sigma(\alpha,t) = \widetilde A(\beta,t) + (\sfrac{1}{\sigma})^2 |\alpha - \beta|^2,
\] 
for some $\beta\in\mathbb S^{d-1}$ satisfying $|a-\beta|\le\sfrac{1}{10}$.
By construction (see, for instance, \cite{Caffarelli_Cabre_1995:ellipticequations}) the inf-convolution $A^\sigma$ is $C^{1,1}_x$ from above with norm $\sfrac{1}{\sigma^2}$, meaning that the graph of $A^\sigma$ has the ball property from above with radius $\sigma^2$.
Moreover, by \cite[Theorem 1]{Fabes_Garofalo_Salsa_1984} there exists $C = C(d,r_0,\|g\|_{L^1})>0$ such that
\[
	\widetilde{A}(\alpha,t) \le C\,A^\sigma(\alpha,t) \qquad\text{for all}\quad \alpha\in\sphere^{d-1}\quad\text{and}\quad t \ge t_1.
\]

Finally we define the set $\Omega_w^+ \subset B_{12r_0}\times[0,t_0)$ as
\[
  \Omega_{w}^+ := \Big\{ (\rho,\alpha,t)\in \R\times\mathbb S^{d-1}\times[0,t_0) : r(\alpha,t) < \rho < 12 r_0 \Big\},
\]
where, for all $\alpha\in\sphere^{d-1}$,
\[
r(\alpha,t):=\begin{cases}
(10-\eta) r_0 &\text{for}\quad t\in[0,t_1],\\
\inf\big\{\rho(\alpha,t),r_0\big\} &\text{for}\quad t\in[t_1,t_0],\\
\end{cases}
\] 
where $\rho(\alpha,\cdot):[t_1,t_0]\to\R$ 
is defined as follows
\begin{equation}\label{e:nofiltration-upper}
    \rho(\alpha,t) := (10-\eta) r_0  - \frac{C_1}{1-\ell} \int_{t_1}^{t} A^\sigma(\alpha,\tau)\dif \tau, 
\end{equation}
where $\eta =\sfrac{1}{10}$ is the constant we discussed above, while ${C_1 = C_1(d,r_0,\|g\|_{L^1})>0}$ is a suitable large constant to be chosen in the following discussion.
Since $A^\sigma(\cdot,t)$ is $C^{1,1}_{x}$ for all $t>t_1$, the same holds for $r(\cdot,t)$ such that $r(\alpha,t)> 2r_0$; therefore, for all times $t_1 < t< t_0$, $\Omega^+_w(t)$ satisfies the exterior ball property with radius $(\sigma r_0)^2$ in $B_{12r_0}\setminus B_{2r_0}$.

Let us finally define $w : \Omega^+_w\times[0,t_0) \to [0,1]$ as the caloric function such that
\[
    \begin{cases}
        \de_t w = \Delta w & \text{in} \quad \Omega^+_w\\
        w(x,t) = u(x,t) & \text{on} \quad \de B_{12r_0} \times(0,t_0)\\
       w(x,t) = 0 & \text{on} \quad \de \Omega_w^+\cap (B_{12r_0} \times(0,t_0))\\
        w(x,0) = g(x) &\text{in} \quad B_{12r_0}\setminus B_{(10-\eta)r_0}.
    \end{cases}
\]
\Cref{prop:finite-speed-inner} provides a speed bounds, thus it follows that
\[
	\Omega_u^+(t) \subset \Omega_w^+(t) \qquad \text{for all}\quad t\le t_1,
\]
and since $u$ and $w$ are caloric in $\Omega_u^+ \cap B_{12r_0} \cap \{t\le t_1\}$, we get
\[
	u(x,t_1) \le w(x,t_1) \qquad\text{for almost every}\quad x\in B_{12r_0}.
\]

We claim that for a certain choice of $C_1 = C_1(d,r_0)>0$, $w(x,t)$ in a classical supersolution of the Stefan problem in $(B_{12r_0}\setminus B_{2r_0})\times(t_1,t_0)$ with respect to $h(x) \equiv (1-\ell)$. Thus, since 
\[
	u(\cdot,t) \equiv 0 \qquad\text{on}\quad \de B_{2r_0} \qquad\text{for all}\quad t\le t_0,
\]
for all $t\le t_0$ it holds
\[
	u(x,t) \le w(x,t) \qquad\text{for almost every}\quad x\in B_{12r_0}.
\]
We next prove that $w(x,t)$ in a classical supersolution of the Stefan problem.
Since $w$ is caloric in $\Omega_w^+$, we only need to check the free boundary condition.
However, for all times $t_1 < t < t_0$, the sets $\Omega^+_w(t)$ are increasing in time and have the exterior ball condition outside $B_{2r_0}$.
Thus, Parabolic Boundary Harnack inequalities apply (see, for instance, \cite{BanerjeeGarofalo2016:Higher-Order-Parabolic-Boundary-Schauder,DeSilvaSavin2022:ParabolicBoundaryHarnack,AudritoFioravantiVita2024:Parabolic-Higher-Order-Schauder,AudritoFioravantiVita2025:Higher-Order-Parabolic-Boundary-Schauder}) and thus there exists $C_2 = C_2(d,r_0,t_1) = C_2(d,r_0,\|g\|_{L^1})>0$ such that
\begin{equation}\label{e:nofiltration-boundary-harnack-above}
    |\nabla w(r(\alpha,t),\alpha) | \le C _2A^{\sigma}(\alpha,t) \qquad \text{for all} \quad t_1 < t < t_0 \quad \text{such that} \quad r(\alpha,t) > 2r_0.
\end{equation}
We observe that, the latter estimate does not depend on $r'(\alpha,t)$, but only on $r(\alpha,t)$ (and so, on $r_0>0$).
Since $V_\nu(\alpha,t) = r'(\alpha,t)$, taking $C_1 = C_2$ as in \eqref{e:nofiltration-boundary-harnack-above}, it follows that $w$ is a classical supersolution of the Stefan problem in $(B_{12r_0}\setminus B_{2r_0})\times(t_1,t_0)$, and thus
\[
	\Omega_u^+(t) \subset \Omega_w^+(t) \qquad\text{for all}\quad t \le t_0.
\]

\begin{center}
    \begin{figure}[ht]

        \begin{tikzpicture}[scale=0.7]
		\begin{scope}
			\clip (0,0) circle (5);

\begin{scope}
    \clip (0,0) circle (6);

\fill[MidnightBlue!10, even odd rule]
    (-10,-10) rectangle (10,10)
    plot[smooth cycle] coordinates {
        (2.5,0) (2,1) (0,3) (-2,2.6) (-3.5,0) (-2.5,-2.5) (0,-3.4) (2,-2.5) 
    };
    \draw[thick] 
    plot[smooth cycle] coordinates {
        (2.5,0) (2,1) (0,3) (-2,2.6) (-3.5,0) (-2.5,-2.5) (0,-3.4) (2,-2.5) 
    };

    \draw[fill=MidnightBlue!40,use Hobby shortcut]
    (6,-6) .. (3,-3) .. (3.3,-2) .. (3.1,0) .. (3.5,2) .. (3,5) .. (6,6) -- (6,-6);

     \end{scope}

		\end{scope}
		\node at (2,3) {$\Omega_w^+(t)$};
		\node at (4,0) {$\Omega_u^+(t)$};
		\node at (2.7,0) {$x_0$};

            \fill[black] (3,0) circle (1pt);
            \draw[dashed] (0,0) circle (3);
            \draw[dashed] (0,0) circle (5);

        \end{tikzpicture}
         \caption{The sets $\Omega_w^+(t)$ and $\Omega_u^+(t)$ for $t\le t_0$. The interior dashed ball $B_{9r_0}$, the larger ones $B_{12r_0}$. }
        \label{fig:nofiltration-barrier-above}
    \end{figure}
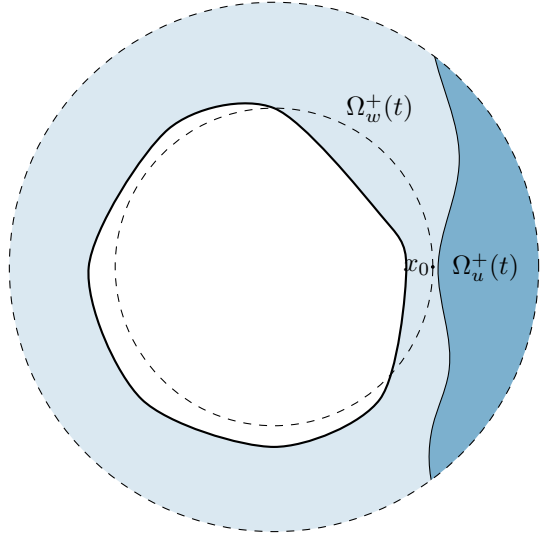
\end{center}

\paragraph{The barrier from below}
This time we define a barrier from below $v$ that is caloric and satisfies the classical subsolution condition only in $(B_{4r_0}(z_0)\setminus \Omega_g^+) \times (t_1,t_0)$.
Let us define the function
\[
    a : (0,t_0) \longrightarrow \R, \qquad
    a(t) := u(11r_0,e_1,t).
\]
By \Cref{prop:one-phase:positivity-increasing}, the positivity sets $\Omega_u^+(t)$ are increasing in time.
Moreover, the interior ball condition and Harnack inequalities for the caloric functions imply that there exists $K_1 = K_1(d,r_0,\|g\|_{L^1})>0$  such that
\[
    \frac{1}{K_1} < a(t) < K_1 \qquad\text{for all}\quad t\ge t_1. 
\]

Then we define $\Omega_v^+\subset \R^{d+1}_+$ as
\begin{equation}\label{e:nofiltration-below-set}
    \Omega_v^+ := \left\{ (x,t) \in \R^d\times(0,t_0) : x_d > \frac{1}{r_0}|x'|^2 + s(t) \right\},
\end{equation}
where $s : (0,t_0) \to (0,+\infty)$ is defined as
\[
    s(t) := 
    \begin{cases}
        10r_0 & \text{for} \quad t\in [0,t_1]\\
      \displaystyle  10r_0 - \frac{c_1}{1+\ell} \int_{t_1}^t a(\tau) \dif \tau &\text{for}\quad t\in[t_1,t_0],
    \end{cases}
\]
where $c_1 = c_1(d,r_0,\|g\|_{L^1(D)}) >0$ is to be defined later.

We define $v : \Omega_v^+ \cap \big(B_{4r_0}(z_0) \times (0,t_0)\big) \to \R $ as the solution of the following differential problem
\[
	\begin{cases}
		\de_t v = \Delta v &\text{in} \quad \Omega_v^+ \cap (B_{4r_0}(z_0) \times(0,t_0))\\
		v(x,0) = \phi(x)\,g(x) & \text{in}\quad  \Omega_v^+(0) \cap B_{4r_0}(z_0)\\
		v(x,t) = \phi(x)\,u(x,t) &\text{on}\quad  \Omega_v^+(t) \cap \de B_{4r_0}\\
		v(x,t) = 0 &\text{on}\quad \de \Omega_v^+(t) \cap B_{4r_0},
	\end{cases}
\]
where $0\le\phi\le\sfrac12$ is a cut-off function with $\Omega_{\phi}^+ = \Omega_v^+(0)$ and $\phi(e_1) = \sfrac12$.

By construction, 
\[
	\Omega_v^+(t) \cap B_{4r_0}(z_0) \subset \Omega_u^+(t) \cap B_{4r_0} (z_0)\qquad\text{for all}\quad t \le t_1,
\]
and thus the comparison principle for the heat equation implies that 
\begin{equation}\label{e:sotto-time-t_1}
    v(x,t_1) \le u(x,t_1) \qquad \text{for almost every}\quad x \in B_{4r_0}(z_0).
\end{equation}
We will prove that for all $t\le t_0$ we have
\begin{equation}\label{e:che-fai-tocchi?Noon-si-tocca}
    \Omega_v^+(t) \subset \Omega_u^+(t) 
        \qquad\text{and}\qquad
    v(x,t) \le u(x,t) \qquad \text{for almost every}\quad x\in B_{4r_0}(z_0).
\end{equation}
In order to prove \eqref{e:che-fai-tocchi?Noon-si-tocca} for all $t\le t_0$, we will first show that for a certain choice of $c_1>0$, $v$ is a classical subsolution of the Stefan problem in $(B_{4r_0}(z_0)\setminus \Omega_g^+) \times (t_1,t_0)$ and, in order to do this, we only need to show that $v$ satisfies the subsolution condition on $\Gamma_v^+ \setminus (\Omega_g^+\times(t_1,t_0))$. Notice that, by construction, for all $t\ge t_1$, $\Omega_v^+(t)$ satisfies the interior ball condition with radius $r_0>0$. Moreover, there exists $C = C(d,r_0,t_1)>0$ such that $s'(t) = \frac{c_1}{1-\ell}a(t) \le\frac{1}{1-\ell}C_1$.
Thus, by the interior Harnack inequality there exists $c  = c(d,r_0,t_1)>0$ such that
\[
    v(11r_0\,e_1,t) \ge c\, a(t) \qquad\text{for all}\quad t\ge t_1.
\]
Thus, the Hopf Lemma applies and so there exists a constant $c_2 = c_2(d,r_0,t_1) = c_2(d,r_0,\|g\|_{L^1})>0$ such that for all $x\in \Gamma^+_v(t)\cap \Omega_g^+$ (by construction these are $r_0$ away from the boundary of $B_{4r_0}(z_0)$ ) it holds
\[
    |\nabla v(x,t)| \ge c_2\, a(t) \qquad \text{for all}\quad x \in \Gamma^+_v(t) \cap \big( B_{4r_0}(z_0) \setminus \Omega_g^+\big) \quad \text{and for all}\quad t\ge t_1.
\]
By comparison principle, the constant $c_2>0$ is uniform for all choices of $c_1\in(0,1]$. Therefore, we can choose $c_1 = \min \big\{1, c_2\big\}$. Now, since $V_\nu \le s'(t)$, we get that for all $t\ge t_1$ it holds
\[
    V_\nu(x,t) \le \frac{c_1}{1+\ell} a(t)\le \frac{c_2}{1+\ell} a(t) \le \frac{1}{1+\ell} |\nabla v(x,t)| \qquad \text{for all}\quad x\in \Gamma_v^+(t) \cap \big(B_{4r_0}(z_0) \setminus \Omega_g^+\big).
\]
This implies that $v$ is a (classical) subsolution to the Stefan problem in $B_{4r_0}(z_0) \setminus \Omega_g^+$. 

We are now in position to show \eqref{e:che-fai-tocchi?Noon-si-tocca} for all $t\le t_0$.
We first notice that, by construction
\[
    \Omega_v^+(t) \cap B_{4r_0} \subset \Omega_g^+ \qquad\text{for all} \quad 0\le t\le t_1.
\]
Thus $v$ and $u$ are both caloric in $\Omega_v^+ \cap \big(B_{4r_0}\times(0,t_1)\big)$ and so, by the maximum principle, we get that
\[
    v(x,t) \le u(x,t) \qquad \text{for all}\quad x \in B_{4r_0}(z_0)\quad \text{and all}\quad t\le t_1.
\]
This proves \eqref{e:che-fai-tocchi?Noon-si-tocca} for $t\le t_1$. When $t\ge t_1$ we argue by contradiction. Suppose that there are $t\in(t_1,t_0)$ and $x\in B_{4r_0}(z_0)$ such that $u(x,t)<v(x,t)$ and let
\[
t^\ast := \inf\big\{ t > 0 : v(\cdot,t) \not<  u(\cdot,t) \quad \text{in}\quad \Omega_v^+(t) \big\}.
\]
As $u$ is continuous, we have that $v(\cdot,t^\ast) \le u(\cdot,t^\ast)$ on $B_{4r_0}(z_0)$ and, in particular, $\Omega_v^+(t^\ast)\subset\Omega_u^+(t^\ast)$, which of course means that \eqref{e:che-fai-tocchi?Noon-si-tocca} holds for all $t\le t^\ast$. Moreover, one of the following holds:
\begin{enumerate}[(A)]
\item\label{item:caseA-sotto} there exists $x^\ast \in B_{4r_0}(z_0)$ such that $v(x^\ast,t^\ast) = u(x^\ast,t^\ast) >0$;
\item\label{item:caseB-sotto} there exists $x^\ast \in B_{4r_0}(z_0)\cap \Gamma_v^+(t^\ast)\cap\Gamma_u^+(t^\ast)$.
\end{enumerate}
Since the positivity sets are ordered for all $t\le t^\ast$, the maximum principle for caloric functions implies that if case \ref{item:caseA-sotto} happens, then $u$ and $v$ have to coincide for all $t\le t^\ast$, which is a contradiction since by construction $v<u$ on $\partial B_{4r_0}(z_0)\cap\{v>0\}$. In order to rule out case \ref{item:caseB-sotto}, we notice that by construction the contact point $x^\ast$ lies necessarily in the interior of the set $ B_{4r_0}\setminus\Omega_g^+$. 
However, the temperatures are ordered $\de_P \big(B_{4r_0}\setminus\Omega_g^+ \times(t_1,t_0)\big)$ (by case \ref{item:caseA-sotto} and \eqref{e:sotto-time-t_1}) while in the interior $v$ is a classical subsolution of the one-phase Stefan problem.
These lead to a contradiction, since one-phase weak solutions are viscosity solutions \cite{Paiano2026:PhD}, and hence we prove \eqref{e:che-fai-tocchi?Noon-si-tocca}.

\paragraph{The movement of $v$ and $w$}
In the previous steps, we proved that
\[
    \Omega_v^+(t) \cap B_{4r_0}(z_0) \subset \Omega_u^+(t) \cap B_{4r_0}(z_0) \subset \Omega^+_w(t) \cap B_{4r_0}(z_0) \qquad \text{for all}\quad  t \le  t_0.
\]
We now show that the velocities of $\Gamma^+_v(t)$ and $\Gamma^+_w(t)$ are comparable. 
Let $t_2 \le t_0$ be such that $(x_0,t_2) \in \Gamma^+_w(t_2)$. 
Since $w$ and $u$ coincide on $\de_P (B_{4r_0} \times (0,t_0))$ and $\distop(11r_0,\Gamma^+_u(t)) \ge r_0$, the Harnack inequality for caloric functions implies that there exists $C_2 = C_2(d,r_0,\|g\|_{L^1})>0$ such that
\[
    A^\sigma(e_d,t) \le C_2\,a(t) \qquad\text{for all} \quad t_1 \le t \le t_0,
\]
and thus we get that 
\begin{align*}
    |x_0|= 9 r_0 
        = (10-\eta)r_0 - \frac{C_1}{1-\ell} \int_{t_1}^{t_2} A^\sigma(e_d,t) \dif t 
        &\ge (10-\eta)r_0 - \frac{C_1C_2}{(1-\ell)c_1} \int_{t_1}^{t_2} \frac{c_1}{1+\ell}a(t)\dif t \\
        &= (10-\eta)r_0 - \frac{C_1C_2}{(1-\ell)c_1} \big(10r_0-s(t_2)\big).
\end{align*}
and so, since $\eta < \sfrac{1}{2}$, we get that for $\sfrac1C = \frac{c_1}{2C_1C_2}$ it holds that
\begin{equation}\label{e:nofiltration-movement-below}
    10r_0-s(t_2) \ge \frac{1-\ell}{C} \,r_0.
\end{equation}

\paragraph{The density estimate}
In order to get \eqref{e:nofiltration-small}, we combine \eqref{e:nofiltration-movement-below} with the geometric properties of $\Omega_v^+(t)$.
Indeed, the latter is a paraboloid of opening $\sfrac{1}{r_0}$ and so there exists a dimensional constant $\kappa = \kappa(d)>0$ such that 
\[
\big|\Omega^+_u(t_2) \cap B_{r_0}(x_0) \big| \ge \kappa \left(\frac{(1-\ell)^{\sfrac12} r_0}{C}\right)^{d}.
\]
Thus, we define $c_d : (0,+\infty) \to (0,+\infty)$ as
\begin{equation}\label{e:function-density}
c_d(d,\|g\|_{L^1(D)}; r_0) := \kappa\left( \frac{r_0}{C(r_0)}\right)^d,
\end{equation}
and yet
\[
    \big| \Omega^+_u(t_0) \cap B_{r_0}(x_0) \big|
        \ge
     \big| \Omega^+_v(t_2) \cap B_{r_0}(x_0) \big|
        \ge
    c_d(r_0).
\]

\paragraph{The constants}
To complete the proof of \Cref{lemma:nofiltration}, we need to show that $c_d'>0$ and that it depends continuously on the $L^1$-norm of $g$; both properties follow from the theory of the boundary Harnack inequalities.
Indeed, by construction, $t_1(r_0)$ is increasing while the $C^{1,1}$-norms of $\Gamma_w^+(t)$ and $\Gamma_v^+(t)$ decrease as $r_0$ grows; thus $c_d'>0$ is increasing in $r_0$.
Similarly, the constants of the boundary Harnack inequality depends linearly by $\| u \|_{L^\infty}$ (and hence on $\| g\|_{L^1(D)}$, as we are working in $\{t>t_1\}$), and hence we can make a continuous choice of $c_d$.
\end{proof}

\begin{remark}
    The dependence of $c_d$ on the $L^1$-norm of $g$ comes from the Gaussian bounds we use in the proof of \Cref{prop:finite-speed-inner} and from the classical $L^1-L^\infty$ estimates for the caloric functions (with homogeneous Dirichlet boundary conditions).
    If one consider the problem with bounded temperature $\|u\|_{L^\infty} \le K$, then the estimates would have been incorporated only the constant $K>0$. However, we decided to argue with a $L^1$-initial datum in light of \Cref{thm:total-monotonicity-formula}.
\end{remark}

\Cref{lemma:nofiltration} has a stronger formulation if we assume $r_0>\sfrac{1}{10}$ (or any universal bound from below).

\begin{lemma}\label{lemma:nofiltration-large}
    Let $D\subset \R^d$ be an open set and $\ell \in (0,1)$. Let $(u,\mu)$ be an enthalpy solution with homogeneous Dirichlet boundary condition, initial tempearture $g \in H^1_0(D) \cap L^1(D)$, and initial mushy coefficient $h\in L^\infty(D)$ that satisfies \eqref{e:phase-separation-section}.
    Then there exists a constant $c_d = c_d(d,\| g\|_{L^1(D})>0$ with the following property.
    For all $r_0>0\sfrac{1}{10}$ and $R_0 > 10r_0$, such that $\Omega_g^+$ satisfies the two-sided ball condition with radius $R_0$.
    Then, for all $(x_0,t_0) \in \Gamma^+$  such that
    \begin{equation}
        x_0 \in \Gamma_u^+(t_0) \cap \de (\Omega_g^+ + B_{r_0}) \qquad\text{and}\qquad
        t_0 := \min\Big\{ t>0 : \Gamma^+_u(t) \cap \de (\Omega_g^+ + B_{r_0}) \ne\emptyset \Big\},
    \end{equation}
    the following density estimate holds
    \begin{equation}\label{e:nofiltration-big}
        \frac{\big|\Omega_u^+(t_0) \cap B_{r_0}(x_0)\big|}{|B_{r_0}|} \ge (1-\ell)^{\sfrac d2} \, c_d.
    \end{equation}
\end{lemma}

\begin{proof}
The proof follows the same steps of \Cref{lemma:nofiltration}, with the following modifications:
\begin{itemize}
    \item the time $t_1>0$ depends on $r_0$ no more, thus $t_1 = t_1(d)$, and the same holds for the constant arising in the Harnack inequalities we used, that degenerate as $t_1\to0$;
    \item there exists $C = C(d)>0$ such that $\| u\|_{L^\infty} \le C \|g\|_{L^1(D)}$; the same holds for $w$, $v$ (and thus $A^\sigma(\alpha,t)$ and $a(t)$);
    \item the $C^{1,1}_x$-norm of $\Omega^+_w$ and $\Omega^+_v$ are bounded by above by $\sfrac{1}{r_0}$; thus if $r_0>\sfrac{1}{10}$, then there exists $C = C(d)>0$ and $c = c(d)>0$ such that
    \[  
        r'(\alpha,t) = -c \,A^\sigma(\alpha,t) \qquad \text{on}\quad \Gamma^+_w(t) 
            \qquad\text{and}\qquad
        s' (t) = - c\, a(t) \qquad \text{on}\quad \Gamma^+_v(t),
    \]
    so the velocities are uniformly bounded by the $L^1$-norm of $g$.
\end{itemize}
Thus, these considerations allow us to define $c_d$ in \eqref{e:function-density} as $c_d(r_0) := c_d \,r_0^{d}$, and thus \eqref{e:nofiltration-big} follows.
\end{proof}

\section{Examples of solutions with asymptotically persistent mushy region}\label{sec:examples}
 
In this section we use the results of \Cref{sec:asymptotic-analysis} to provide examples of two-phase solutions whose mushy region is asymptotically persistent, that is, the mushy region has a measure lower bound that persists as $t\to+\infty$. We will construct the examples by multiple odd reflections via the following proposition, while the persistence of the mushy region will be a consequence of \Cref{thm:finite-melting}.

\begin{proposition}\label{prop:odd-reflections}
    Let $D\subset\R^d$ be an open set, $x_0\in \de D$, and be $\nu \in \mathbb S^{d-1}$ such that $D \cap D_{x_0,\nu} = \emptyset$, where
    \[
         {D}_{x_0,\nu} := \Big\{ x \in \R^d : \tilde x = \tilde x_{x_0,\nu} := \big(x - 2 ((x-x_0)\cdot \nu)\nu\big)  \in D  \Big\},
    \]
    and
    \[
    \widetilde D:=\Interior\Big(D\cup D_{x_0,\nu}\cup\big(\partial D\cap\{(x-x_0) \cdot\nu =0\}\big)\Big).
    \]
    For all functions $v : D \to \R$ (resp. $v : D_\infty \to \R$) let us define $\tilde v = \tilde{v}_{x_0,\nu} : D \cup D_{x_0,\nu} \to \R$ (resp. $\tilde v : D_\infty \cup (D_{x_0,\nu})_{\infty} \to \R$) as 
    \[
    \tilde v(x) := 
    \begin{cases}
        v(x) &\text{if} \quad x \in D,\\
        -v(\tilde x) &\text{if}\quad x \in D_{x_0,\nu}.
    \end{cases}
    \qquad
    \left( \text{resp.}\quad
        \tilde v(x,t) := 
    \begin{cases}
        v(x,t) &\text{if} \quad x \in D,\\
        -v(\tilde x,t) &\text{if}\quad x \in D_{x_0,\nu}.
    \end{cases}
    \right)
    \]
    Let $g \in H^1_0(D)$, $h \in L^\infty(D)$ with $|h(x)| \le 1$, and $(u,\mu)$ the relative enthalpy solution in $D_\infty$ with homogeneous Dirichlet boundary conditions.
    Then $(\tilde u,\tilde\mu)$ is the enthalpy solution in $\widetilde D_\infty$ with initial temperature $\tilde g$, initial mushy coefficient $\tilde h$, and homogeneous Dirichlet boundary conditions.

\end{proposition}
    \begin{proof}
  Without loss of generality, we take $x_0=0$ and $\nu=e_d$.
  We only need to check that $(\widetilde u,\widetilde \mu)$ satisfies the integral identity from \Cref{def:enthalpy_solution} for test functions $\eta\in C_c^\infty(\widetilde D\times(0,+\infty))$. Let $R,T>0$ be such that $\eta(x,t) \equiv 0$ if either $|x|\ge R$ or $t\ge T$. 
We define $p_\sigma\in C^\infty(\R)$ as
        \[
        p_\sigma(z) = 0 \qquad\text{if}\quad z\le \sigma, 
            \qquad 
        p_\sigma(z) = 1 \qquad\text{if}\quad z\ge2\sigma, 
            \qquad\text{and}\qquad 
        0\le p_\sigma'(z) \le \frac{2}{\sigma},
        \]
and we consider the truncated functions $\eta_\sigma := (p_\sigma(x_d) + p_\sigma(-x_d))\eta\in C^\infty_c((\widetilde D\setminus\{x_d=0\})\times(0,+\infty))$. Since $\eta_\sigma$ is supported away from $\{x_d=0\}$, we have that
        \[
            \iint(\widetilde u+\widetilde\mu) \de_t \eta_\sigma - \nabla \widetilde u\cdot \nabla \eta_\sigma \dif x \dif t =0.
        \]
        Through the change of variables $x_d \to (-x_d)$, the previous identity reads as
        \[
            \int_0^{T}\int_{\widetilde D} \big((\widetilde u+\widetilde \mu)\de_t \eta - \nabla \widetilde u\cdot \nabla \eta \big) (p_\sigma(x_d) + p_\sigma(-x_d)) \dif x \dif t 
            = \int_0^{T}\int_{\widetilde D} \eta\nabla \widetilde u\cdot \nabla (p_\sigma(x_d) + p_\sigma(-x_d))  \dif x \dif t.
        \]
        We next estimate the right-hand side of the above identity:
        \begin{align*}
            \left\vert\int_0^{T}\int_{\widetilde D} \eta\nabla \widetilde u\cdot \nabla (p_\sigma(x_d) + p_\sigma(-x_d))  \dif x \dif t\right\vert
            &\le 
            \int_0^{T}\int_{D} |\eta(x',x_d,t) - \eta(x',-x_d,t)|\,|\nabla u|\, p'_\sigma(x_d) \dif x \dif t\\
             &\le 
            2\|\nabla\eta\|_{L^\infty(\R^d\times(0,+\infty))}\int_0^{T}\int_{D} |x_d|\,|\nabla  u|\, p'_\sigma(x_d) \dif x \dif t\\
            &\le
            8\|\nabla\eta\|_{L^\infty(\R^d\times(0,+\infty))} \int_0^T \int_{D\cap E_{R}^\sigma} |\nabla u| \dif x \dif t,
        \end{align*}
        where $E_R^\sigma :=  B_R \cap \{|x_d|\le 2\sigma\}$ and $\eta(x,t) = 0$ for all $x\in \R^d\setminus B_R$ and $t \ge T$. 
        By the H\"older inequality and the following energy estimate 
        $$\|\nabla u(\cdot,t)\|_{L^2(D)}^2\le 8\|g\|_{L^2(D)}^2\quad\text{for all}\quad t\in(0,T),$$ 
        which was proved in \cite[Theorem 1.7 (ii.a)]{Paiano_Velichkov2026:ellipticregularization}, we get
        \[
            \int_0^T \int_{D\cap E_{R}^\sigma} |\nabla u| \dif x \dif t
            \le 
            T |D\cap E_{R}^\sigma|^{\sfrac12}\sup_{t\in(0,T)}\| \nabla  u (\cdot,t)\|_{L^2(D)}  \le C_d  T R^{\frac{d-1}{2}} \sigma^{\sfrac12}\| \nabla  g\|_{L^2(D)},
        \]
        where $C_d$ is a dimensional constant. Finally, since  $(p_\sigma(x_d) + p_\sigma(-x_d))\nearrow 1$, we get that 
          \begin{align*}
            \iint(\widetilde u+\widetilde\mu) \de_t \eta - \nabla \widetilde u\cdot \nabla \eta \dif x \dif t&= \lim_{\sigma\to0}\int_0^{T}\int_{\widetilde D} \big((\widetilde u+\widetilde \mu)\de_t \eta - \nabla \widetilde u\cdot \nabla \eta \big) (p_\sigma(x_d) + p_\sigma(-x_d)) \dif x \dif t\\
            &= \lim_{\sigma\to0}\int_0^{T}\int_{\widetilde D} \eta\nabla \widetilde u\cdot \nabla (p_\sigma(x_d) + p_\sigma(-x_d))  \dif x \dif t=0,
        \end{align*}
        which concludes the proof.
    \end{proof}



\subsection{Solutions with compact initial data}
Finally, we are in a position to provide two non-trivial examples of $H^1$ enthalpy solution with initial temperature $g\not\in L^\infty(\R^d)$ with compact support and for which the mushy region persists in a compact set.

\begin{example}[Symmetric two-phase solutions]\label{ex:bounded-singular}
    Let $d\ge3$, $D := \{x_d>0\}$, $e_d = (0,\dots,0,1)$, $\beta < \frac{d-2}{2}$, $c>0$, and $ g =  g_{\beta,c} \in H^1_0(D)$ be defined as
    \begin{equation}\label{e:example-radial}
          g(x) := c\left(|x-e_d|^{-\beta} -1 \right)^+, 
    \end{equation}
    so that $\Omega_{ g}^+= B_1(e_d) \subset D$.
    We now define $(u,\mu)$ and $(\overline u,\overline \mu)$ as the enthalpy solutions (of the one-phase Stefan problem) in $D\times[0,+\infty)$ and in $\R^d\times [0,+\infty)$ respectively with initial temperature $g$ and initial mushy coefficient $h= \ind_{\Omega_g^+}$.    
    Since $ g$ and $h$ are radially symmetric (with respect to $e_d$), $\Omega^+_{\overline u}(t) = B_{R(t)}(e_d)$ and
    \Cref{lemma:measure-est} implies that for $c < c(d,\beta) = (2^d-1)(d-\beta)|B_1|$ it holds that
    \[
        \Omega_{ u}^+(t) \subset \Omega_{\overline u}^+(t) \subset B_2(e_d) \qquad\text{for all}\quad t>0.
    \]
    Let us define $(\tilde u,\tilde \mu)$ as in \Cref{prop:odd-reflections} (with $D=\{x_d>0\}$, $x_0=0$, and $\nu = e_d$). By \Cref{prop:one-phase:positivity-increasing} $\Omega_{\tilde u}^+(t)$ and $\Omega_{\tilde u}^-(t)$ are increasing in time and so, by comparison with $\overline u$ (and $-\overline u$), it follows that there exists a function $0 < r(t) < \sqrt{2}$, $r'(t)\ge0$ such that 
    \[
        \Gamma_{tp}(t) = B_{r(t)} \cap \{ x_d=0\} \qquad \text{for all}\quad t\ge0.
    \]
    In particular, for all $\alpha \in \mathbb S^{d-1}$ there exists a curve of \emph{branching points} $x_{\alpha}(t) = r(t) \alpha \in \Gamma_{tp}(t)$ such that 
    \[
        B_r(x_\alpha(t)) \cap \Omega_{u}^+(t)\ne\emptyset, 
            \quad
        B_r(x_\alpha(t)) \cap \Omega_{u}^-(t)\ne\emptyset,
            \quad\text{and}\quad
        B_r(x_\alpha(t)) \cap \{u(\cdot,t)=0\}\ne\emptyset,
    \]
    for all $t\ge0$.

\end{example}

\begin{example}[Non symmetric singular example]\label{ex:non-symmetric-bounded-singular}
    Let $d\ge3$, $D= \{x_d>0\}$, and $ g =  g_{\beta,c}$ as in \Cref{ex:bounded-singular}. Let $c\in (0,1)$, we define 
    \begin{equation*}
        \hat g(x) :=
            \begin{cases}
                 g(x',x_d) & \text{if}\quad x_d>0\\
                -c  g(x',-x_d)& \text{if}\quad x_d<0,
            \end{cases}
    \end{equation*}
    and $(\hat u,\hat\mu)$ as the solution of the Stefan problem in $\R^{d+1}_+$ with initial temperature $G$ and initial mushy coefficient $h=\widetilde\ind_{\Omega_g^+}$. By comparison with the one-phase solution $(\overline u,\overline \mu)$, it follows that $\Omega_u(t)^+ \subset B_{2}(e_d)$ for all $t>0$.
    Then, $u\le 0$ in $(\R^d\setminus B_{2}(e_d)) \times[0,+\infty)$, and thus, as in \Cref{prop:one-phase:positivity-increasing} it follows that $u<0$ in $(B_{2}(-e_d) \setminus B_{2}(e_d))\times[0,+\infty)$.
    Analogously, $u>0$ in $(B_{2}(e_d) \setminus B_{2}(-e_d))\times[0,+\infty)$, and thus $(\hat u,\hat \mu)$ is a non-trivial non-symmetric two-phase solution with non-empty mushy region.
\end{example}

\begin{remark}[The case $d=2$]
        Let $d=2$, $D= \{x_2>0\}$, $e_2 = (0,1)$, $\beta <\sfrac12$, and $ g_{\beta,c}\in H^1_0(D)$ defined as
    \begin{equation}\label{e:example-radial-R2}
        g =  g_{\beta,c}(x) := c \left(\left(-\log\left\vert x-\frac12e_2 \right\vert\right)^{\beta}-(\log2)^\beta\right)_+.
    \end{equation}
    Then, the same arguments of \Cref{ex:bounded-singular,ex:non-symmetric-bounded-singular} apply, and prove that there exist non-trivial solutions with persistent mushy region in $\R^2$.
\end{remark}

\begin{center}
    \begin{figure}[ht]
	\begin{tikzpicture}[use Hobby shortcut,scale=1]
        \begin{scope}[shift={(-4,0)},scale=0.7]
			\clip (0,0) circle (4);
            \draw[dashed] (-4,0) -- (4,0);
            
			\draw[thick, fill = MidnightBlue!10] 
                (-2,0) -- (0,0) -- (2,0) .. (2.1,0.01) .. (3,1) .. (0,3.5) .. (-3,1) .. (-2.1,0.01) .. (-2,0);
			\draw[thick, fill = MidnightBlue!20,closed] 
                (-1.5,1) .. (0,0).. (1.5,1) .. (0,2.5);
                
			\draw[thick, fill = black!7,yscale =-1]     
                (-2,0) -- (0,0) -- (2,0) .. (2.1,0.01) .. (3,1) .. (0,3.5) .. (-3,1) .. (-2.1,0.01) .. (-2,0);
			\draw[thick, fill = black!14,closed,yscale=-1]     
                (-1.5,1) .. (0,0).. (1.5,1) .. (0,2.5);

			\draw[dashed, thick] (0,0) circle (3.99);

						\node at (0,-1) {$\Omega_u^-(0)$};
			\node at (0,1) {$\Omega_u^+(0)$};
			\node at (0,3) {$\Omega_u^+(t)$};
			\node at (0,-3) {$\Omega_u^-(t)$};

		\end{scope}

        \begin{scope}[shift={(3,0)},scale=0.7]
			\clip (0,0) circle (4);
            \draw[dashed](-4,0) --(4,0);
            
			\draw[thick, fill = MidnightBlue!10] 
                (-1.65,-0.3) --  (1.65,-0.3) .. (1.9,-0.345) .. (2,-0.37) .. (3,1) .. (0,3.5) .. (-3,1) .. (-2,-0.37) .. (-1.9,-0.345) .. (-1.65,-0.3);
                
			\draw[thick, fill = black!7,shift={(0,-0.3)},yscale =-1,scale=0.8]
                (-2,0) -- (0,0) -- (2,0) .. (2.1,0.01) .. (3,1) .. (0,3.5) .. (-3,1) .. (-2.1,0.01) .. (-2,0);

			\draw[dashed, thick] (0,0) circle (3.99);

			\node at (0,-1) {$\Omega_u^-(t)$};
			\node at (0,1) {$\Omega_u^+(t)$};

		\end{scope}

    \end{tikzpicture}
         \caption{The figure on the left represents as solution as in \Cref{ex:bounded-singular}. On the right the non-symmetric case \Cref{ex:non-symmetric-bounded-singular}. The dashed line corresponds to $\{x_d=0\}$.}
        \label{fig:ex:bounded-singular}
    \end{figure}
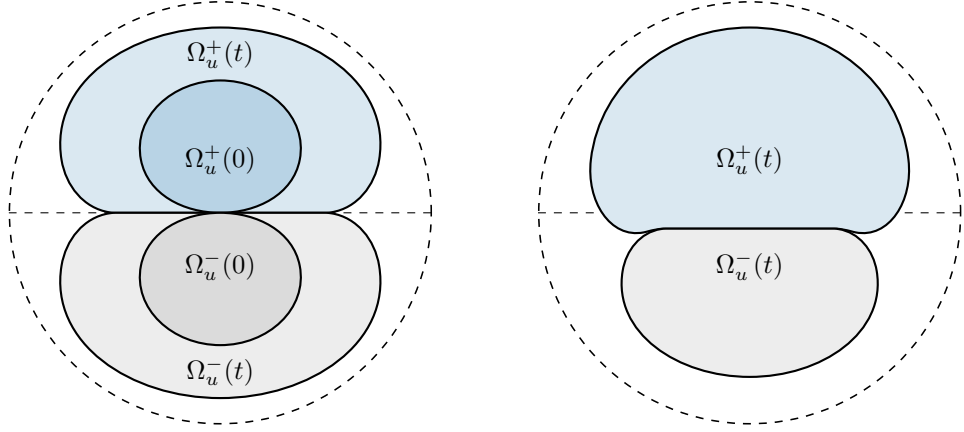
\end{center}

\begin{remark}
    In \Cref{ex:bounded-singular,ex:non-symmetric-bounded-singular} we do not fully rely on \Cref{thm:finite-melting}. Indeed, as we rely on a radial competitor, the existence of such solutions follows directly from \Cref{lemma:measure-est}. This works properly radial initial data, but it fails to guarantee that $R_g \to0$ as $g\to0$ for non-radial ones.
\end{remark}

\subsection{Non compact initial temperature}

In this section we provide a two examples of solutions with persistent mushy region and non compact initial support. 
To present the first one, we need no further results than \Cref{lemma:measure-est}.

\begin{example}[Fat dense set]
	Let $\{q_j\}_{j\in\NN}\subset\mathbb Q^d$ be a dense family of rational points in $D=\{x_d>0\}$. We define the following open set $\Omega \subset D$ as follows
	\[
	\Omega := \bigcup_{j\in\NN} \big(B_{2^{-j}}(q_j) \cap D\big),
	\]
	and for $\beta,c>0$ let $ f =  f_{\beta,c}\in H^1_0(D)$ be defined as
	\[
	 f(x) := (x_d \wedge 1) \left(\sum_{j\in\NN} 2^{-j} g_{\beta,c}\left(2^{j}(x+e_d-q_j)\right)\right),
	\]
	where  $ g_{\beta,c}$ are defined as in \eqref{e:example-radial} (or \eqref{e:example-radial-R2}, if $d=2$).
	Then $f\in H^1_0(D)\cap L^1(D)$ and $\Omega_f^+$ is an open dense set with finite measure.
	Let $( u, \mu)$ be the one-phase enthalpy solution of the Stefan problem in $D\times[0,+\infty)$ with initial temperature $f$ and initial mushy coefficient $h=\ind_\Omega$, then for all $t\ge0$ it holds
	\[
		|\Omega_{ u}^+(t)\setminus \Omega_{f}^+| \le \|\overline f \|_{L^1(D)}.
	\]
	In particular, even if $\Omega_{f}$ is dense, $\Omega_{u}(t)$ does not invade the whole space.
	
	Taking $\tilde f \in H^1(\R^d)$ defined from $f$ as in \Cref{prop:odd-reflections}, the same strategy provides a two-phase solution with dense initial set that do not invade the whole space as well.
\end{example}

The previous example, despite being interesting thanks to its singular nature, does not present definite branches as the examples we did in the compact case.

\begin{example}\label{ex:non-compact-one}
Let $d\ge2$, $D := \{x=(x_1,x_2,x') \in \R\times\R\times\R^{d-2} : x_1>0, \,x_2>0\}$, 
the function $\phi : [0,+\infty \to \R$ be defined as
\[
    \phi(x_1) =
        1 + \sqrt{2x_1(1-x_1)} \quad\text{if}\quad 0 < x_1 < 1 \qquad\text{and}\qquad 
        0 \quad \text{if}\quad x_1\ge1,
\]
and $\Omega\subset D$ be an open set defined as
\[
    \Omega := \Big\{ x=(x_1,x_2,x') \in D :  x_2 > \phi(x_1)\Big\}.
\]
and let us define $g \in H^1_0(D)$ as
\[
    g(x) := \big(\delta \wedge \distop(x,\R^d\setminus\Omega) \big) e^{-|x|^2},
\]
where $\delta \ll 1$ is a constant to be chosen later. Let $(u,\mu)$ and $(\overline u,\overline\mu)$ be the enthalpy solutions with initial temperature $g$ and mushy coefficient $h = \ind_{\Omega}$ in $D_\infty$ and $\R^{d+1}_+$ respectively.
The Comparison principle and \Cref{thm:finite-melting} imply that, for $\delta >0$ sufficiently small, then
\[
    \Omega_u^+(t) \subset \Omega_{\overline u}^+(t) \subset \R^d\setminus B_{\sfrac14} \qquad\text{for all}\quad t\ge0.
\]
Applying \Cref{prop:odd-reflections} two times (first with respect to $e_1$ and then with respect to $e_d$) we get a two-phase solution $(\tilde u,\tilde \mu)$ in $\R^{d+1}_+$ defined as follows
\[
    \tilde u(x,t) =
        \begin{cases}
        \begin{array}{rl}
            u(|x_1|,|x_2|,x') &\text{if} \quad x_1 x_2 >0\\
            -u(|x_1|,|x_2|,x') &\text{if} \quad x_1 x_2 <0,
        \end{array}
        \end{cases}
        \hspace{-0.15cm}
        \text{and}\quad
    \tilde \mu(x,t) =
        \begin{cases}
        \begin{array}{rl}
            \mu(|x_1|,|x_2|,x') &\text{if} \quad x_1 x_2 >0\\
            -\mu(|x_1|,|x_2|,x') &\text{if} \quad x_1 x_2 <0.
        \end{array}
        \end{cases}\hspace{-0.2cm}
\]
As for \Cref{ex:bounded-singular}, $\Omega_{\tilde u}^+(t)$ and $\Omega_{\tilde u}^-(t)$ are non decreasing in time and so (by comparison) there exists $1/4 < r(t) < 1$, $r'(t)\ge0$ such that
\[
    \Gamma_{tp}(t) = \Big\{(x_1,x_2,x') \in \R\times\R\times\R^{d-2} : x_1 = r(t) \; \text{and}\; x_2 =0 \quad \text{or} \quad x_1 = 0\; \text{and}\; x_2 = r(t) \Big\}.
\]
\end{example}

\begin{example}\label{ex:multiple-sections}
    With the same construction of \Cref{ex:non-compact-one}, we can produce other examples. For instance, let $D \subset \R^d$ be defined as $D:= D_k \times \R^{d-2}$ where for all $k\in \NN$ and $k\ge 2$, $D_k \subset\R^2 = \mathbb C$ is defined as follows
    \[
        D_k := \left\{ z = r e^{i\theta} : 0 < \theta < \frac{\pi}{k} \right\}.
    \]
    Then, we define $\Omega_2$ as $\Omega$ of \Cref{ex:non-compact-one}, and for all $k\ge3$ we define
    \[
        \Omega_k := \big\{ (z,x') \in D_k \times \R^{d-2} : z^{\sfrac k2} \in \Omega_2 \big\} \qquad\text{and}\qquad g_k(x) := \big( \delta \wedge \distop(x,\R^d\setminus \Omega_k) \big) e^{-|x|^2},
    \]
    and we denote by $(u_k,\mu_k)$ the one-phase weak solution  of the Stefan problem in $D$ with initial datum $g_k$; we notice that in this construction the set $D\setminus \Omega_k$ is bounded.
   By making $2k$ odd reflections of $u_k$, $\mu_k$, and $g_k$ across the rays $\theta=j\pi/k$, $j=1,\dots, 2k$, (see \Cref{fig:four-phases} for the cases $k=2$ and $k=3$) we obtain 
   a two-phase solution $(\tilde u,\tilde \mu)$ in $\R^{d+1}_+$. We notice that, by construction, the region $\text{Int}(\{\widetilde u=0\})$ is contained in a ball centered in the origin. By choosing the parameter $\delta$ in the definition of $g_k$ small enough, we get that $\emptyset\ne\Gamma_{tp}(t) \subset \de D$ for all $t\ge0$ and 
    \[
        B_{\sfrac14} \subset \R^d \setminus \big(\Omega_{\tilde u}^+(t) \cup \Omega_{\tilde u}^-(t) \big) \qquad\text{for all}\quad t\ge0.
    \]
\end{example}

\begin{remark}\label{oss:GZ-application-2k-symmetry}
    In \Cref{ex:multiple-sections}, for $k\ge3$ the persistence of the mushy region at the origin can be deduced also from \cite{GotzZaltman95:MushyavoidAngles}. Indeed, from \cite{GotzZaltman95:MushyavoidAngles} it follows that exists $r(t)>0$ such that 
    \[
        B_{r(t)} \cap D_k \subset D_k \setminus \Omega_{u_k}^+(t) \qquad\text{for all}\quad t\ge0,
    \]
    and so $0\in\text{Int}(\{\widetilde u(\cdot,t)=0\})$ for all times $t\ge0$. 
    However, \cite{GotzZaltman95:MushyavoidAngles} applies only when $g \in L^\infty(\R^d)$ and does not guarantee that $r(t)>c$ for all $t\ge0$. We also notice that this argument can only be applied to $2k$-symmetric solutions as the one constructed in \Cref{ex:multiple-sections}.
\end{remark}

\begin{center}
\begin{figure}[ht]
\begin{tikzpicture}[use Hobby shortcut,thick,scale=0.9]


            

            

            

            


\begin{scope}[xshift=100,scale=0.7]
    \clip (-4,-4) rectangle (4,4);

    \draw[dashed,gray] (0,0) -- (0:5);
    \draw[dashed,gray] (0,0) -- (60:5);
    \draw[dashed,gray] (0,0) -- (120:5);
    \draw[dashed,gray] (0,0) -- (180:5);
    \draw[dashed,gray] (0,0) -- (240:5);
    \draw[dashed,gray] (0,0) -- (300:5);
    
    \begin{scope}[scale=0.8]
        \clip (-4.5,-4.5) rectangle (4.5,4.5);

        \draw[fill = MidnightBlue!10,scale=0.4]
            (60:20) -- (60:10) .. (60:5) .. (55:3) .. (50:2.5)  .. (30:2) .. (10:2.5) ..  (5:3) .. (0:5)  .. (0:10) -- (20,0) -- (50,50);
            
        \draw[fill = MidnightBlue!20]
            (60:20) -- (60:10) .. (60:5) .. (55:3) .. (50:2.5)  .. (30:2) .. (10:2.5) ..  (5:3) .. (0:5)  .. (0:10) -- (20,0) -- (50,50);


        \begin{scope}[rotate=120] 
        \draw[fill = MidnightBlue!10,scale=0.4]
            (60:20) -- (60:10) .. (60:5) .. (55:3) .. (50:2.5)  .. (30:2) .. (10:2.5) ..  (5:3) .. (0:5)  .. (0:10) -- (20,0) -- (50,50);
            
        \draw[fill = MidnightBlue!20]
            (60:20) -- (60:10) .. (60:5) .. (55:3) .. (50:2.5)  .. (30:2) .. (10:2.5) ..  (5:3) .. (0:5)  .. (0:10) -- (20,0) -- (50,50);

        \end{scope}

                \begin{scope}[rotate=-120] 
        \draw[fill = MidnightBlue!10,scale=0.4]
            (60:20) -- (60:10) .. (60:5) .. (55:3) .. (50:2.5)  .. (30:2) .. (10:2.5) ..  (5:3) .. (0:5)  .. (0:10) -- (20,0) -- (50,50);
            
        \draw[fill = MidnightBlue!20]
            (60:20) -- (60:10) .. (60:5) .. (55:3) .. (50:2.5)  .. (30:2) .. (10:2.5) ..  (5:3) .. (0:5)  .. (0:10) -- (20,0) -- (50,50);

        \end{scope}

        \begin{scope}[rotate=60] 
        \draw[fill = black!7,scale=0.4]
            (60:20) -- (60:10) .. (60:5) .. (55:3) .. (50:2.5)  .. (30:2) .. (10:2.5) ..  (5:3) .. (0:5)  .. (0:10) -- (20,0) -- (50,50);
            
        \draw[fill = black!14]
            (60:20) -- (60:10) .. (60:5) .. (55:3) .. (50:2.5)  .. (30:2) .. (10:2.5) ..  (5:3) .. (0:5)  .. (0:10) -- (20,0) -- (50,50);

        \end{scope}

        \begin{scope}[rotate=180] 
        \draw[fill = black!7,scale=0.4]
            (60:20) -- (60:10) .. (60:5) .. (55:3) .. (50:2.5)  .. (30:2) .. (10:2.5) ..  (5:3) .. (0:5)  .. (0:10) -- (20,0) -- (50,50);
            
        \draw[fill = black!14]
            (60:20) -- (60:10) .. (60:5) .. (55:3) .. (50:2.5)  .. (30:2) .. (10:2.5) ..  (5:3) .. (0:5)  .. (0:10) -- (20,0) -- (50,50);

        \end{scope}

        \begin{scope}[rotate=-60] 
        \draw[fill = black!7,scale=0.4]
            (60:20) -- (60:10) .. (60:5) .. (55:3) .. (50:2.5)  .. (30:2) .. (10:2.5) ..  (5:3) .. (0:5)  .. (0:10) -- (20,0) -- (50,50);
            
        \draw[fill = black!14]
            (60:20) -- (60:10) .. (60:5) .. (55:3) .. (50:2.5)  .. (30:2) .. (10:2.5) ..  (5:3) .. (0:5)  .. (0:10) -- (20,0) -- (50,50);
            
        \end{scope}
    \end{scope}

\end{scope}

\begin{scope}[xshift=-100,scale=0.7]
    \begin{scope}[scale=0.8]
        \clip (-4.5,-4.5) rectangle (4.5,4.5);

        \draw[fill = MidnightBlue!10,scale=0.4]
            (0,20) -- (0,5) .. (0,4) .. (0.5,1.5) .. (1.5,0.5) ..(4,0) .. (5,0) -- (20,0) --(20,20);
            
        \draw[fill = MidnightBlue!20]
            (0,20) -- (0,5) .. (0,4) .. (0.5,1.5) .. (1.5,0.5) ..(4,0) .. (5,0) -- (20,0) --(20,20);

        \begin{scope}[scale=-1] flippato
        \draw[fill = MidnightBlue!10,scale=0.4]
            (0,20) -- (0,5) .. (0,4) .. (0.5,1.5) .. (1.5,0.5) ..(4,0) .. (5,0) -- (20,0) --(20,20);
            
        \draw[fill = MidnightBlue!20]
            (0,20) -- (0,5) .. (0,4) .. (0.5,1.5) .. (1.5,0.5) ..(4,0) .. (5,0) -- (20,0) --(20,20);
        \end{scope}

        \begin{scope}[xscale=-1] 
        \draw[fill = black!7,scale=0.4]
            (0,20) -- (0,5) .. (0,4) .. (0.5,1.5) .. (1.5,0.5) ..(4,0) .. (5,0) -- (20,0) --(20,20);
            
        \draw[fill = black!14]
            (0,20) -- (0,5) .. (0,4) .. (0.5,1.5) .. (1.5,0.5) ..(4,0) .. (5,0) -- (20,0) --(20,20);
        \end{scope}

        \begin{scope}[yscale=-1] 
        \draw[fill = black!7,scale=0.4]
            (0,20) -- (0,5) .. (0,4) .. (0.5,1.5) .. (1.5,0.5) ..(4,0) .. (5,0) -- (20,0) --(20,20);
            
        \draw[fill = black!14]
            (0,20) -- (0,5) .. (0,4) .. (0.5,1.5) .. (1.5,0.5) ..(4,0) .. (5,0) -- (20,0) --(20,20);
        \end{scope}
    \end{scope}

    \draw[dashed,gray] (-4,0) -- (4,0);
    \draw[dashed,gray] (0,-4) -- (0,4);
\end{scope}
    
\end{tikzpicture}
    
    \caption{On the left, the example of \Cref{ex:non-compact-one}, on the right the one of \Cref{ex:multiple-sections} with $k=3$. The darker colors represent the original phases, the lighter ones their evolution through time.}
    \label{fig:four-phases}
\end{figure}
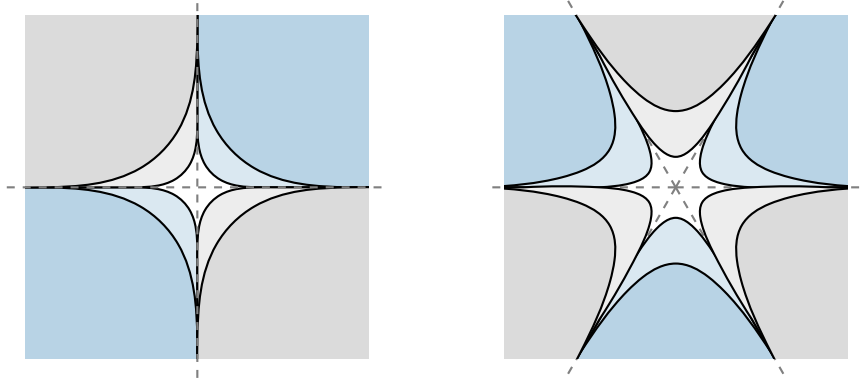
\end{center}

\vspace{-0.5cm}

\begin{example}\label{ex:non-compact-2}
    In the case of \Cref{ex:non-compact-one,ex:multiple-sections}, the gaussian weight make possible to argue avoiding \Cref{thm:finite-melting} with special radial one-phase barriers, in a similar way that works for \Cref{ex:bounded-singular}. Let $D$ and $\Omega_k$ be as in \Cref{ex:non-compact-one,ex:multiple-sections} and let 
    \[
        \hat \Omega_k := \Omega_k \setminus \big(e_k + \Omega_k  \big),
    \]
    where $e_k:= (\cos(\pi/2k),\sin(\pi/2k),0)$. Consider as initial temperature the function 
    \[
        \hat g_k(x) := (\delta \wedge \distop(x,\R^d\setminus \hat\Omega_k) |x-e_k|^{-d+\sfrac12},
    \]
    and let $(u_k,\mu_k)$ be the one-phase solution in $D$ with initial datum $\hat g_k$. Proceeding as in \Cref{ex:non-compact-one,ex:multiple-sections}, starting from $(u_k,\mu_k)$, we construct a two-phase $(2k)$-symmetric solution $(\tilde u,\tilde\mu)$ in $\R^{d+1}_+$ with persistent mushy region that contains a ball $B_\rho$ for all times $t\ge 0$. 
    We notice that in this case the argument from \Cref{oss:GZ-application-2k-symmetry} cannot be applied, since $|x-e_k|^{-d+\sfrac12}\not\in L^1(\R^d)$ and $|x|^{-d+\sfrac12}\not\in L^1(\R^{d-1})$ and so the temperature does not decays sufficiently fast to lie below a radial competitor.
\end{example}

\vspace{-0.3cm}

\begin{center}
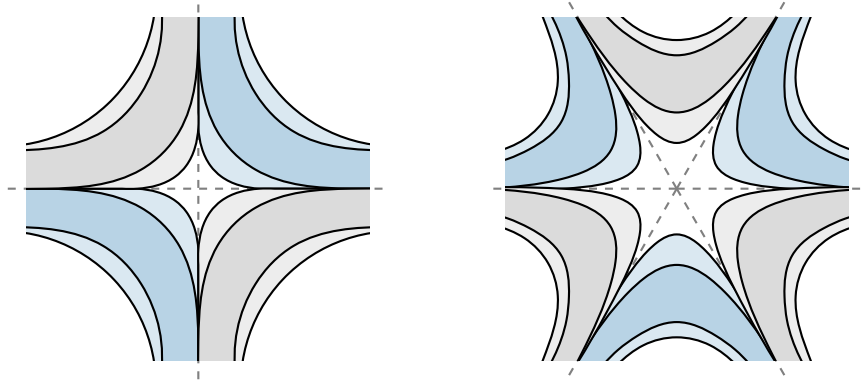
\begin{figure}[ht]
\begin{tikzpicture}[use Hobby shortcut,thick,scale=0.9]

\begin{scope}[xshift=-100,scale=0.7]

    \draw[dashed,gray] (-4,0) -- (4,0);
    \draw[dashed,gray] (0,-4) -- (0,4);

    \begin{scope}[scale=0.8]
        \clip (-4.5,-4.5) rectangle (4.5,4.5);

        \draw[fill = MidnightBlue!10,scale=0.4]
            (0,20) -- (0,5) .. (0,4) .. (0.8,1.5) .. (1.5,0.8) ..(4,0) .. (5,0) -- (20,0)
            -- (20,3) -- (15,3) .. (6,5.5) 
            .. (5.5,6) .. (3,15) -- (3,20);
            
        \draw[fill = MidnightBlue!20]
            (0,20) -- (0,5) .. (0,4) .. (0.5,1.5) .. (1.5,0.5) ..(4,0) .. (5,0) -- (20,0)
            -- (20,1) -- (6,1) .. (5,1) .. (2.5,1.5) .. (1.5,2.5) .. (1,5) -- (1,20);

        \begin{scope}[scale=-1] 
        \draw[fill = MidnightBlue!10,scale=0.4]
            (0,20) -- (0,5) .. (0,4) .. (0.8,1.5) .. (1.5,0.8) ..(4,0) .. (5,0) -- (20,0)
            -- (20,3) -- (15,3) .. (6,5.5) 
            .. (5.5,6) .. (3,15) -- (3,20);
            
        \draw[fill = MidnightBlue!20,scale=1]
            (0,20) -- (0,5) .. (0,4) .. (0.5,1.5) .. (1.5,0.5) ..(4,0) .. (5,0) -- (20,0)
            -- (20,1) -- (6,1) .. (5,1) .. (2.5,1.5) .. (1.5,2.5) .. (1,5) -- (1,20);
        \end{scope}

        \begin{scope}[yscale=-1] 
        
        \draw[fill = black!7,scale=0.4]
            (0,20) -- (0,5) .. (0,4) .. (0.8,1.5) .. (1.5,0.8) ..(4,0) .. (5,0) -- (20,0)
            -- (20,3) -- (15,3) .. (6,5.5) 
            .. (5.5,6) .. (3,15) -- (3,20);
            
        \draw[fill = black!14,scale=1]
            (0,20) -- (0,5) .. (0,4) .. (0.5,1.5) .. (1.5,0.5) ..(4,0) .. (5,0) -- (20,0)
            -- (20,1) -- (6,1) .. (5,1) .. (2.5,1.5) .. (1.5,2.5) .. (1,5) -- (1,20);
        \end{scope}        

        \begin{scope}[xscale=-1] 
        
        \draw[fill = black!7,scale=0.4]
            (0,20) -- (0,5) .. (0,4) .. (0.8,1.5) .. (1.5,0.8) ..(4,0) .. (5,0) -- (20,0)
            -- (20,3) -- (15,3) .. (6,5.5) 
            .. (5.5,6) .. (3,15) -- (3,20);
            
        \draw[fill = black!14,scale=1]
            (0,20) -- (0,5) .. (0,4) .. (0.5,1.5) .. (1.5,0.5) ..(4,0) .. (5,0) -- (20,0)
            -- (20,1) -- (6,1) .. (5,1) .. (2.5,1.5) .. (1.5,2.5) .. (1,5) -- (1,20);
        \end{scope}             
        
    \end{scope}

\end{scope}

\begin{scope}[xshift=100,scale=0.7]
    \clip (-4,-4) rectangle (4,4);
    \draw[dashed,gray] (0,0) -- (0:5);
    \draw[dashed,gray] (0,0) -- (60:5);
    \draw[dashed,gray] (0,0) -- (120:5);
    \draw[dashed,gray] (0,0) -- (180:5);
    \draw[dashed,gray] (0,0) -- (240:5);
    \draw[dashed,gray] (0,0) -- (300:5);

            \begin{scope}[scale=0.8]
        \clip (-4.5,-4.5) rectangle (4.5,4.5);

        \draw[fill = MidnightBlue!10,scale=0.6]
            (60:40) -- (60:10) .. (60:5) .. (55:3) .. (50:2.5)  .. (30:2) .. (10:2.5) ..  (5:3) .. (0:5)  .. (0:10) -- (20,0) 
            -- (5:20) -- (2:14) .. (8:9.5)  .. (18:7) ..(30:6.5) ..(42:7) .. (52:9.5) .. (58:14) -- (55:40);
            
        \draw[fill = MidnightBlue!20]
            (60:20) -- (60:10) .. (60:5) .. (55:3) .. (50:2.5)  .. (30:2) .. (10:2.5) ..  (5:3) .. (0:5)  .. (0:10) -- (20,0) 
            -- (5:20) -- (5:11) .. (5:6.5) .. (10:4.5) .. (15:4) ..(30:3.5) ..(45:4) .. (50:4.5) .. (55:6.5) .. (55:11) -- (55:20);

        \begin{scope}[rotate=120] 
        \draw[fill = MidnightBlue!10,scale=0.6]
            (60:40) -- (60:10) .. (60:5) .. (55:3) .. (50:2.5)  .. (30:2) .. (10:2.5) ..  (5:3) .. (0:5)  .. (0:10) -- (20,0) 
            -- (5:20) -- (2:14) .. (8:9.5)  .. (18:7) ..(30:6.5) ..(42:7) .. (52:9.5) .. (58:14) -- (55:40);
            
        \draw[fill = MidnightBlue!20]
            (60:20) -- (60:10) .. (60:5) .. (55:3) .. (50:2.5)  .. (30:2) .. (10:2.5) ..  (5:3) .. (0:5)  .. (0:10) -- (20,0) 
            -- (5:20) -- (5:11) .. (5:6.5) .. (10:4.5) .. (15:4) ..(30:3.5) ..(45:4) .. (50:4.5) .. (55:6.5) .. (55:11) -- (55:20);

        \end{scope}

                \begin{scope}[rotate=-120] 
        \draw[fill = MidnightBlue!10,scale=0.6]
            (60:40) -- (60:10) .. (60:5) .. (55:3) .. (50:2.5)  .. (30:2) .. (10:2.5) ..  (5:3) .. (0:5)  .. (0:10) -- (20,0) 
            -- (5:20) -- (2:14) .. (8:9.5)  .. (18:7) ..(30:6.5) ..(42:7) .. (52:9.5) .. (58:14) -- (55:40);
            
        \draw[fill = MidnightBlue!20]
            (60:20) -- (60:10) .. (60:5) .. (55:3) .. (50:2.5)  .. (30:2) .. (10:2.5) ..  (5:3) .. (0:5)  .. (0:10) -- (20,0) 
            -- (5:20) -- (5:11) .. (5:6.5) .. (10:4.5) .. (15:4) ..(30:3.5) ..(45:4) .. (50:4.5) .. (55:6.5) .. (55:11) -- (55:20);

        \end{scope}

        \begin{scope}[rotate=60] 
        \draw[fill = black!7,scale=0.6]
            (60:40) -- (60:10) .. (60:5) .. (55:3) .. (50:2.5)  .. (30:2) .. (10:2.5) ..  (5:3) .. (0:5)  .. (0:10) -- (20,0) 
            -- (5:20) -- (2:14) .. (8:9.5)  .. (18:7) ..(30:6.5) ..(42:7) .. (52:9.5) .. (58:14) -- (55:40);
            
        \draw[fill = black!14]
            (60:20) -- (60:10) .. (60:5) .. (55:3) .. (50:2.5)  .. (30:2) .. (10:2.5) ..  (5:3) .. (0:5)  .. (0:10) -- (20,0) 
            -- (5:20) -- (5:11) .. (5:6.5) .. (10:4.5) .. (15:4) ..(30:3.5) ..(45:4) .. (50:4.5) .. (55:6.5) .. (55:11) -- (55:20);

        \end{scope}

        \begin{scope}[rotate=180] 
        \draw[fill = black!7,scale=0.6]
            (60:40) -- (60:10) .. (60:5) .. (55:3) .. (50:2.5)  .. (30:2) .. (10:2.5) ..  (5:3) .. (0:5)  .. (0:10) -- (20,0) 
            -- (5:20) -- (2:14) .. (8:9.5)  .. (18:7) ..(30:6.5) ..(42:7) .. (52:9.5) .. (58:14) -- (55:40);
            
        \draw[fill = black!14]
            (60:20) -- (60:10) .. (60:5) .. (55:3) .. (50:2.5)  .. (30:2) .. (10:2.5) ..  (5:3) .. (0:5)  .. (0:10) -- (20,0) 
            -- (5:20) -- (5:11) .. (5:6.5) .. (10:4.5) .. (15:4) ..(30:3.5) ..(45:4) .. (50:4.5) .. (55:6.5) .. (55:11) -- (55:20);

        \end{scope}

        \begin{scope}[rotate=-60] 
        \draw[fill = black!7,scale=0.6]
            (60:40) -- (60:10) .. (60:5) .. (55:3) .. (50:2.5)  .. (30:2) .. (10:2.5) ..  (5:3) .. (0:5)  .. (0:10) -- (20,0) 
            -- (5:20) -- (2:14) .. (8:9.5)  .. (18:7) ..(30:6.5) ..(42:7) .. (52:9.5) .. (58:14) -- (55:40);
            
        \draw[fill = black!14]
            (60:20) -- (60:10) .. (60:5) .. (55:3) .. (50:2.5)  .. (30:2) .. (10:2.5) ..  (5:3) .. (0:5)  .. (0:10) -- (20,0) 
            -- (5:20) -- (5:11) .. (5:6.5) .. (10:4.5) .. (15:4) ..(30:3.5) ..(45:4) .. (50:4.5) .. (55:6.5) .. (55:11) -- (55:20);
            
        \end{scope}
    \end{scope}       
        
    \end{scope}

\end{tikzpicture}
    
    \caption{In the figure the solutions granted by \Cref{ex:non-compact-2} in the cases \Cref{ex:non-compact-one} (on the left) and \Cref{ex:multiple-sections} (on the right). }
    \label{fig:four-phases-2}
\end{figure}
\end{center}

\vspace{-0.5cm}

\begin{example}
    The same approach of \Cref{ex:non-compact-one}, produces examples with countable many branching regions and non-symmetric ones. 
    Indeed, for $k=1,2$ we consider two families of nonnegative bump function $\{\rho_i^{k}\}_{i\in\NN}$ with the following properties
    \begin{itemize}
        \item for all $k=1,2$ they hold 
        \[
            \sup_{j}\|\rho_{j}^k\|_{L^\infty} \le 1
            \qquad\text{and}\qquad
            \inf_{j} \|\rho_{j}^k\|_{L^1} \le 1/4;
        \]
        \item there exists two families of points $\{x_j^k\}_{j\in\NN} \subset \R^{d-1}$ such that for $k=1,2$ it holds
        \[
        \inf_{j\in\NN}\distop(x_{j}^k,x_{i}^k) \ge 2, \qquad \rho_j^k(x_j^k) \ge 1/2, \qquad\text{and}\qquad \spt \rho_j^k \subset B_{1}(x_j^k) \qquad\text{for all}\quad j\in\NN;
        \]
        \item it holds
        \[
            (x_{j}^k)_k \ge 2 \qquad\text{for}\quad k=1,2 \qquad\text{and}\qquad\text{for all}\quad j\in\NN.
        \]
    \end{itemize}
    Let then consider the set $\Omega\subset D$ defined as follows
    \[
        \Omega = 
        \left\{x \in D : 
            x_2 > \phi(x_1) + \sum_{i\in\NN} \rho_{i}^1(x_1,x') 
            \quad\text{and}\quad 
            x_1 > \phi(x_2) +  \sum_{k\in\NN} \rho_j^2(x_2,x')
        \right\}.
    \]
    Taking $g$ as in \Cref{ex:non-compact-one} and $h = \ind_\Omega$, the previous construction provides a family of solutions with countable many disjoint and persistent mushy regions. The same holds if we consider $\hat \Omega$ of \Cref{ex:multiple-sections}, \ref{ex:non-compact-2}, or if we allow more general and non odd-symmetric configurations as for \Cref{ex:non-symmetric-bounded-singular}.
\end{example}


\subsection*{Use of AI} No AI tools were used to produce any of the mathematical content of this paper, nor to generate parts of the text; the problem, the strategy, the proofs and the text were all conceived and written exclusively by the authors. 

\subsection*{{Acknowledgments}}
The authors acknowledge the MIUR Excellence Department Project awarded to the Department of Mathematics, University of Pisa, CUP I57G22000700001. 
F.P. is also partially supported by the INdAM - GNAMPA Project "Struttura fine e regolarita' in problemi variazionali non-lineari" (CUP E53C25002010001).
B.V. acknowledges the support of the European Research Council (ERC) via the project ERC
FiRM - Fine structure and regularity of stationary and moving free boundaries (grant agreement No. 101230705).

\bibliographystyle{amsalpha}
\bibliography{Stefan}

\end{document}